\documentclass[11pt]{article}
\usepackage[dvipsnames, svgnames, x11names]{xcolor}
\usepackage{changes}
\usepackage[UTF8]{ctex}
\usepackage[utf8]{inputenc}
\usepackage{graphicx} 

\usepackage{tikz}
\usepackage{tikz-imagelabels}
\usetikzlibrary{decorations.pathreplacing, arrows.meta}

\usepackage[dvipsnames]{xcolor}
\usepackage{tcolorbox}
\usepackage{geometry}
\usepackage{amsthm,amssymb,amsmath}
\usepackage{amsmath}
\usepackage{amsthm}
\usepackage{float}
\usepackage{color}
\usepackage{xcolor}
\usepackage[dvipsnames, svgnames, x11names]{xcolor}
\usepackage[mathscr]{euscript}
\usepackage{graphicx}
\usepackage{ulem}
\usepackage{setspace}
\usepackage{enumitem}
\usepackage{hyperref}  
\hypersetup{
    colorlinks=true,   
    linkcolor=blue,    
    citecolor=green,   
    urlcolor=magenta   
}

\usepackage{soul}
\setstcolor{red}
\soulregister\ref7
\soulregister\cite7
\soulregister\citeyear7

\usepackage{overpic}

\newtheorem{theorem}[equation]{Theorem}
\newtheorem{definition}[equation]{Definition}
\newtheorem{lemma}[equation]{Lemma}
\newtheorem{proposition}[equation]{Proposition}
\newtheorem{corollary}[equation]{Corollary}

\newtheorem{remark}[equation]{Remark}

\usepackage{bm}
\usepackage{forest}

\usepackage{enumitem}

\numberwithin{equation}{section}

\newcommand{\RR}{\mathbb{R}}

\newcommand{\LL}{L}

\newcommand{\dd}{\operatorname{dist}}
\newcommand{\ccup}{\cup}

\newcommand{\A}{\mathscr{A}}
\newcommand{\B}{\mathscr{B}}
\newcommand{\ppi}{\overline{\pi}}

\newcommand{\ma}{\marginpar}

\newcommand\blfootnote[1]{%
  \begingroup
  \renewcommand\thefootnote{}\footnote{#1}%
  \addtocounter{footnote}{-1}%
  \endgroup
}

\title{A generalization of Reifenberg's theorem in $\RR^N$ for flat cones}
\author{Xiangyu\ Liang,\and Sicheng\ Zhang}
\date{}
\begin{document}

\maketitle

\blfootnote{
    \textbf{Xiangyu Liang} \\
    School of Mathematical Sciences, Beihang University, Beijing, China. 
    {E-mail}: \href{mailto:maizeliang@gmail.com}{\texttt{maizeliang@gmail.com}}
}

\blfootnote{
    \textbf{Sicheng Zhang} \\
    School of Mathematical Sciences, Beihang University, Beijing, China. 
    {E-mail}: \href{mailto:sichengzhang@buaa.edu.cn}{\texttt{sichengzhang@buaa.edu.cn}}
}

\begin{abstract}
In this paper we prove that if a closed set in $\RR^N$ is close to a cone over a simplicial complex at each point and at each scale, then it is locally bi-H$\ddot{\text{o}}$lder equivalent to such a cone. This generalizes Reifenberg's Topological Disk Theorem in 1960 and  G. David, T. De Pauw and T. Toro's result in 2008.
\end{abstract}

\section{Introduction}
In 1960, Reifenberg  established the remarkable Topological Disk Theorem in \cite{R1}. It shows that if a closed set $E$ in $\RR^N$ is sufficiently close to an $n$-dimensional plane in the Hausdorff distance sense at each point and at each scale, then it is locally bi-H$\ddot{\text{o}}$lder equivalent to a ball of dimension $n$. 


Since then, generalizations of Reifenberg's theorem have broadly followed two main directions: quantitative characterizations and singular approximating models.
The quantitative direction focuses on quantitative versions of the theorem, notably developed by Naber and Valtorta in \cite{NV17}. In contrast to the classical topological results in \cite{R1} which rely on pointwise flatness to obtain bi-H$\ddot{\text{o}}$lder parameterizations, \cite{NV17} uses an integral condition on flatness defined by Jones' $\beta$-numbers. Under this weaker assumption, they established the rectifiability of the set and derived uniform measure estimates. These quantitative results were subsequently extended to Hilbert and Banach spaces in \cite{ENV19} and to general measures without density assumptions in \cite{ENV25}.

The other direction involves generalizations to singular sets, where the approximating models are no longer topological disks. David and Toro \cite{DT12} extended Reifenberg's construction to sets with holes, locally parameterizing sets that are close to planes with holes.
In a broader context, Badger and Lewis \cite{BL15} developed a comprehensive framework investigating the interplay between local approximation and asymptotic geometry (tangent sets), with applications to zero sets of harmonic polynomials discussed in \cite{BET17}. However, their focus was primarily on structure theorems and dimension estimates for the singular sets, rather than constructing explicit parameterizations.
In contrast, the work of David, De Pauw, and Toro \cite{DDT} takes a constructive approach. They established the existence of a bi-H$\ddot{\text{o}}$lder parameterization for sets approximated by minimal cones in $\RR^3$. Our work is a direct generalization of this result.


Specifically, in 2008, they generalized Reifenberg's Topological Disk Theorem to the case when the set is close to a minimal cone of dimension 2 in $\RR^3$, see \cite{DDT}.  It shows that if a closed set $E$ in $\RR^3$ is sufficiently close to a minimal cone of dimension 2 in the Hausdorff distance sense at each point and at each scale, then it is locally bi-H$\ddot{\text{o}}$lder equivalent to a minimal cone of dimension 2. 
By using the hierarchical structure of $E$, they split 
$E$ into disjoint subsets and constructed the parameterization of each subset respectively. Here the minimal cones of dimension 2 in $\RR^3$ are, modulo isometry, a plane, a $\mathbb{Y}$ set (the product of a $Y$ and a line, here $Y$ is the union of three half lines which meet at a common point and make angles of 120$^{\circ}$ in a plane), and 
a $\mathbb{T}$ set (the cone over the six segments whose endpoints are exactly two of the vertices of a regular tetrahedron centered at the origin), classified by J. E. Taylor (see \cite{T}). See Figure \ref{minimal cone} below.

\begin{figure}[h]
    \centering
    \includegraphics[width=0.5\linewidth]{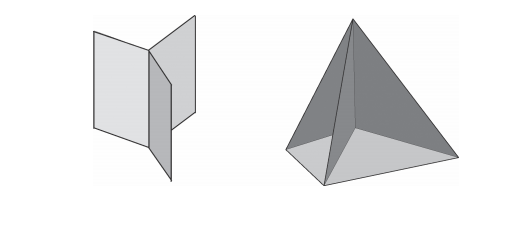}
    \caption{The $\mathbb{Y}$ set and $\mathbb{T}$ set}
    \label{minimal cone}
\end{figure}

G. David, T. De Pauw and T. Toro also gave a flexible version (without proof) of this theorem, where the three types of minimal cones in $\RR^3$ are replaced by sets of type G1, G2 and G3 in $\RR^n$.  Modulo isometry, sets of type G1 are $d$-planes in $\RR^n$, where $d\geq 2$. Sets of type G2 are $Y\times \RR^{d-1}$. Finally, a 2-dimensional set of type G3 is a cone over a subset $\Gamma\subset \partial B(0,1)$, where $\Gamma$ is composed of not too small arcs of circles that only meet at their endpoints, with angles greater than $\pi/2$, and each endpoint is a common endpoint of exactly three arcs. A $d$-dimensional set of type G3 is the product of  $\RR^{d-2}$ and a 2-dimensional set of type G3.


{
In this paper, we generalize the theorem of G. David, T. De Pauw, and T. Toro \cite{DDT} to a general geometric framework. While \cite{DDT} focuses on 2-dimensional minimal cones in $\mathbb{R}^3$ and cones of type G1, G2, G3 in $\mathbb{R}^n$, we here consider a more general class of cones in $\mathbb{R}^N$, which we classify into sets of type $m$ for each integer $0 \leq m \leq n$.

Specifically, a set of type $m$ is isometrically equivalent to the product $C \times \mathbb{R}^m$, where $C$ is an $(n-m)$-dimensional cone over a simplicial complex (see Definition \ref{复合锥}); the detailed definition of sets of type $m$ is given in Definition \ref{type m}. Modulo isometries, we define the spines of such a set as follows: the $m$-dimensional spine is $\{0\} \times \mathbb{R}^m$. For any integer $t$ with $m < t \leq n$, the $t$-dimensional spine is defined as the product of the sub-cone generated by the $(t-m-1)$-dimensional faces of the underlying complex and $\mathbb{R}^m$. Under this classification, the $\mathbb{Y}$-set corresponds to type 1 (product of a 1-dimensional cone and $\mathbb{R}^1$), and the $\mathbb{T}$-set corresponds to type 0 (product of a 2-dimensional cone and $\mathbb{R}^0$).

Let us introduce our main theorem. Before doing so, we explicitly fix some notations. Precise definitions and properties will be provided in Section 2.
We denote by $\mathscr{TA}$ the collection of all sets of type $m$ for every $0 \le m \le n$. Let $\B$ be a finite subset of $\mathscr{TA}$ modulo isometries (that is, we consider $E$ and $F$ equivalent if $E=R(F)$ for some isometry $R$). Furthermore, we denote by $\A(\B)$ the finite set of all blow-up limits derived from $\B$, which are also contained in $\mathscr{TA}$. Throughout the paper, $n_0, \delta_0, \alpha$ represent constants depending only on $\B$ and $n$.

}


\begin{theorem}\label{主定理1}
    For each $\B\subset\mathscr{TA}$ such that $ \B/\sim$ is finite, there exist $C=C(\B)$, $\varepsilon_0=\varepsilon(\B)$ that depend only on $\B$ s.t. the following holds. Let $E\subset \RR^N$ be a closed set that contains the origin and $\varepsilon<\varepsilon_0$. If for each $x\in E$ and radius $r>0$, there is a set $Z(x,r)\in \A(\B)$ that contains $x$, such that
    \begin{equation}
        d_{x,r}(E,Z(x,r))<\varepsilon,
    \end{equation}
    then there is a set $Z\in\A(\B)$ through the origin and an injective mapping $f: B(0,1.95)\to B(0,2)$, with the following properties:
    \begin{equation}\label{f的包含与被包含}
        B(0,1.9)\subset f(B(0,1.95))\subset B(0,2),
    \end{equation}
    \begin{equation}\label{f作用于Z(0,2)}
        E\cap B(0,1.9)\subset f(Z\cap B(0,1.95))\subset E\cap B(0,2),
    \end{equation}
    \begin{equation}\label{f Holder}
        (1+C\varepsilon)^{-1} |x-y|^{1+C\varepsilon}< |f(x)-f(y)|<(1+C\varepsilon)|x-y|^{1/(1+C\varepsilon)}\text{ for }x,y\in B(0,1.95), 
    \end{equation}
    \begin{equation}\label{f移动一点距离}
        |f(x)-x|<C\varepsilon\text{ for }x\in B(0,1.95).
    \end{equation}
    Here 
    \begin{equation}
    \begin{aligned}
        d_{x,r}(F_1,F_2)=\frac{1}{r}\max \{&\sup\{\dd(z,F_1):z\in F_2\cap B(x,r)\},\\&\sup\{\dd(z,F_2):z\in F_1\cap B(x,r)\}
        \} 
        \end{aligned}
    \end{equation}
    whenever $F_1,F_2$ are closed sets in $\RR^N$ that meet $B(x,r)$.
\end{theorem}


By the bi-H$\ddot{\text{o}}$lder condition (\ref{f Holder}), we see that $f$ is a homeomorphism between $B(0,1.95)$ and its image. And (\ref{f作用于Z(0,2)}) implies that $f$ gives a parameterization of the set $E$ from a set of type $m$ for some integer $m$.


{

For the proof, we follow the general strategy of \cite{DDT}, but the construction of the parameterization is quite different. In \cite{DDT}, the authors relied on the classification of minimal cones in $\mathbb{R}^3$. Since there are only 3 geometric types (planes, $\mathbb{Y}$-sets, and $\mathbb{T}$-sets), they could use a case-by-case analysis. In our setting, however, we deal with cones over simplicial complexes in $\mathbb{R}^N$. The structures are too numerous to list, making a case-by-case analysis impossible. Instead, we develop a systematic induction on the dimension of the spines. This allows us to handle all geometric types in a unified way, without checking specific shapes one by one. This generalized result provides a tool for studying the regularity of sets whose blow-up limits are cones over simplicial complexes.

}


{  

The strategy of the proof relies on the hierarchical structure of the set $E$. Specifically, if the cone $Z$ is of type $m$, the set $E$ admits a stratification into disjoint subsets $E_i$ for $i=m, \dots, n$. For each $i$, $E_i$ consists of points on $E$ at which the blow-up limit is of type $i$. Geometrically, $E_i$ behaves like the $i$-dimensional spine of $Z$ at appropriate scales. These strata satisfy a closure condition analogous to the skeletal structure of a polyhedron: $E_{i-1}$ is contained in the closure of $E_i$. Crucially, away from the lower-dimensional strata $E_{i-1}$, the set $E_i$ locally satisfies the assumptions of Reifenberg’s Topological Disk Theorem of dimension $i$. See Section 3 for details.


We construct the parameterization $f: Z \cap B(0, 1.95) \to E$ by induction on the dimension of the spines of $Z$. As an illustrative example, assume $Z$ is of type 0, so that it possesses $m$-dimensional spines denoted by $L^m$ for all $0 \le m \le n$. We first define the parameterization $f^0$ from the 0-dimensional spine $L^0$ to $E_0$. We then define $f^1$ on the 1-dimensional spine $L^1$ as an extension of $f^0$ satisfying $f^1|_{L^0} = f^0$. We proceed with this inductive construction until we obtain $f^n$ on the entire cone $L^n = Z$. Finally, $f$ is extended to the ambient neighborhood $B(0, 1.95)$.

For each dimension $m$, the map $f^m$ is obtained as the limit of a sequence of homeomorphisms $\{f^m_k\}_{k=0}^\infty$. Setting $f^m_0 = id$, we define $f^m_{k+1} = g^m_k \circ f^m_k$, where each $g^m_k$ moves points by a distance comparable to $2^{-k}$ with constants independent of $k$. Geometrically, each $f^m_k$ maps the spine $L^m$ into a $C 2^{-k}$-neighborhood of the set $E_m$, so that the limit map $f^m$ sends $L^m$ precisely onto $E_m$.

}

{The plan for the rest of the article is the following. In Section 2, we introduce some basic notations and the definitions of sets of different types. Then we discuss the geometric facts of the sets we defined, which imply that sets of different types have different topological structures, so $d_{x,r}$ between different sets controls the distance of their spines.  In Section 3, we define the different subsets $E_m$ of $E$ and show the properties of $E_m$ inductively. In general, they satisfy the condition of Reifenberg's Topological Disk Theorem of dimension $m$ locally. We also discuss the relationship between $E$ and $E_m$. In Section 4, we define a partition of unity and then 
show the similarity of cones corresponding to close balls as a preparation for the construction of $f$. In Section 5, we construct the map $f$ inductively and prove that it satisfies all properties in the main theorem.
}


\section{Notations and geometric facts}
\subsection{Basic notations}

Unless otherwise specified, a ball $B$ in $\RR^N$ denotes an open ball and $\overline{B}$ denotes the closure of $B$.

Let $x\in \RR^N$ and $r>0$. We denote by $B(x,r)$ the open ball centered at $x$ with radius $r$. For $k>0$, $k B(x,r)$ denotes $B(x,kr)$.

Let $x,y,z$ be three different points in $\RR^N$, $\angle xyz$ denotes the angle between the vectors $x-y$ and $z-y$.

For a vector $v$ in $\RR^N$, $|v|$ denotes the Euclidean norm of $v$.

Let $F$ be a set in $\RR^N$, the affine hull of $F$ is $\text{aff}(F)=\{\sum_{i=1}^k \lambda_i x_i :x_i\in F,\lambda_i\in \mathbb{R},\sum_{i=1}^k \lambda_i=1,k\in \mathbb{N}_{+}\}.$

Let $F$ be a convex set in $\RR^N$, the relative interior of $F$ is $F^{\circ}=\{x\in F:$ there is $r>0$    
 such that aff$(F)\cap B(x,r)\subset F\}.$ 
And the relative boundary of $F$ is $\overline{F}\backslash F^{\circ}$, where $\overline{F}$ is the closure of $F$.

We say a plane of dimension $m$ is an $m$-plane.

Let $F$ be a set. We denote by $\#  F$  the number of elements in $F$.

We say a set $W$ coincides with another set $Y$ in a ball $B$ if and only if $W\cap B=Y\cap B$.

\subsection{Definitions and properties of cones}

{  
Throughout this article, we fix the integers $N > 0$ and $n > 0$. We assume that the ambient dimension $N$ is sufficiently large relative to $n$ ($N \gg n$) to ensure that the ambient space $\mathbb{R}^N$ admits an isometric embedding of the complex cones defined below. Specifically, since a complex cone is formed by the union of multiple simple cones (as in Definition \ref{简单锥} and Definition \ref{复合锥}), a sufficiently large ambient dimension is required to ensure that intersections only occur at common faces. This assumption is made without loss of generality, as our main results hold for any ambient dimension $N$ capable of embedding these structures. We also fix an integer $N'$ such that $n < N' < N$. In the discussion that follows, $m$ will denote an integer variable such that $0 \le m \le n$.

In this subsection, we first establish the definitions of simple cones and general complex cones of dimension $m$. These general complex cones form the broad collection of geometric objects modeled on simplicial complexes. From this general collection, we identify a specific subset of interest: sets of type $m$ for all integers $m\in \{0,1,\dots ,n\}$. These are products of $\RR^{m}$ and $(n-m)$-dimensional complex cones that satisfy an additional requirement, which we call the non-flat condition (see Definition \ref{非平坦条件}). We will refer to the cones in this subset as non-flat cones. The main idea of this condition is to make sure that boundaries do not disappear when we take the union of cones. For example, if two half-planes meet at an angle of 180 degrees, their common boundary becomes an interior line and is no longer a boundary. Our condition avoids such cases, ensuring that the boundaries do not degenerate into interior points.

Our main goal is to classify these non-flat cones into disjoint categories, which we call sets of type $m$ ($0 \le m \le n$) in Definition \ref{type m}. To validate this classification, we first study the structure of non-flat cones. In Lemma \ref{complex cone的内部是局部和plane重合的点}, we show that the spine (or boundary) of a non-flat cone of dimension $m$ is itself a non-flat complex cone of dimension $m-1$. This implies that the class of non-flat cones is closed under taking spines. Using this result, we prove in Proposition \ref{每个set对应的type唯一} that our definition of `type' is well-posed, meaning that these types form a disjoint partition (i.e., a set cannot be of both type $m$ and type $m'$ for $m\neq m'$). Finally, in Proposition \ref{A(W)的完备性}, we show that the class of non-flat cones is stable when taking blow-up limits.
}

\begin{definition}\label{简单锥} {\rm {(simple cone of dimension $m$).}}
    {
     Let $X=\{x_1, \dots, x_m\} \subset \mathbb{S}^{N'-1}$ be a set so that  $\{0,x_1, \dots, x_m\}$ is affinely independent. Then we define the simple cone of dimension $m$, $C(X)$, as the cone over the convex hull of $X$. That is, let $conv(X)=\{ \sum_{i=1}^{m} \lambda_i x_i, \text{ }\lambda_i\geq 0,\enspace \sum_{i=1}^{m} \lambda_i=1\}$ and define
    \begin{equation}
        C(X)=\{ tx: t\geq 0\text{ and }x\in conv(X) \}.
    \end{equation}
    For a general closed set $E \subset \mathbb{R}^N$, we say $E$ is a simple cone of dimension $m$ if it can be written as $E=C(X)$ for some such set $X$.
    }
    
    In addition, we say  $\emptyset$ is of dimension -1 and $C(\emptyset)=\{0\}$ is a simple cone of dimension 0. 
    For each $Y\subset X$ such that { $\#  Y=t\in\{0,\dots m\}$}, we say $C(Y)$ is a $t$-boundary of $C(X)$. 

\end{definition}

\begin{definition}{\rm {(complex cone of dimension $m$).}}\label{复合锥}
    {Let $T$ be a set of the form $T=\cup_{i=1}^k C(X_i)$, where $\{C(X_i)\}_{i=1}^k$ is a collection of distinct simple cones of dimension $m$ in $\RR^{N'}$ and $k \in \mathbb{N}_+$. Then we say that $T$ is a complex cone of dimension $m$ if the following condition holds: defining $\Omega_T$ by}
      \begin{equation}\label{Omega T}
        \Omega_T=\{Y\subset X_i\text{ for some }1\leq i\leq k\},
    \end{equation} 
     then for every $X, Y\in \Omega_T$,  $C(X)\cap C(Y)=C(X\cap Y)$.

\end{definition}

\begin{definition}\label{角的定义}{\rm {($\angle(F_1,F_2)$).}} Let $T=\ccup_{i=1}^k C(X_i)$ be a complex cone of dimension $m$ and let $F_1,F_2$ be two different non-empty elements in $ \Omega_T$, where $F_1\not\subset F_2$ and $F_2\not\subset F_1$. Suppose $F_1\cap F_2=Z$ and $\#  Z=t$.  Let
\begin{equation}\label{角F1F2的定义}
    \angle(F_1,F_2)=\{\angle f_1 z f_2: z\in C(Z),\text{ } f_j\in C(F_j)\backslash C(Z),\text{ }\dd(f_j, C(Z))=|f_j- z|,\text{ }j=1,2\}.
\end{equation}
Fix $F_1,F_2$, then $\inf \angle(F_1,F_2)>0$. Let 
\begin{equation}\label{角T}
\angle(T)=\min\{\inf\angle(F_1,F_2):F_i\in \Omega_T,i=1,2\}, 
\end{equation}
then we have $\angle(T)>0$.
\end{definition}

{
\begin{remark}
In the case where $F_1 \cap F_2 = \emptyset$, we have $Z = \emptyset$ and by Definition \ref{简单锥}, $C(Z) = \{0\}$. Consequently, the point $z$ in \eqref{角F1F2的定义} must be the origin $0$. In this case, the definition of $\angle(F_1, F_2)$ simplifies to the set of standard Euclidean angles $\angle f_1 0 f_2$ between any non-zero vectors $f_1 \in C(F_1)$ and $f_2 \in C(F_2)$ at the origin.
\end{remark}
}

\begin{definition}\label{非平坦条件}{\rm {(non-flat condition)}.} Let $T=\ccup_{i=1}^k C(X_i)$ be a complex cone of dimension $m$. We say $T$ satisfies the non-flat condition if the following holds. Let $Y\in \Omega_T$ and suppose $\#  Y=t (0\leq t<m)$. Then for every $Z_1,Z_2\in \Omega_T$ such that $Y=Z_1\cap Z_2$ and $\# Z_1=\#  Z_2=t+1$,  $\sup\angle(Z_1, Z_2)<\pi$.
\end{definition}

\begin{definition}{\rm {(type $m$).}}\label{type m}
    A set of type $m$ in $\RR^N$ is a set 
    $W=R(T\times \RR^{m}\times\{0\})$, 
    where $T$ is a complex cone of dimension $n-m$ that satisfies non-flat condition, R is an isometry in $\RR^N$ and $0\in \RR^{N-N'-m}$. 
\end{definition}

Note that if we move $W$  along the direction parallel to $R(\{0\}\times \RR^m\times\{0\})$,  the set { obtained coincides with} $W$.  It is clear that a set of type $n$ is an $n$-plane in $\RR^N$.  And we say $W_1\sim W_2$ if and only if there exists an isometry $R$ in $\RR^N$ such that $W_1=R(W_2)$.

\begin{definition}{\rm {(spine).}}\label{spine}
    Suppose $W=R (T \times \RR^{m}\times \{0\})$ is a set of type $m$, where $T=\ccup_{i=1}^k C(X_i)$ and $k\in \mathbb{N}_{+}$. We can define the $(m+t)$-spine $L^{m+t}(W)$ of $W$ for $ 0\leq t\leq n-m$:
    \begin{equation}
        L^{m+t}(W)=R( \ccup_{Y\in \Omega_T,\#  Y=t} (C(Y)  \times \RR^m \times \{0\}) ). 
    \end{equation}
    As a result, $L^n(W)=W$ and $L^m(W)=R(\{0\}\times \RR^m\times \{0\})$, where the first zero is in $\RR^{N'}$ and the second is in $ \RR^{N-N'-m}$. In addition, $L^{m+t}(W)$ is of dimension $m+t$ and $L^{m_1}(W)\subset L^{m_2}(W)$ for all $m\leq m_1<m_2\leq n $. For each $Y\in \Omega_T$ such that $\#  Y=t$, we say $R(C(Y)  \times \RR^m \times \{0\} )$ is a branch of $L^{m+t}(W)$. $L^{m+t}(W)$ is composed by finitely many such branches, that is, $L^{m+t}(W)=\cup_{l} L^{m+t,l}(W)$. For each $0\leq s<m$, set $ L^s(W)=\emptyset$. And we say $L^0(W)$ is the center of $W$ when $W$ is of type 0.
\end{definition}

{ 
By Definition  \ref{简单锥}, a simple cone $C(X)$ is a closed convex set. It is obvious that the relative interior $C(X)^{\circ}$ satisfies $C(X)^{\circ} = C(X) \setminus \bigcup \{C(Y) : Y \subset X, \#Y = m-1\}$. Furthermore, each $s$-boundary of $C(X)$ is itself a simple cone of dimension $s$. Consequently, if $T$ is a complex cone as in Definition \ref{复合锥}, then for any $t\in \{0,\dots ,m\}$, the set $\cup_{Y\in \Omega_T,\# Y=t} C(Y)$ is a complex cone of dimension $t$. This follows directly from the property that $C(X) \cap C(Y) = C(X \cap Y)$ for all $X, Y \in \Omega_T$.
}

{ 
In order to prove that the definition of type in Definition \ref{type m} is well-posed (which will be established in Proposition \ref{每个set对应的type唯一}), we first need to examine the hierarchical structure of non-flat cones. Specifically, in the following Lemma \ref{complex cone的内部是局部和plane重合的点}, we show that the class of non-flat cones is closed under taking spines. This structural result will then be used in the proof of Proposition \ref{每个set对应的type唯一} to demonstrate the uniqueness of the type.
}


\begin{lemma}\label{complex cone的内部是局部和plane重合的点}
    Let $T=\ccup_{i=1}^k C(X_i)$ be a complex cone of dimension $m\in\{1,\dots,n\}$ that satisfies the non-flat condition. Let 
    \begin{equation}
        RI_m(T)=\{x\in T: T\cap B(x,r)=P\cap B(x,r) \text{ for some }r>0 \text{ and $m$-plane }P\}
    \end{equation}
     be the ``relative interior" of $T$. Then we have
    \begin{equation}\label{式子：complex cone的边界}
        RI_m(T)=\cup_{i} C(X_i)^{\circ}.
    \end{equation}
    In addition, $T\backslash RI_m(T)= \cup_{Y\in \Omega_T,\#  Y=m-1} C(Y)$ is a complex cone of dimension $m-1$ that satisfies the non-flat condition.
\end{lemma}

\begin{proof}
     We first show (\ref{式子：complex cone的边界}). It is obvious that  $C(X_i)^{\circ}\subset RI_m(T)$ for each $i$. 
    For the converse, we claim that
    \begin{equation}\label{C(Y)内部和T的内部不交}
        C(Y)^{\circ}\cap RI_m(T)=\emptyset \text{ for each $Y\in \Omega_T$ such that $\#  Y=m-1$.}
    \end{equation} 
    Otherwise, there is $x\in C(Y)^{\circ}\cap RI_m(T)$ for some $Y\subset X_i$ such that $\#  Y=m-1$. 
    
    Since $x\in RI_m(T)$, there is $r>0$ and an $m$-plane $P$ such that $T\cap B(x,r)=P\cap B(x,r)$. It indicates that we can find $j\neq i$ such that $C(Y)\subset C(X_j)$. Otherwise, $Y\not\subset X_i\cap X_j$ for any other $j\neq i$, then $C(Y)^{\circ} \cap C(X_j)=\emptyset$. Since $C(X_j)$ is closed and $x\in C(Y)^{\circ}$, there exists $\rho_0>0$ such that $ B(x,\rho)\cap C(Y)\subset C(Y)^{\circ}$ and $ B(x,\rho)\cap C(X_j)=\emptyset$ for each $\rho<\rho_0$. Without loss of generality, let $\rho<r$, then $T\cap B(x,\rho)= C(X_i)\cap B(x,\rho)$. The intersection is a half $m$-plane, which contradicts $T\cap B(x,r)=P\cap B(x,r)$. So we can find $j\neq i$ such that $C(Y)=C(X_i)\cap C(X_j)$.
    Denote by $P'$ the $(m-1)$-plane that contains $C(Y)$. Then we have $P'=$aff$(C(Y))\cap B(x,r)\subset$aff$(T\cap B(x,r))=$aff$(P\cap B(x,r))\subset P$. Therefore, $P'\subset P$. Denote by $\pi$ the orthogonal
    projection onto the subspace that orthogonal to $P'$. Then $\pi(T\cap B(x,r))=\pi(P)\cap \pi(B(x,r))$, where $\pi(P)$ is a line passing through 0. At the same time, assume that $X_i\backslash Y=\{x_i\}$. Then $ \pi(C(X_i)\cap B(x,r))=\pi(C(X_i))\cap \pi(B(x,r))=C(\{\pi(x_i)\})\cap \pi(B(x,r))$, where $C(\{(x_i)\})$ is a ray with 0 as its beginning. And the condition for $j$ is the same. 
    Thus we can find $p_i\in C(X_i)$, $p_j\in C(X_j)$ such that $\dd(p_i, C(Y))=|p_i-x|$, $\dd(p_j,C(Y))=|p_j-x|$
    and $\angle p_i x p_j=\pi$, which contradicts the non-flat condition. Thus (\ref{C(Y)内部和T的内部不交}) follows. 
    
    Now we consider points in $C(Y)\backslash C(Y)^{\circ}$ for each $i$ and each $Y\subset X_i$ such that $\#  Y=m-1$. If there exists $x\in C(Y)\backslash C(Y)^{\circ}$ such that $x\in RI_m(T)$, then we can find $r>0$ such that $B(x,r)\cap T\subset RI_m(T)$. But $x$ is a limit point of $C(Y)^{\circ}$, thus $B(x,r)\cap C(Y)^{\circ}\neq \emptyset$, which contradicts (\ref{C(Y)内部和T的内部不交}). So we have
    \begin{equation}
        (\cup_{Y\in \Omega_T,\#  Y=m-1} C(Y)) \cap RI_m(T)=\emptyset.
    \end{equation}
    Therefore, (\ref{式子：complex cone的边界}) follows.

    Note that $T\backslash RI_m(T)=\cup_{Y\in \Omega_T,\#  Y=m-1} C(Y)$ is a direct result of (\ref{式子：complex cone的边界}). 
    {  To verify that this set is non-flat, recall that  Definition \ref{非平坦条件}  imposes conditions on the angles at the $t$-boundaries for every dimension $0 \le t < m$.  Since $T$ satisfies these conditions for the full range up to $m$, it automatically satisfies the requirements for the restricted range $0 \le t < m-1$. Thus, $T \setminus RI_m(T)$ satisfies the non-flat condition.
    
    Consequently, $T \setminus RI_m(T)$ is a complex cone of dimension $m-1$ that satisfies the non-flat condition, and Lemma \ref{complex cone的内部是局部和plane重合的点} follows.} 
\end{proof}

{ 
Using the hierarchical structure established in Lemma \ref{complex cone的内部是局部和plane重合的点}, we now prove that the definition of type is well-posed; that is, the type of a set is unique.
}


\begin{proposition}\label{每个set对应的type唯一}
    Let $W=R(T\times \RR^m\times\{0\})$ be a set of type $m(0\leq m\leq n)$, then for any $m'\neq m$, $W$ is not of type $m'$.
\end{proposition}

\begin{proof}
Without loss of generality, let $m'<m$. Consider the subset {  of points in $W$ where the set is locally flat of dimension $n$:} $RI_n(W)=\{x\in W: W\cap B(x,r)=P\cap B(x,r)$ for some $r>0$ and $n$-plane $P\}$. For a given $W$, $RI_n(W)$ 
{  is well-defined}. 
If $m=n$, then $W_{n-1}=\emptyset$.

 {  
 If $m<n$, recall that $W = R(T \times \mathbb{R}^m \times \{0\})$, where $T$ is a complex cone of dimension $n-m$. Observe that a point $x \in W$ is locally flat of dimension $n$ (i.e., $x \in RI_n(W)$) if and only if its projection onto $T$ lies in the relative interior of $T$. Therefore, we have the explicit identification:
 \begin{equation}
     W \setminus RI_n(W) = R\left( (T \setminus RI_{n-m}(T)) \times \mathbb{R}^m \times \{0\} \right).
 \end{equation}
 By applying Lemma \ref{complex cone的内部是局部和plane重合的点} to $T$, we know that $ T \setminus RI_{n-m}(T)= \cup_{Y\in \Omega_T,\#  Y=n-m-1} C(Y)$ is a non-flat complex cone of dimension $n-m-1$. Consequently, $W_{n-1}=R(( \cup_{Y\in \Omega_T,\#  Y=n-m-1} C(Y))\times \RR^{m}\times\{0\})$. 
 }
 
 When $m<n$, define $RI_{n-1}(W_{n-1})$ and let $W_{n-2}=W_{n-1}\backslash RI_{n-1}(W_{n-1})$ similarly. By repeating this process, we get an ascending collection $\{W_{k}\}_{k=m-1}^{n-1}$ such that $W_m\neq\emptyset$ and $W_{m-1}=\emptyset$. For each $k$, $W_k$ is unique for $W$. But if $W$ is of type $m'$, then $W_{m'}\neq \emptyset$ and therefore $W_{m-1}\neq\emptyset$, which leads to a contradiction. 
 
\end{proof}

{ 
\begin{remark}
    The uniqueness of the type of a set reflects the geometric rigidity enforced by the non-flat condition. Without the non-flat condition, the representation might be ambiguous.
    For instance, let $x_1=(1,0),\ x_2=(-1/2,\sqrt{3}/2)$ and $\ x_3=(-1/2,-\sqrt{3}/2)$ in $\RR^2$. Let $ X_1=\{x_1,x_2\},\ X_2=\{x_1,x_3\}$ and $X_3=\{x_2,x_3\}$. The union $T=\cup_{i=1}^3 C(X_i)$ (which does not satisfy the non-flat condition) in $\mathbb{R}^2$ forms the entire plane $\mathbb{R}^2$. This set $T$ could be viewed as a complex cone of dimension 2. However, it is also isometric to $\mathbb{R} \times \mathbb{R}$, which is a set  of type 1 (as in Definition \ref{type m}), making its type ill-defined. 
    The non-flat condition precludes such ambiguity and ensures that the type of a set is uniquely determined. Consequently, a non-flat cone cannot be represented as the product of a lower-dimensional complex cone and $\mathbb{R}^t$ ($t>0$).
\end{remark}
}

 Despite Proposition \ref{每个set对应的type唯一},  a set of type $m$ may coincide with a set of higher type locally.  Precisely speaking, let $W=R(T\times \RR^m\times\{0\})$ be a set of type $m$ and let $x\in W$. Suppose that $x\in L^{t,l}(W)\backslash L^{t-1}(W)$ for some $t\in\{m,...,n\}$. Then we can find a constant $\lambda:=\lambda(W,t,l)>1$, depending only on $W,t,l$, such that for each ball $B(x,r)$ satisfying $\lambda B(x,r)\cap L^{t-1}(W)=\emptyset$, the blow-up limit of $W$ at the point $x$
\begin{equation}\label{局部锥的定义式}
    C( W\cap B(x,r)-x)=\{ cy: y\in W\cap B(x,r)-x\text{ and }c\geq 0\}
\end{equation}
is a set of type $t$. We denote it by $W(t,l)$. Let 
\begin{equation}
    \A (W)=\{R( W(t,l)): m\leq t\leq n, l\in\mathbb{N}, R \text{ is an isometry in $\RR^N$}\}.
\end{equation} 
It is obvious that $\A(W)/\sim$ is finite.

Let $\mathscr{TA}=\{W: W\text{ is of type }m,\text{ }0\leq m\leq n\}$. For each subset $\B\subset \mathscr{TA}$, let $\A(\B)=\cup_{W\in \B}\A(W)$.
Then we have the following proposition which says that $\A(W)$ is complete for each $W$.

\begin{proposition}\label{A(W)的完备性}
For each $0\leq m\leq n$ and each set $W$ of type $m$,
    \begin{equation}
        \A(\A(W))=\A(W).
    \end{equation}
\end{proposition}

\begin{proof}
    It is clear that $\A(W)\subset \A(\A(W))$, so we only need to prove the converse. If $T\in\A(\A(W))$, then $T\in \A(X)$ for some $X\in \A(W)$. Without loss of generality, suppose $ X=C(W\cap B(x,r) -x)$. Then we can find $ y\in X$ and $\rho>0$ such that $T=C(X\cap B(y,\rho)-y)=C(C(W\cap B(x,r)-x)\cap B(y,\rho)-y)$. We can choose $y\in W-x$ close enough to $0$ and $\rho>0$ small enough such that $B(y,\rho)\subset B(0,r)$ and does not meet $X\backslash (W\cap B(x,r)-x)$. Then $T=C( (W\cap B(x,r)-x) \cap B(y,\rho)-y)=C(W\cap B(x,r)\cap B(x+y,\rho)-(x+y))=C(W\cap B(x+y,\rho)-(x+y))$, thus $T\in \A(W)$.
\end{proof}

By similar argument in Proposition \ref{A(W)的完备性}, we can show in Proposition \ref{A(t)的包含性} that for a given $W$, a set in $\A(W)$ of higher type is the blow-up limit of a set in $\A(W)$ of lower type.

\begin{proposition}\label{A(t)的包含性}
    Let $0\leq m<n$ and let $W$ be a set of type $m$. Then for each $m\leq t\leq n-1$, if $W_{t+1}\in \A(W)$ is a set of type $t+1$, then we can find $W_t\in \A(W)$ be of type $t$, such that $W_{t+1}\in \A(W_t)$.
\end{proposition}

\begin{proof}
    Without loss of generality, assume that $m=0$. If $t=0$, let $W_t=W$ and the proposition follows. Now suppose $t>0$. Since $W_{t+1}$ is of type $t+1$, we can find $Y\in \Omega_T$ with $\#  Y=t+1$ and $\lambda_1>0$ such that $W_{t+1}=C(W\cap B(y,\rho)-y)$ for each $y\in C(Y)^{\circ}$ and each $\rho>0$ such that $\lambda_1 B(y,\rho) \cap L^{t}(W)=\emptyset$. (In order to be precise, we should write  $y\in C(Y)^{\circ}\times \{0\}$, where $0\in \RR^{N-N'}$, but for convenience, $\{0\}$ can be omitted.) 
    
    Let us construct $W_t\in \A(W)$ now. 
    Choose $X\in \Omega_T$ such that $X\subset Y$ and $\#  X=t$, then $C(X)$ is a $t$-boundary of $C(Y)$. Let $W_t$ be a blow-up limit of $W$ at $C(X)$. That is, choose $x\in C(X)^{\circ}$ and $r>0$ small, then let
    \begin{equation}
        W_t:= C(W\cap B(x,r)-x).
    \end{equation} 
    Also, there is $\lambda_2>0$ such that for each $z\in C( C(Y)^{\circ}-x)$ and $\tau>0$ satisfying $\lambda_2 B(z,\tau)\cap L^{t}(W_t)=\emptyset$, $W_t$ coincides with a set of type $t+1$ in $B(z,\tau)$. 
    Choose $z\in C(Y)^{\circ}-x$ close to 0 and $ \tau>0$ small enough such that $B(z,\tau)\subset B(0,r)$, $B(z,\tau)\cap W_t\backslash (W\cap B(x,r)-x)=\emptyset$, $\lambda_1 B(z+x,\tau)\cap L^t(W)=\emptyset$ and $\lambda_2 B(z,\tau)\cap L^t(W_t)=\emptyset$. Then $ C(W_t\cap B(z,\tau)-z)=C(C(W\cap B(x,r)-x)\cap B(z,\tau)-z)=C(W\cap B(x+z,\tau)-(x+z))$. Since $x+z\in C(Y)^{\circ}$ and $\lambda_1 B(z+x,\tau)\cap L^t(W)=\emptyset$, $C(W_t\cap B(z,\tau)-z)=W_{t+1}$. Thus $W_{t+1}\in \A(W_t)$.
\end{proof}

\subsection{Geometrical facts of $\B$}

In this subsection, we only consider the case when $\B\subset \mathscr{TA}$ is such that $\B/\sim$ is finite. Fix $\B\subset \mathscr{TA}$ such that  $\mathscr{B}/\sim$ is finite and $\mathscr{B}$ contains at least a set of type 0. Let $\A=\A(\B)$ and $\A=\cup_{m=0}^n \A(m)$, where $\A(m)$ is the set of sets of type $m$ in $\A$. Then $\A/\sim$ is a finite set and for each $0\leq m\leq n$, $\A(m)\neq\emptyset$.
Let 
\begin{equation}\label{alpha定义}
    \alpha(\B)=\min\{ \angle (T): W=R(T\times \RR^m\times\{0\} )\in \A(m), 0\leq m\leq n\},
\end{equation}
where $\angle(T)$ is defined in (\ref{角T}), then $\alpha (\B)>0$. According to the construction, sets of different types in $\A$ have different topological structures. 

{  
\begin{lemma}\label{分支分离性}
    Let $W \in \A(\B)$ and let $\alpha = \alpha(\B)$ be the constant defined in \eqref{alpha定义}. Let $L$ and $L'$ be two distinct non-empty branches of the spines of $W$ (possibly of different dimensions).  Suppose that $L \not\subset L'$, and $L' \not\subset L$. Then for any $x \in L$,
    \begin{equation}\label{分支分离性ineq}
        \dd(x, L') \ge \dd(x,  L\cap L') \cdot \sin \alpha.
    \end{equation}
\end{lemma}

\begin{proof}
    If $x\in L\cap L'$, then $\dd(x, L\cap L') = 0$ and the inequality \eqref{分支分离性ineq} holds trivially. Assume $x\not\in L\cap L'$. Then $\dd(x, L\cap L') > 0$. Let $y \in L'$ be a point such that $|x - y| = \dd(x, L')$. Then $y$ lies in the relative interior of a unique branch $L''$ of the spine of $W$ such that $L'' \subset L'$.
    We distinguish two cases $ L''\subset L$ and $L'' \not\subset L $.

    If $L'' \subset L$, then $y\in L'' \subset L \cap L' $. This implies $\dd(x, L') = |x-y| \ge \dd(x, L\cap L')$. Since $0<\sin \alpha<1$, \eqref{分支分离性ineq} holds.

    If $L'' \not\subset L$, let $P$ be the affine plane spanned by $L''$. Since $y$ lies in the relative interior of $L''$, the vector $x - y$ is orthogonal to  $P$. Let $z$ be the closest point to $x$ in $S = L \cap L''\subset L\cap L'$. Then $|x-z|=\dd(x,S)$. Since $S \subset L'' \subset P$ and $(x-y) \perp P$, the closest point to $x$ in $S$ coincides with the closest point to $y$ in $S$. Thus, $|y-z|=\dd(y,S)$. By Definition \ref{角的定义}  and \eqref{alpha定义}, we have $\angle xzy \ge \alpha$. Therefore, we have $\dd(x, L') = |x - y| = |x - z|\cdot \sin(\angle xzy) \ge \dd(x, S)\cdot \sin \alpha \ge \dd(x,  L\cap L')\cdot \sin \alpha$.
    This completes the proof.

\end{proof}

}

\begin{proposition}\label{锥的dx,r设定}
There is $\delta_0>0$ depending only on $\B$ such that for each $0\leq m\leq n-1$,
\begin{equation}
    d_{x,r}(W_m,W)>\delta_0, \forall r>0
\end{equation}
where $W_m$ is an arbitrary set in $\A(m)$, whose $t$-spine $\LL^{t}(W_m)$ ($m\leq t\leq n-1$) passes through  $x$. $W$ is an arbitrary set in $\ccup_{k=t+1}^{n} \A(k)$.
\end{proposition}

\begin{proof}
    We may assume that $x=0$ and $r=1$. We aim to show that when the two constants $m\in \{0,...,n-1\}$ and $t\in \{m,...,n-1\}$ are fixed, 
    \begin{equation}\label{固定Wm的d0,1}
        \inf\{ d_{0,1}(W_m, W): W_m\in \A(m),\text{ }0\in L^t(W_m),\text{ } W\in \ccup_{k=t+1}^n \A(k)\}>0.
    \end{equation}
    Let us prove by compactness. If (\ref{固定Wm的d0,1}) does not hold, by the finiteness of $\B$, there is $W_m\in \A(m)$ such that $0\in L^t(W_m)$ and a series of isometries $\{R_i\}_{i=1}^{\infty}$ in $\RR^N$ such that $0\in R_i( L^t(W_m))$, satisfying $\lim_{i\to\infty} \inf \{ d_{0,1}(R_i(W_m), W):W\in\cup_{k=t+1}^n \A(k)\}=0$. Similarly, there is a subsequence $\{i_j\}_{j=1}^{\infty} \subset \{i\}_{i=1}^{\infty}$ (for convenience, let $\{i_j\}_{j=1}^{\infty} = \{i\}_{i=1}^{\infty} $), a set $Y\in \A(s)$ $(t+1\leq s\leq n)$ and a series of isometries 
    $\{R'_{i}\}_{i=1}^{\infty}$ in $\RR^N$ such that $\lim_{i\to\infty} d_{0,1}(R_{i}(W_m), R'_{i}(Y))=0$.
    Since $0\in R_i( L^t(W_m))$ for each $i$, we can always find a subsequence of $\{ R_i(W_m)\}_{i=1}^{\infty}$ that converges to a set of type $m$, whose $t$-spine contains 0. Denote by $T$ this set and assume that $\lim_{i\to\infty}R_i(W_m)=T$.

    Now we have  $\lim_{i\to\infty} d_{0,1}(T, R_i'(Y))=0$. Without loss of generality, assume that $0\in L^s(Y)$. For each $i$, suppose  $R_i'(x)=A_i x+b_i$, where $A_i\in O(N)$ and $b_i\in \RR^N$. Then $\{A_i\}_{i=1}^{\infty}$ has a subsequence that converges. For convenience, we assume that $\{A_i\}_{i=1}^{\infty}$ converges. If $\{b_i\}_{i=1}^{\infty}$ has a subsequence that converges, then there is an isometry $R$ such that $d_{0,1}(W_m, R(Y))=0$. Otherwise, $\lim_{i\to\infty} b_i=\infty$. If we have
    \begin{equation}\label{L^s(Y)不远}
        \sup_j \dd(0, R_{i_j}'(L^s(Y)))<M \text{ for some }\{i_j\}_{j=1}^{\infty}\subset \{i\}_{i=1}^{\infty} \text{ and }0<M<\infty ,
    \end{equation}
    then there is $w_j\in R_{i_j}'(L^s(Y))$ such that $|w_j|<M$ for each $j$. Let $R_{i_j}''(x)=A_{i_j}x+w_j$. Then $w_j-b_{i_j}\in L^s(Y)$ and therefore, $ R_{i_j}''(Y)\cap B(0,1)=R_{i_j}'(Y)\cap B(0,1)$ when $j$ is large enough. So we can replace $R_{i_j}'$ by $R_{i_j}''$, where $\{w_j\}_{j=1}^{\infty}$ is bounded, thus converges. So we can find an isometry $R$ such that $d_{0,1}(T, R(Y))=0$.
    If (\ref{L^s(Y)不远}) does not hold, but we have
    \begin{equation}\label{L^{s+1}(Y)不远}
        \sup_j \dd(0, R_{i_j}'(L^{s+1,l}(Y)))<M \text{ for some }l,\{i_j\}_{j=1}^{\infty}\subset \{i\}_{i=1}^{\infty} \text{ and }M<\infty ,
    \end{equation}
    then we can find $w_j\in R_{i_j}'(L^{s+1,l}(Y))$ such that $|w_j|<M$. When $j$ is large enough, $R_{i_j}(Y)$ coincides with $R_{i_j}(Y_{s+1})$ in $B(w_j,M+2)$ for some $Y_{s+1}\in \A(Y)\cap \A(s+1)$. Since $B(0,1)\subset B(w_j,M+2)$,  (\ref{L^{s+1}(Y)不远}) turns to the case that $ \sup_j \dd(0, R_{i_j}'(Y_{s+1}))<M$, which is similar to (\ref{L^s(Y)不远}) and we can also find a subsequence that converges. Recall that $B(0,1)\cap R_{i_j}'(Y)\neq \emptyset$, we always have $B(0,1)\cap L^n(Y)\neq\emptyset$. That is, 
    { by repeating the argument inductively for higher dimensional spines as in \eqref{L^{s+1}(Y)不远}, we} obtain the fact that 
    \begin{equation}
        d_{0,1}(T,W)=0
    \end{equation}
    for some $W\in \A(s), t+1\leq s\leq n $. That is, $T\cap B(0,1)=W\cap B(0,1)$ and therefore, $ L^s(T)\cap B(0,1)=L^{s}(W)\cap B(0,1)$. Since $0\in L^t(T)$ and $t<s$, $L^s(T)\cap B(0,1)$ is not empty. It indicates that a set whose type is at most $t$ coincides with a set of type $s$, which contradicts Proposition \ref{每个set对应的type唯一}. Then we have (\ref{固定Wm的d0,1}). Let 
    \begin{equation}
        \begin{split}
        \delta_0= \min \{\inf\{ d_{0,1}(W_m, W): W\in \ccup_{k=t+1}^n \A(k), W_m\in \A(m),0\in L^{t}(W_m)\},\\
        0\leq m\leq n,m\leq t\leq n-1\},
        \end{split}
    \end{equation}
  then $\delta_0>0$ and Proposition \ref{锥的dx,r设定} follows.
  
\end{proof}

\begin{proposition}\label{用次一级的锥替代}
    There is $n_0>0$ depending only on $\B$ satisfying the following. 
    For each $m\in\{0,...,n-1\}$ and $ t\in\{m,...,n-1\}$, let $W$ be an arbitrary set in $\A(m)$. Suppose B is an open ball such that $B\cap W \neq \emptyset$ and $n_0 B$ does not meet $L^t(W)$.  Then there is a set $Y\in \cup_{u=t+1}^n \A(u)$ such that 
    \begin{equation}\label{用次一级的锥替代式子}
        W\cap B=Y\cap B.
    \end{equation}
    Furthermore, if $W$ does not coincide with any set in $\A(t+1)$ in an open ball $B'$, then $n_0 B'\cap L^t(W)\neq\emptyset$.
\end{proposition}

\begin{proof}
Let $W\in \A(m)$ be fixed. There is $\lambda_{W}>0$ only depends on $W$ such that for each $s\in\{m,...,n-1\}$, if $\tilde{B}$ is an open ball centered on $L^{s+1}(W)$ and $ \lambda_{W} \tilde{B}\cap L^s(W)=\emptyset$, then $W\cap \tilde{B}=Y\cap \tilde{B}$ for some $Y\in\A(s+1)\cap \A(W)$.

Now fix a constant $ \tau_{W}>3\cdot (10 \lambda_W)^n$, we are going to show that for each $t \in \{m,... ,n-1\}$, if $B=B(x,r)$ is an open ball such that $B\cap W\neq\emptyset\text{ and }\tau_{W} B\cap L^t=\emptyset,$ then 
\begin{equation}\label{用t+1到nset替代}
    W\cap B=Y\cap B \text{ for some }Y\in (\cup_{u=t+1}^n \A(u))\cap \A(W). 
\end{equation}
Let $d_s=\dd(x, L^s(W))$ for each $s\in \{t,...,n-1\}$. Then we have $d_n<r$ and $d_t \geq \tau_{W} r$. So there exists $s\in \{t,...,n-1\}$ such that 
\begin{equation}\label{式子：离L^s远L^s+1近}
d_s\geq 10 \lambda_{W}(d_{s+1}+r).
\end{equation} Otherwise, we have
    $d_t<10 \lambda_{W}(d_{t+1}+r)<...<(10 \lambda_{W})^{n-t}(d_n+r)+\frac{(10 \lambda_W)^{n-t}-10 \lambda_{W}}{10 \lambda_{W}-1}r <\tau_{W} r$, 
which leads to a contradiction. Thus, there is $y\in L^{s+1}(W)$ such that $|x-y|=d_{s+1}$ and $ \lambda_W B(y, 2d_{s+1}+r)\cap L^s(W)=\emptyset$. And we can find $Y\in \A(s+1)\cap \A(W)$ such that $Y\cap B(y,2 d_{s+1}+r)=W\cap B(y,2 d_{s+1}+r)$. Since $B=B(x,r)\subset B(y,2d_{s+1}+r)$ and $t+1\leq s+1\leq n$, (\ref{用t+1到nset替代}) follows. Let $n_0=\max\{\tau_W,W\in \A\}$, then we have (\ref{用次一级的锥替代式子}). 

Let $B'$ be as in the statement of Proposition \ref{用次一级的锥替代}, if $n_0 B'\cap L^t(W)=\emptyset$, then there is $Y\in \A(u)\cap \A(W)$ such that $W\cap B'=Y\cap B'$, where $t+1\leq u\leq n$.  By Proposition \ref{A(t)的包含性}, there is $Z\in \A(t+1)\cap\A(W)$ such that $Y\in \A(Z)$, so $W\cap B'=Z\cap B'$, which leads to a contradiction. And Proposition \ref{用次一级的锥替代} follows.

\end{proof}

\begin{lemma}\label{3个集合的dx,r}
    Let $x,y\in \RR^N$. Let $F,G,H$ be three closed sets in $\RR^N$ such that $d_{x,r_1}(F,G)<\varepsilon_1$ and $d_{y,r_2}(G,H)<\varepsilon_2,$
    where $r_i>0, \varepsilon_i>0$ for $i=1,2$. Then for every $z\in \RR^N$ and $\rho>0$ such that 
    $B(z,\rho+\varepsilon_1 r_1)\subset B(y,r_2)$ and $B(z,\rho+\varepsilon_2 r_2)\subset B(x,r_1)$,
    we have $d_{z,\rho}(F,H)<(\varepsilon_1 r_1+\varepsilon_2 r_2)/{\rho}$.
\end{lemma}

\begin{proof}
    For each $w\in F\cap B(z,\rho)$, we have $w\in B(x,r_1)$. Since $d_{x,r_1}(E, G)<\varepsilon_1$, there exists $w'\in G$ such that $|w-w'|<\varepsilon_1 r_1$. Then $w'\in G\cap B(z,\rho+\varepsilon_1 r_1)\subset G\cap B(y,r_2)$. Since $d_{y,r_2}(G,H)<\varepsilon_2$, there exists $w''\in H$ such that $|w'-w''|<\varepsilon_2 r_2$. Therefore, $\dd(w, H)<\varepsilon_1 r_1+\varepsilon_2 r_2$ for each $w\in F\cap B(z,\rho)$. By the same argument, we have $\dd(w,F)<\varepsilon_1 r_1+\varepsilon_2 r_2$ for each $w\in H\cap B(z,\rho)$. Hence, $d_{z,\rho}(F,H)<(\varepsilon_1 r_1+\varepsilon_2 r_2)/\rho$.
\end{proof}





\begin{lemma}\label{type m集合若dxr距离小于delta，则最高级spine距离受控}
    Let $W\in \A(m)(0\leq m< n)$ and let its $m$-spine pass through x. If
    $d_{x,r}(W,Z)<\tau$
    for some $Z\in \A$ and $\tau\leq\delta_0$, then $Z\in \ccup_{k=0}^m \A(k)$ and
    \begin{equation}\label{距离受控的式子}
        \dd (x,L^{m}(Z))<n_0(1+\delta_0^{-1})\tau r.
    \end{equation}
\end{lemma}

\begin{proof}
    By Proposition \ref{锥的dx,r设定}, $Z\notin \ccup_{k=m+1}^{n} \A(k)$. Hence $ Z\in \ccup_{k=0}^m \A(k)$.
    So let us prove (\ref{距离受控的式子}).
    If $\dd(x,\LL^{m}(Z))\geq n_0(1+\delta_0^{-1})\tau r$, then $n_0 B(x,(1+\delta_0^{-1})\tau r)$ does not meet $\LL^{m}(Z)$.
    By Proposition \ref{用次一级的锥替代}, there exists a set $Y\in \ccup_{k=m+1}^{n} \A(k)$ such that $Y\cap B(x,(1+\delta_0^{-1})\tau r)=Z\cap B(x,(1+\delta_0^{-1})\tau r)$.
    
    Now, consider $d_{x,\tau r/\delta_0}(W,Y)$, where the radius $\tau r/\delta_0\leq \min\{(1+\delta_0^{-1})\tau r,r\}$. For each $y\in W\cap B(x,\tau r/\delta_0 )$, there is $y'\in Z$  such that $|y-y'|<\tau r$ because $d_{x,r}(W,Z)<\tau$. Then $|y'-x|<(1+\delta_0^{-1})\tau r$, which indicates that $y'\in Y$ and  $\dd(y,Y)<\tau r$. For each $y\in Y\cap B(x,\tau r/\delta_0)$, it is evident that $y\in Z$ and $\dd(y,W)<\tau r$. Thus $d_{x,\tau r/\delta_0}(W,Y)<\delta_0$. This contradicts Proposition \ref{锥的dx,r设定} because $W\in \A(m)$ and $x$ is contained in its $m$-spine, while $Y\in\ccup_{k=m+1}^n \A(k)$. Lemma \ref{type m集合若dxr距离小于delta，则最高级spine距离受控} follows.
\end{proof}


{ 
Proposition \ref{锥的dx,r设定} shows the stability of sets of type 0: specifically, if $d_{x,r}(W, Z) < \delta_0$ for some $W \in \mathcal{A}(0)$ centered at $x$, then $Z$ must also be of type 0. In the following lemma, we extend this stability result to sets of type $m > 0$.
}

\begin{lemma}\label{引理：两个锥dxr距离小，则另一个也是一样的type}
     Fix an integer $m$ such that $0<m\leq n$. Let $W\in \A(m)$ and let its $m$-spine $L^{m}(W)$ pass through $x$. If 
    \begin{equation}
        d_{x,r}(W,Z)<20^{-1}n_0^{-n}\delta_0
    \end{equation}
    for some $Z\in \A$, then we can find $Y_m\in \A(m)$ such that
    \begin{equation}\label{引理式子：Z局部是typem}
        Y_m\cap B(x,10^{-1}n_0^{-n}r)=Z\cap B(x,10^{-1}n_0^{-n}r).
    \end{equation}
\end{lemma}

\begin{proof}

    It is evident that $Z \in \ccup_{k=0}^m \A(k)$ based on Proposition \ref{锥的dx,r设定}. If $Z \in \A(m)$, then (\ref{引理式子：Z局部是typem}) follows. Otherwise, if $Z \in \A(t)$ for some $0 \leq t < m$, we aim to demonstrate that $Z$ coincides with sets of increasing type within a sequence of progressively smaller balls, which are centered at $x$.
    
    For each $s\in\{t,..., m\}$, set
     \begin{equation}
         \theta_{s}=10^{-1}n_0^{-m}\cdot \sum_{k=0}^{m-s} n_0^k.
     \end{equation}
     Then $\{\theta_{s}\}_{s=t}^{m}$ is a decreasing sequence with $\theta_{t}<1/2$ and $\theta_m=10^{-1}n_0^{-m}$.
     In addition, $\theta_{s}=n_0\theta_{s+1}+ \theta_m$. 
     

    Now we want to show that:
    \begin{center}
        $Z\cap B(x,\theta_{t} r)=Y_{t}\cap B(x,\theta_{t} r)$ for some $Y_{t}\in \A(t)$,
        
        $Z\cap B(x,\theta_{t+1} r)=Y_{t+1}\cap B(x,\theta_{t+1} r)$ for some $Y_{t+1}\in \A(t+1)$,

        ...

        $Z\cap B(x,\theta_m r)=Y_m\cap B(x,\theta_m r)$ for some $Y_m\in \A(m)$.
    \end{center}

    Otherwise, assume $q$ is the first index for which the statement fails. We have $q>t$ since the first statement is always true. That is, for all $Y\in \A(q)$, we have $Z\cap B(x,\theta_q r)\neq Y\cap B(x,\theta_q r)$. Additionally, for all  $p\in\{t,...,q-1\}$, $Z\cap B(x, \theta_p r)=Y_p\cap B(x,\theta_p r) $ for some $Y_p\in \A(p)$.
    By Proposition \ref{用次一级的锥替代}, $n_0 B(x,\theta_q r)\cap L^{q-1}(Z)\neq\emptyset$. Hence, there exists $l\in L^{q-1}(Z)$ such that $|x-l|<n_0\theta_q r$. At the same time, $Z$ coincides with $Y_{q-1}$ in $B(x,\theta_{q-1} r)$, where $\theta_{q-1}=n_0\theta_q+\theta_m$, thus $l\in L^{q-1}(Y_{q-1})$. The ball $B(l,\theta_m r) $ is contained in $B(x,\theta_{q-1}r)$, thus we can consider $d_{l,\theta_m r/2}(Y_{q-1}, W)$. For every $y\in W\cap B(l,\theta_m r/2)$, there is $y'\in Z$ such that $|y-y'|<20^{-1}n_0^{-n}\delta_0 r$. So $y'\in Y_{q-1}$, implying $\dd(y,Y_{q-1})<20^{-1}n_0^{-n}\delta_0 r$. For every $y\in Y_{q-1}\cap B(l,\theta_m r/2)$, $y\in Z\cap B(x,r)$ so we can find $y'\in W$ such that $|y-y'|<20^{-1}n_0^{-n}\delta_0 r$. Therefore, $\dd(y,W)<20^{-1}n_0^{-n}\delta_0 r$ and $d_{l,\theta_m r/2}(Y_{q-1}, W)<\delta_0$. It contradicts Proposition \ref{锥的dx,r设定} since $Y_{q-1}$ is a set in $\A(q-1)$ with $l\in L^{q-1}(Y_{q-1})$ and  $W$ is of type greater than $q-1$. Then (\ref{引理式子：Z局部是typem}) holds because $m\leq n $ and Lemma \ref{引理：两个锥dxr距离小，则另一个也是一样的type} follows.
    
\end{proof}

\section{Main theorem and the decomposition of E}
For each $\B\subset \mathscr{TA}$ such that $\B/\sim$ is finite, we can get the constants $\alpha:=\alpha(\B)${ (defined in \eqref{alpha定义}),} $\delta_0${ (from Proposition \ref{锥的dx,r设定})}, $n_0${ (from Proposition \ref{用次一级的锥替代}) } that only depend on $\B$. Now we state the main theorem again.

\noindent
\textbf{Theorem 1.1} \emph{For each $\B\subset\mathscr{TA}$ such that $ \B/\sim$ is finite, there exist $C=C(\B)$, $\varepsilon_0=\varepsilon(\B)$ that depend only on $\B$ s.t. the following holds. Let $E\subset \RR^N$ be a closed set that contains the origin and $\varepsilon<\varepsilon_0$. If for each $x\in E$ and radius $r>0$, there is a set $Z(x,r)\in \A(\B)$ that contains $x$, such that
    \begin{equation}
        d_{x,r}(E,Z(x,r))<\varepsilon,
    \end{equation}
    then there is a set $Z\in\A(\B)$ through the origin and an injective mapping $f: B(0,1.95)\to B(0,2)$, with the following properties:
    \begin{equation}\label{f的包含与被包含2}
        B(0,1.9)\subset f(B(0,1.95))\subset B(0,2),
    \end{equation}
    \begin{equation}\label{f作用于Z(0,2)2}
        E\cap B(0,1.9)\subset f(Z\cap B(0,1.95))\subset E\cap B(0,2),
    \end{equation}
    \begin{equation}\label{f Holder2}
        (1+C\varepsilon)^{-1} |x-y|^{1+C\varepsilon}< |f(x)-f(y)|<(1+C\varepsilon)|x-y|^{1/(1+C\varepsilon)}\text{ for }x,y\in B(0,1.95), 
    \end{equation}
    \begin{equation}\label{f移动一点距离2}
        |f(x)-x|<C\varepsilon\text{ for }x\in B(0,1.95).
    \end{equation} }

Without loss of generality, we can let $\delta_0<1/10$ and $n_0>10$. Set $ N_0=n_0/ \delta_0$, thus $N_0>100$.
Let $E$ be as in the main theorem. Now we need to classify points of $E$ into different types. For $x\in E\cap B(0,2)$ and $r>0$ such that $B(x,r)\subset B(0,2)$, set
\begin{equation}
    a_m(x,r)=\inf\{d_{x,r}(E,W):W\in \A(m)\enspace and\enspace x\in\LL^{m}(W)\}
\end{equation}
for each $0\leq m\leq n$.

\begin{definition}
    Let $x\in E\cap B(0,2)$. We say that $x\in E_m$ if there exists $r_x>0$ such that $a_m(x,r)< C_0\varepsilon$ for all $0<r<r_x$, where $C_0=10^8(n_0+n)^{3n} N_0$.
\end{definition}

\begin{lemma}\label{Em不交}
    $\{E_m\}_{m=0}^n$ are disjoint.
\end{lemma}

\begin{proof}
    Suppose that we can find $x \in E_i \cap E_j$ where $i \neq j$. For sufficiently small $r > 0$, we have $a_i(x,2r) < C_0\varepsilon$, and there exists a set $W_i \in \A(i)$ whose $i$-spine passes through $x$ and satisfies $d_{x,2r}(E,W_i) < C_0\varepsilon$. The same holds for $j$. Therefore, we can conclude that $d_{x,r}(W_i,W_j) < 4C_0\varepsilon < \delta_0$. This contradicts Proposition \ref{锥的dx,r设定}. Hence, we can conclude that the sets $\{E_m\}_{m=0}^n$ are disjoint.
\end{proof}

\subsection{Properties of $E_m$}

In this subsection, we aim to investigate the following properties of $E_m$ ($0\leq m\leq n$).

First, in Lemma \ref{lem:小影响大}, we will show how being close to a set of type $m$ at a small scale determines the behavior of the set at larger scales.
Also, we will show that $E_m$ is close to an $m$-plane locally in Lemma \ref{Em的m-flatness}. Let us prove these by employing an induction approach from $m=0$ to $m=n$.

\begin{lemma}\label{lem:小影响大}
 Let $x\in E$, $r>0$ such that $B(x,20 C_1 r)\subset B(0,2)$ and $a_m(x,r)\leq C_2$, where $C_1=n_0^n  N_0$ and $C_2=1/(200 N_0 )$. For $m>0$, we additionally require that $B(x,20C_1 r)\cap (\ccup_{k=0}^{m-1} E_k)=\emptyset$. Let $Z=Z(x,20C_1 r)$. Then there exists $Z_m\in \A(m)$ such that $Z\cap B(x, 9 N_0r )=Z_m\cap B(x,9 N_0r)$. Denote by $\LL^{m}$ the $m\mbox{-}$spine of $Z_m$. Then $\dd(x,\LL^{m})<4N_0 a_m(x,r)r+2N_0 (20C_1+1)\varepsilon r$. 
 Moreover, we have
\begin{equation}
    a_m(x,8 N_0r)<\frac{1}{2} a_m(x,r)+ 7C_1\varepsilon.
\end{equation}
\end{lemma}

\begin{corollary}\label{推论：Em中的点满足am(x,r)很小，即使r很大}
 Let $x\in E_m,r>0$ such that $B(x,5 n_0^n r/2)\subset B(0,2)$. For $m>0$, we additionally require that $B(x,5 n_0^n r/2)\cap (\ccup_{k=0}^{m-1} E_k)=\emptyset$. Then $a_m(x,r)<14C_1\varepsilon$.
\end{corollary}

\begin{lemma}\label{Em的m-flatness}
  Let $x\in E$, $r>0$ such that $B(x,r)\subset B(0,2)$ and $a_m(x,r)<C_3\varepsilon$, where $C_3=100(n+n_0)^{2n}$. That is, we can find $W_m\in \A(m)$ whose $m\mbox{-}$spine $\LL^{m}$ passes through $x$ and such that $d_{x,r}(W_m, E)<C_3\varepsilon$.  Then we have
    \begin{equation}\label{lem,flatness: Em is m-rei}
    \begin{aligned}
         &\dd(x, E_0)<C_4\varepsilon r,&\text{ when } m=0\\
         &d_{x,0.99r}(E_m,\LL^{m})<(0.99)^{-1}C_4\varepsilon,&\text{ when } m>0
    \end{aligned}
    \end{equation}
    where $C_4=10^3 C_3 N_0 n_0^{n^2}$. When $m>0$, $B(x,0.99r)$ does not meet $\ccup_{k=0}^{m-1} E_k$.
\end{lemma}

First we want to prove that Lemma \ref{lem:小影响大}, Corollary \ref{推论：Em中的点满足am(x,r)很小，即使r很大} and Lemma \ref{Em的m-flatness} are true for $m=0$.

\begin{proof}[\rm \textbf{Proof of Lemma \ref{lem:小影响大} for the base case \boldmath$m=0$.}]
    Since $a_0(x,r)\leq C_2=1/(200 N_0)$, we can find a set $W_0\in \A(0)$ whose center is $x$ and $d_{x,r}(W_0,E)<2a_0(x,r)\leq 2 C_2$. Let $\rho=[2a_0(x,r)+(20 C_1+1)\varepsilon]r/ \delta_0$. {  Recall that $N_0=n_0/\delta_0$, as defined in the text following \eqref{f移动一点距离2}.} It follows that $\rho<r/(10n_0)$ for $\varepsilon$ small ($\varepsilon$ depends only on $\delta_0, n_0$) and $d_{x,\rho}(Z,W_0)<\delta_0$ since $d_{x,20C_1 r}(Z,E)<\varepsilon$. By Lemma \ref{type m集合若dxr距离小于delta，则最高级spine距离受控}, we have $Z\in \A(0)$ and $|x-z_0|<n_0(1+\delta_0)\rho$, where $z_0$ denotes the center of $Z$. It is clear that $Z$ coincides with  itself in $B(x,9N_0 r)$. 
    Next, we estimate $a_0(x,8N_0 r)$. Let $Z'=Z+(x-z_0)$, then $Z'\in\A(0)$ and is centered at $x$. 
      By Lemma \ref{3个集合的dx,r}, $d_{x,8N_0 r}(Z', E)<\frac{n_0(1+\delta_0)\rho+ 20 C_1 \varepsilon r }{8N_0 r}<\frac{1}{2}a_0(x,r)+7 C_1\varepsilon$. 
     Therefore, $a_0(x,8N_0 r)\leq d_{x,8N_0r}(Z',E)<\frac{1}{2}a_0(x,r)+7 C_1\varepsilon$. Thus, Lemma \ref{lem:小影响大} holds for $m=0$.
\end{proof}

    Then we are ready to prove Corollary \ref{推论：Em中的点满足am(x,r)很小，即使r很大} for $m=0$.
    
\begin{proof}[\rm \textbf{Proof of Corollary \ref{推论：Em中的点满足am(x,r)很小，即使r很大} for the base case  \emph{m} = 0}]

    Since $x\in E_0$, we have $a_0(x,(8N_0)^{-k}r)<C_0\varepsilon$ for $k>0$ large enough. By applying Lemma \ref{lem:小影响大}, we obtain 
     $a_0(x,(8N_0)^{-k+1}r)< \frac{1}{2} a_0(x,(8N_0)^{-k} r)+7C_1\varepsilon<C_0\varepsilon/2+ 7C_1\varepsilon<C_0\varepsilon$.
    By repeating this argument, we have $a_0(x,r)< 2^{-k}C_0\varepsilon +  7C_1\varepsilon \cdot(\sum_{l=0}^{k-1}(1/2)^l)$. This inequality holds for all sufficiently large $k$. Thus, we have $a_0(x,r)\leq \lim_{k\to\infty} (C_0\varepsilon/ 2^k+ 7C_1\varepsilon\cdot(\sum_{l=0}^{k-1}(1/2)^l) )\leq 14C_1\varepsilon$. Therefore, Corollary \ref{推论：Em中的点满足am(x,r)很小，即使r很大} holds for $m=0$.
\end{proof}

\begin{proof}[\rm \textbf{Proof of Lemma \ref{Em的m-flatness} for the base case  \emph{m} = 0}]
        
   Now we are ready to prove Lemma \ref{Em的m-flatness} for $m=0$. Let $W_0\in \A(0)$ be centered at $x$ such that $d_{x,r}(W_0,E)<C_3\varepsilon$. Set $Z=Z(x,r)$, then $x\in Z$ and $d_{x,r}(Z,E)<\varepsilon$. Set $\rho=(C_3+2)\varepsilon r/\delta_0$. Then $d_{x,\rho}(W_0,Z)<(C_3\varepsilon r+\varepsilon r)/\rho<\delta_0$ by Lemma \ref{3个集合的dx,r}.
   And by Lemma \ref{type m集合若dxr距离小于delta，则最高级spine距离受控}, $Z\in \A(0)$. Let $z_0$ be the center of $Z$, then $|z_0-x|<n_0(1+\delta_0)\rho<2N_0(C_3+2)\varepsilon r $. Since $d_{x,r}(Z,E)<\varepsilon$, we can find $x_1\in E$ such that $|x_1-z_0|<\varepsilon r$. Consequently, $|x_1-x|<2N_0(C_3+2)\varepsilon r+\varepsilon r$. Set $Z'=Z+(x_1-z_0)$, then $Z'\in \A(0)$ and it is centered at $x_1$, at the same time, $a_0(x_1,r/2) \leq d_{x_1,r/2}(Z',E)<4\varepsilon$. The pair $(x_1,r/2)$ satisfies the condition of this lemma, namely, $B(x_1,r/2)\subset B(0,2)$ and $a_0(x_1,r/2)<4 \varepsilon$. By repeating the discussion above, we obtain a series of points $\{x_k\}_{k=1}^{\infty}\subset E$. These points satisfy $a_0(x_k,r/2^k)<4\varepsilon$ and $|x_{k+1}-x_k|<(12 N_0 +1)2^{-k}\varepsilon r$. Let $\xi=\lim_{k\to\infty} x_k$. Since $E$ is closed, $\xi\in E$. Moreover, we have $|\xi-x|\leq \sum_{k=1}^{\infty}|x_{k+1}-x_k|+|x_1-x|< (2C_3 N_0+16 N_0+2)\varepsilon r$.

    To prove that $\xi\in E_0$, we consider the value of $a_0(\xi,t)$ for $0<t<r/10$. Fix $t$, and choose $k\geq 1$ such that $\frac{9}{10}\frac{r}{2^{k+1}}<t\leq \frac{9}{10}\frac{r}{2^k}$. We already know that $a_0(x_k,{r}/{2^k})<4\varepsilon$. Hence, there exists $Y\in \A(0)$ centered at $x_k$ such that $d_{x_k,r/2^k}(Y,E)<4 \varepsilon$. Moreover, we have $|x_k-\xi|\leq \sum_{l=k}^{\infty}|x_{l+1}-x_l|=(12N_0+1)\varepsilon r\cdot 2^{-k+1}$. Consider the set $Y'=Y+(\xi-x_k)$, which belongs to $\A(0)$ and is centered at $\xi$. By using $Y'$ to estimate $a_0(\xi,t)$, we have $a_0(\xi,t)\leq d_{\xi,t}(Y',E)< 55N_0\varepsilon$. This implies that $\xi\in E_0$. Therefore, $\dd(x,E_0)\leq  |x-\xi|< (2C_3 N_0+16N_0+2)\varepsilon r< C_4\varepsilon r$. Thus, Lemma \ref{Em的m-flatness} holds for $m=0$.

\end{proof}

Before passing to $m>0$, we need to state another conclusion which indicates that $E_0$ is a single point in a large region of $B(0,2)$.

\begin{lemma}\label{E0只有一个点}
    There is at most one point in $E_0\cap B(0,1.99)$.
\end{lemma}

\begin{proof}

Suppose $x \neq y$ are two distinct points in $E_0 \cap B(0,1.99)$. Consequently, for any $0 < r < 10^{-3} n_0^{-n}$, we have $B(x,5 n_0^n r) \subset B(0,2)$ and $B(y, 5 n_0^n r) \subset B(0,2)$. By applying Corollary \ref{推论：Em中的点满足am(x,r)很小，即使r很大} with $m=0$, we find that $a_0(x,2 r) < 14 C_1\varepsilon$ and $a_0(y,2 r) < 14 C_1 \varepsilon$. Let $\rho = |x-y|$. If $\rho < 10^{-3} n_0^{-n}$, it follows that $a_0 (x,2\rho) < 14 C_1 \varepsilon$. Hence, there exists a set $W_1 \in \mathcal{A}(0)$ centered at $x$, such that $d_{x,2\rho}(W_1,E) < 14 C_1 \varepsilon$. The same conditions hold for $y$, yielding that there is $W_2\in \A(0)$ centered at $y$, such that $d_{y,2\rho}(W_2,E)<14 C_1\varepsilon$. 
Then we have $d_{x,\rho/(2n_0)}(W_1,W_2) < \delta_0$. By Lemma \ref{type m集合若dxr距离小于delta，则最高级spine距离受控}, we deduce that $|x-y|<n_0(1+\delta_0)\frac{\rho}{2n_0}<\rho$, which contradicts the assumption that $|x-y|=\rho$. Alternatively, if $\rho \geq 10^{-3} n_0^{-n}$, we have the fact that $a_0(x,10^{-3} n_0^{-n})<14 C_1 \varepsilon$, which implies the existence of $W_1 \in \mathcal{A}(0)$ centered at $x$ with $d_{x,10^{-3} n_0^{-n}}(W_1,E) < 14 C_1 \varepsilon$. Let $Z=Z(0,2)$. Since $Z$ passes through 0 and satisfies $d_{0,2}(E,Z) < \varepsilon$, we have $d_{x,10^{-4} n_0^{-n-1}}(Z,W_1) < \delta_0$. Consequently, $Z \in \mathcal{A}(0)$. Denote by $z_0$ the center of $Z$, then we obtain $|z_0-x|<n_0(1+\delta_0)\cdot{10^{-4}} {n_0^{-n-1}}<(1+\delta_0)10^{-4}n_0^{-n} $ based on Lemma \ref{type m集合若dxr距离小于delta，则最高级spine距离受控}. Similarly, $|z_0-y|< (1+\delta_0)10^{-4}n_0^{-n}$ holds. Thus, $|x-y|<4\cdot
10^{-4} n_0^{-n}$, contradicting the assumption $\rho\geq 10^{-3} n_0^{-n}$. Consequently, the initial assumption is incorrect, and Lemma \ref{E0只有一个点} follows.
\end{proof}

Assume that Lemma \ref{lem:小影响大}, Corollary \ref{推论：Em中的点满足am(x,r)很小，即使r很大} and Lemma \ref{Em的m-flatness} hold for $0,...,m-1$, where $ m\geq 1$, now we are ready to prove them for $m$.

\begin{proof}[\rm \textbf{Proof of Lemma \ref{lem:小影响大} for the inductive step of dimension \boldmath$m$.}]

    Recall that $C_1=n_0^n N_0$ and $C_2=\frac{1}{200 N_0}$. Since $a_m(x,r)\leq C_2$, there exists $Y\in \A(m)$ such that $x\in L^m(Y)$ and $d_{x,r}(Y,E)<2a_m(x,r)\leq 2C_2$. Set $\rho=[2a_m(x,r)r+(20C_1+1)\varepsilon r]/ \delta_0$. It follows that $\rho<r/(10 n_0)$. Let $Z=Z(x,20 C_1 r)$, then $d_{x,\rho}(Y,Z)<(2a_m(x,r)r+20 C_1\varepsilon r)/\rho<\delta_0$ by Lemma \ref{3个集合的dx,r}.
    Since $Y\in \A(m)$ and $x$ is contained in the $m\mbox{-}$spine of $Y$, by Lemma \ref{type m集合若dxr距离小于delta，则最高级spine距离受控}, we deduce that $Z\in\ccup_{k=0}^m \A(k)$ and 
    \begin{equation}\label{lem小影响大中x到L^m(Z)距离}
        \dd(x,L^{m}(Z))<n_0(1+\delta_0)\rho.
    \end{equation}
    Then we continue to prove that there is $Z_m\in \A(m)$ such that 
    \begin{equation}\label{lem小影响大中:Z局部是type m}
        Z_m\cap B(x,9N_0 r)=Z\cap B(x,9N_0 r).
    \end{equation}
     When $Z\in\A(m)$, it is clear that (\ref{lem小影响大中:Z局部是type m}) holds. Otherwise, suppose $Z\in\A(k)$ for some $0\leq k<m$, to establish
    (\ref{lem小影响大中:Z局部是type m}), we aim to demonstrate that $Z$ coincides with sets of increasing type within a sequence of progressively smaller balls.
    
    Set $\theta_m=9N_0$ and $\lambda=n_0^{-n}/2$. For $k \leq s<m$, set
    \begin{equation}
        \theta_s= n_0^{m-s}\theta_m+(\lambda+40 C_1\varepsilon)\sum_{l=0}^{m-s-1} n_0^l.
    \end{equation}
    Then $\{\theta_s\}_{s=k}^m$ is a decreasing sequence with $\theta_k\leq 9 C_1+1/10$. In addition, $\theta_s=n_0 \theta_{s+1}+\lambda+40 C_1 \varepsilon$. Then we want to show that the following statement is true,
    \begin{center}
        $Z\cap B(x,\theta_s r)=Z_s\cap B(x,\theta_s r)$ for some $Z_s\in\A(s)$ for $s$ from $k$ to $ m$.
    \end{center}
    Otherwise, assume $q$ is the first number for which the statement is not true. It is clear that $q>k$ because the first statement is always true. Then for all $Y\in\A(q)$, $Z\cap B(x,\theta_q r)\neq Y\cap B(x,\theta_q r)$, but there exists $ Z_{q-1}\in\A(q-1)$ such that $Z\cap B(x,\theta_{q-1} r)=Z_{q-1}\cap B(x,\theta_{q-1}r)$. By Proposition 
    \ref{用次一级的锥替代}, $n_0 B(x,\theta_q r)\cap L^{q-1}(Z)\neq\emptyset$.
     Thus, we can find $y\in L^{q-1}(Z)$ such that $|y-x|< n_0\theta_q r$ and $e\in E$ such that $|y-e|<20 C_1\varepsilon r$.
     Hence, $y\in L^{q-1}(Z_{q-1})$. Set $Z'=Z_{q-1}+(e-y)$, then $Z'\in\A(q-1)$ and $e$ passes through the $(q-1)$-spine of $Z'$. By Lemma \ref{3个集合的dx,r}, we have $d_{e,\lambda r}(Z',E)<40 C_1\varepsilon r/(\lambda r)<80 n_0^{2n} N_0\varepsilon <C_3\varepsilon$. Thus, $a_{q-1}(e,\lambda r)\leq d_{e,\lambda r}(Z',E)<C_3\varepsilon $.
     Recall we have assumed that Lemma \ref{Em的m-flatness} holds for $0,...,m-1$, 
     so $d_{e,0.99\lambda r}(E_{q-1} ,L^{q-1}(Z') )< 0.99^{-1} C_4\varepsilon$ when $q>1$ and $\dd( e, E_0)<C_4\varepsilon\lambda r $ when $q=1$,  where $C_4=10^3 C_3 N_0 n_0^{n^2}$. 
     Consequently, $\dd(x, E_{q-1})\leq |x-e|+ C_4 \varepsilon\lambda r<20 C_1 r$, which contradicts the fact that $ B(x,20C_1 r)\cap (\ccup_{t=0}^{m-1} E_t)=\emptyset$ because $q\leq m$. Thus, (\ref{lem小影响大中:Z局部是type m}) follows. 
     By (\ref{lem小影响大中:Z局部是type m}) and   $L^m=L^m(Z_m)$, we have $ L^m\cap B(x,9N_0 r)=L^m(Z)\cap B(x,9N_0 r)$. By (\ref{lem小影响大中x到L^m(Z)距离}) and $\delta_0<1$, we get that $\dd (x,\LL^{m})<2 n_0 \rho=4N_0 a_m(x,r)r+2N_0(20C_1+1)\varepsilon r$. 
 
     At last,  we consider $a_m(x,8N_0 r)$. Pick a point $w\in\LL^{m}$ such that $|w-x|<4N_0 a_m(x,r)r+2N_0(20 C_1+1)\varepsilon r$ and set $W=Z_m+(x-w)$. Then $W\in \A(m)$ and $x\in L^m(W)$. Thus, we have $a_m(x,8N_0 r)\leq d_{x,8N_0 r}(W,E)<\frac{1}{2}a_m(x,r)+7C_1\varepsilon$ and Lemma \ref{lem:小影响大} for $m$ follows.

\end{proof}

\begin{proof}[\rm \textbf{Proof of Corollary \ref{推论：Em中的点满足am(x,r)很小，即使r很大} for the inductive step of dimension \boldmath$m$.}]

    The proof is essentially the same as for $m=0$. Since $x\in E_m$, we know $a_m(x,(8N_0)^{-k}r)<C_0\varepsilon$ for all $k$ large enough. Then $a_m(x,r)<{2}^{-k} C_0\varepsilon +7C_1\varepsilon\cdot \sum_{l=0}^{k-1} {2}^{-l}$.  Taking the limit as $k$ tends to infinity, we have $a_m(x,r)\leq 14C_1 \varepsilon$. Corollary \ref{推论：Em中的点满足am(x,r)很小，即使r很大} for $m$ follows.

\end{proof}

\begin{proof}[\rm \textbf{Proof of Lemma \ref{Em的m-flatness} for the inductive step of dimension \boldmath$m$.}]
   
    Recall that $C_3=100(n+n_0)^{2n}$ and $C_4=10^3 C_3 N_0 n_0^{n^2} $. In order to show that (\ref{lem,flatness: Em is m-rei}) holds for $m>0$, we need to prove
    \begin{equation}\label{Em离m-spine近}
        \dd(y, L^m)<C_4\varepsilon r,  \text{ for all } y\in E_m\cap B(x,0.99r)
    \end{equation}
    and
    \begin{equation}\label{引理的式子：骨架上的点离Em近}
        \dd(y,E_m)<C_4\varepsilon r \text{ for all } y\in L^m\cap B(x,0.99r).
    \end{equation} 
    
    Let us prove (\ref{Em离m-spine近}). When $m=n$, $L^m=W_m$. Since $d_{x,r}(E,W_m)<C_3\varepsilon$ and the constant $C_3$ is much smaller than $C_4$, it is evident that (\ref{Em离m-spine近}) holds. Therefore, we focus on the situations when $m<n$.
    For each $z\in E\cap B(x,0.99r)$ such that $\dd(z,\LL^{m})\geq C_4\varepsilon r $, let us check that 
    \begin{equation}\label{式子：z附近没有更高级的点}
        B(z,10^{-3}n_0^{-tn }r)\cap (\ccup_{k=0}^{t-1} E_k)=\emptyset, \text{ for each } t\in\{1,...,m\}
    \end{equation}
    by induction from $t=1$ to $t=m$. For $t=1$, we need to prove that $B(z,10^{-3}n_0^{-n} r )\cap E_0=\emptyset$. Otherwise, there is $z'\in E_0\cap B(z, 10^{-3}n_0^{-n} r  )$ while the distance between $z'$ and the boundary of $B(x,r)$ is greater than $ 9.9\cdot 10^{-3}r$. Thus $ B(z',(5n_0^n/2) \cdot {10^{-4}}{n_0^{-n}}r)\subset B(x,r)\subset B(0,2)$. Using Corollary \ref{推论：Em中的点满足am(x,r)很小，即使r很大} for 0, we get $a_0(z',{10^{-4}}{n_0^{-n}}r)<14C_1 \varepsilon$, which indicates that there is $T\in \A(0)$ centered at $z'$, satisfying  $d_{z',{10^{-4}}{n_0^{-n}}r}(T,E)<14 C_1 \varepsilon$. By Lemma \ref{3个集合的dx,r}, $d_{z',{10^{-5}}{n_0^{-n}}r}(T,W_m)<(C_3+ 1.4\cdot 10^{-3}N_0)\varepsilon r /({10^{-5}}{n_0^{-n}}r ) <\delta_0$.
     It contradicts the fact that $T\in \A(0)$ is centered at $z'$ and $W_m\in \A(m)$ for $m>0$.
     Hence, we can conclude that $B(z,10^{-3} n_0^{-n}r)$ does not meet $E_0$.

    Assume that we have proved $B(z,{10^{-3}}{n_0^{-tn}}r)\cap (\ccup_{k=0}^{t-1} E_k)=\emptyset$, here $t$ is a number smaller than $m-1$. Now we want to prove that $B(z,{10^{-3}}{n_0^{-(t+1)n}}r)$ does not meet $E_t$. Otherwise, pick $z'\in B(z,{10^{-3}}{n_0^{-(t+1)n}}r)\cap E_t$, then we have $({5}n_0^n/2) B(z',{10^{-4}}{n_0^{-(t+1)n}}r)\subset B(z,{10^{-3}}{n_0^{-tn}}r)\subset B(0,2)$. By assumption, $(5n_0^n/2) B(z',{10^{-4}}{n_0^{-(t+1)n}}r)$ does not meet $\ccup_{k=0}^{t-1} E_k$. Since we have assumed that Corollary 
    \ref{推论：Em中的点满足am(x,r)很小，即使r很大} holds for $0,...,m-1$, it follows that $a_t(z',{10^{-4}}{n_0^{-(t+1)n}}r)<14C_1\varepsilon$. i.e. there exists $T\in \A(t)$ whose $t\mbox{-}$spine $\LL^{t}(T)$ passes through $z'$, such that $d_{z',{10^{-4}}{n_0^{-(t+1)n}}r}(T,E)<14C_1\varepsilon$. By a similar argument as for $t=0$, we can  deduce that $d_{z',{10^{-5}}{n_0^{-(t+1)n}}r}(T,W_m)<\delta_0$. 
    It contradicts the fact that $T\in\A(t)$ with $z'\in L^t(T)$ and $W_m\in\A(m)$, where $t\leq m-1$.
    Thus we show that $B(z,{10^{-3}}{n_0^{-(t+1)n}}r)$ does not meet $E_t$. Furthermore, $B(z,{10^{ -3}}{n_0^{-(t+1)n}}r) \cap (\cup_{k=0}^t E_k)=\emptyset$, and (\ref{式子：z附近没有更高级的点}) follows.
   
    According to the discussion above, we have established that $z\notin \ccup_{k=0}^{m-1} E_k$. Now we continue to show that $z\notin E_m$. Suppose that $z\in E_m$. Let $\rho=\min\{\frac{1}{2n_0}\dd(z,\LL^{m}),\frac{1}{200}r\}$, where $L^m$ is the $m$-spine of $W_m$. Then it follows that $n_0 B(z,2\rho)$ does not meet $L^m$. By Proposition \ref{用次一级的锥替代}, we can find $W\in \ccup_{k=m+1}^{n} \A(k)$ such that $W\cap B(z,2\rho)=W_m\cap B(z,2\rho)$.
    Since $\dd(z,\LL^{m})\geq C_4\varepsilon r$, we have $\varepsilon r\leq {2n_0}\rho/C_4$. 
    The fact $\rho\leq r/200$ indicates that $B(z,{ n_0^{-mn}}\rho/{5})$ is contained in $ B(z,{10^{-3}}{n_0^{-mn}}r)$, where $ B(z, 10^{-3}n_0^{-mn}r )$ does not meet $ \ccup_{k=0}^{m-1} E_k$ by (\ref{式子：z附近没有更高级的点}). Moreover, $B(z,{ n_0^{-mn}}\rho/{5})\subset B(0,2)$. Since we have already proved Corollary \ref{推论：Em中的点满足am(x,r)很小，即使r很大} for $m$, we know that $a_m(z,2n_0^{-(m+1)n}\rho / 25 )<14C_1\varepsilon$. Hence, there exists $T\in \A(m)$ with its $m\mbox{-}$spine $\LL^{m}(T)$ passing through $z$, such that $d_{z, 2n_0^{-(m+1)n}\rho/25 }(T,E)<14C_1\varepsilon$. By Lemma \ref{3个集合的dx,r}, $d_{z,n_0^{-(m+1)n}\rho/25}(T,W)<(28N_0 n_0^{-mn}\varepsilon\rho/25+2C_3 n_0\rho/C_4) / (n_0^{-(m+1)n}\rho/25)<\delta_0$. 
    Since the type of $W$ is greater than $m$, it is impossible. Therefore, our assumption that $z\in E_m$ must be false.
    As a conclusion, if $z\in E_m\cap B(x,0.99r)$, then $\dd(z,\LL^{m})<C_4\varepsilon r$. And (\ref{Em离m-spine近}) follows.

   For (\ref{引理的式子：骨架上的点离Em近}), let $y\in \LL^{m}\cap B(x,0.99 r)$ be given. Since $d_{x,r}(E,W_m)<C_3\varepsilon $, we can find $e\in E$ such that $|y-e|<C_3\varepsilon r$. Set $W'=W_m+(e-y)$ and $T=Z(e,r/200)$, then $W'$ is a set of type $m$ with $e\in L^m(W')$ and $d_{e,r/400}(T,W')<(800C_3+2)\varepsilon$. Then we can use  Lemma \ref{引理：两个锥dxr距离小，则另一个也是一样的type} and know that there exists $T_m\in \A(m)$  such that 
    \begin{equation}
        T\cap B(e, \frac{r}{4\cdot 10^{3} n_0^n})=T_m\cap B(e, \frac{r}{4\cdot 10^{3} n_0^n}).
    \end{equation}
    To estimate the distance between $e$ and $\LL^{m}(T_m)$, we consider in a much smaller ball. Set $\rho=5C_3\varepsilon r/\delta_0$, then $d_{e,\rho}(W',T_m)<\delta_0/2$. By Lemma \ref{type m集合若dxr距离小于delta，则最高级spine距离受控}, we get that $\dd(e,\LL^{m}(T_m))<M\varepsilon r$, where $M=(1+\delta_0)5C_3 N_0$. Thus, there is $l\in \LL^{m}(T_m)$  such that $|l-e|<M\varepsilon r$, thus $l\in T$. And there is $e_1\in E$ such that $|l-e_1|<\varepsilon r/200$ because $T=Z(e,r/200)$. It follows that $|e-e_1|\leq (M+1/200)\varepsilon r$, therefore, $B(e_1,10^{-3}n_0^{-n}r/8)$ is contained in $ B(e, 10^{-3} n_0^{-n}r/4)$. Let $T'=T_m+(e_1-l)$ and $r_1=10^{-4}n_0^{-n}r$, then $T'\in\A(m)$ and the $m$-spine of $T'$ passes through $e_1$. Furthermore, $d_{e_1,r_1}(T',E)<100n_0^n\varepsilon <C_3\varepsilon$, which indicates that $a_m(e_1,r_1)<C_3\varepsilon$. The pair $(e_1,r_1)$ also satisfies the condition of this lemma, so we can replace $x,y,e$ with $e_1$ and replace $r$ with $r_1$. By repeating this process, we obtain a sequence $\{e_k\}_{k=1}^{\infty}\subset E$ and $r_k=(10^{-4}n_0^{-n})^k r$, such that $|e_k-e_{k+1}|<(M+1/200)\varepsilon r_k$ and $a_m(e_k,r_k)<C_3\varepsilon$. Let $\xi=\lim_{k\to\infty}e_k$, then we have $\xi\in E$ because $E$ is closed. For any $0<t<10^{-4}n_0^{-n}r/2$, we can find $k\in \mathbb{N}_{+}$ such that $9 r_{k+1}/10 \leq t<9r_{k}/10$. Let $Y\in \A(m)$ be the set such that $d_{e_k,r_k}(Y,E)<C_3\varepsilon$. Then $|\xi-e_k|\leq 3M 10^{-4k} n_0^{-nk}\varepsilon r$ 
    and $a_m(\xi,t)\leq d_{\xi,t}(Y+(\xi-e_k),E)< 10^5 n_0^n (3M+C_3)\varepsilon<C_0\varepsilon$, which indicates that $\xi\in E_m$. 
     Therefore, $\dd(y,E_m)\leq |y-\xi|\leq |y-e_1|+|e_1-\xi|<(C_3 +3M+3)\varepsilon r<C_4\varepsilon r$, (\ref{引理的式子：骨架上的点离Em近}) follows. Thus we conclude that $d_{x,0.99r}(E_m,L^m)<(0.99)^{-1}C_4\varepsilon$ and (\ref{lem,flatness: Em is m-rei}) holds for $m$.

     At last we will prove $B(x,0.99r)\cap (\ccup_{k=0}^{m-1} E_k)=\emptyset$ by induction. For $k=0,...,m-1$, set a sequence of radius $\sigma_k=(1-\frac{k+1}{100n})r$. Then $\{\sigma_k\}_{k=0}^{m-1}$ is decreasing and $\sigma_{m-1}>0.99r$. First we prove that $B(x,\sigma_0 r)$ does not meet $E_0$. Otherwise, we can find $y\in B(x,\sigma_0)\cap E_0$ and thus $B(y,{r}/{(100n)})\subset  B(0,2)$. By Corollary \ref{推论：Em中的点满足am(x,r)很小，即使r很大}, $a_0(y,r/(250 n n_0^n))<14C_1\varepsilon$. So we can find $Y\in \A(0)$ centered at $y$ and while $d_{y,r/(250 n n_0^n)}(Y,E)<14 C_1 \varepsilon $. Therefore, $d_{y,r/(500 n n_0^n)}(Y,W_m)<\delta_0$. This leads to a contradiction.

     Assume that we have proved that $B(x,\sigma_k)\cap E_k=\emptyset$, for $k=0,...,t-1$, while $t\leq m-1$. Now we are ready to show that $B(x,\sigma_t)\cap E_t=\emptyset$. Otherwise, we can find $y\in B(x,\sigma_t)\cap E_t$. Recall that we have proved Corollary \ref{推论：Em中的点满足am(x,r)很小，即使r很大} for $0$ to $m$, as a result, $a_t(y,r/(250 n n_0^n))<14C_1 \varepsilon$. So we can find $Y\in \A(t)$, whose $t\mbox{-}$spine passes through $y$, satisfying $d_{y,r/(250 n n_0^n)}(Y,E)<14 C_1 \varepsilon$. So $d_{y,r/(500 n n_0^n)}(Y,W_m)<\delta_0$. This leads to a contradiction. Thus we have $E_t\cap B(x,\sigma_t)=\emptyset.$ Now we know that $B(x,0.99r)\cap (\ccup_{k=0}^{m-1} E_k)=\emptyset$. Lemma \ref{Em的m-flatness} for $m$ follows.

\end{proof}

{   This ends} the proofs of Lemma \ref{lem:小影响大}, Corollary \ref{推论：Em中的点满足am(x,r)很小，即使r很大} and Lemma \ref{Em的m-flatness} from $0$ to $n$. Then we want to show in Proposition \ref{lem4:用于投影的锥} that near a point of $E_m$, $E_k$ looks like the $k$-spine of a set of type $m$ for each $m\leq k\leq n$. We prove it for $0\leq m\leq n$ by induction from $m=n$ to $m=0$.

\begin{proposition}\label{lem4:用于投影的锥}
    Let $x\in E_m$, $r>0$ be such that $B(x,5n_0^n r/2)\subset B(0,2)$. When $m>0$, we also ask that $B(x,5n_0^n r/2)$ does not meet $\ccup_{k=0}^{m-1} E_k$. Then 
    
    1) we can find $T:=T(x,r)\in \A(m)$  such that $x\in L^{m}$ and
    \begin{equation}\label{lem4:E和T近}
        d_{x,r}(E,T)<14 C_1\varepsilon,
    \end{equation}
    where $C_1=n_0^n N_0$ and $L^m=L^m(T)$. 
    
    2) Moreover, if $m<n$, then for each $k\in\{m,..., n-1\}$, we also have 
    \begin{equation}\label{lem4:Ek和T的k维骨架近}
         d_{x,\lambda_k r}(E_k,\LL^{k})<C_5\varepsilon,
    \end{equation}
    where $L^{k}=L^k(T)$, $C_5=2(100 n_0^{2n} N_0^2 n+10^3 n_0^{2n+1} N_0^2 n^3)$ and $\lambda_k=1-(k+1)/(100n)$.  Thus $\{\lambda_k\}_{k=m}^{n-1}$ is a positive decreasing sequence with $\lambda_m\leq 1-1/(100n)$ and $\lambda_{n-1}=0.99$.
\end{proposition}

\begin{proof}[\rm \textbf{Proof of Proposition \ref{lem4:用于投影的锥} for the base case \boldmath$m=n$.}]
By Corollary \ref{推论：Em中的点满足am(x,r)很小，即使r很大}, when $m=n$, Proposition \ref{lem4:用于投影的锥} follows directly. 
\end{proof}

Assume that Proposition \ref{lem4:用于投影的锥} holds for $m+1,...,n$, where $m\leq n-1$. Let us prove it for $m$. 

\begin{proof}[\rm \textbf{Proof of Proposition \ref{lem4:用于投影的锥} for the inductive step of dimension \boldmath$m$.}]

    Let $ x,r$ be as in the statement. 

    For 1), by Corollary \ref{推论：Em中的点满足am(x,r)很小，即使r很大}, we have $a_m(x,r)<14C_1 \varepsilon$. Therefore, there exists $T\in \A(m)$, whose $m\mbox{-}$spine $L^{m}$ passes through $x$,  such that $d_{x,r}(E,T)<14 C_1 \varepsilon$. Let $T$ be the corresponding $T(x,r)$, then we have (\ref{lem4:E和T近}) and 1) holds.

    For 2), we will prove the 2-sided inequalities (\ref{lem4证明中的式子：Ek离k-spine近}) and (\ref{lem4:骨架上的点离Ek近}) for $m\leq k\leq n-1$,
    \begin{equation}\label{lem4证明中的式子：Ek离k-spine近}
        \dd(y,L^{k})<B_1\varepsilon r \text{ for each }y\in E_k\cap B(x,\lambda_k r),
    \end{equation}
    \begin{equation}\label{lem4:骨架上的点离Ek近}
        \dd(y,E_k)<(B_1+n B_2)\varepsilon r \text{ for each }y\in L^k\cap B(x,\lambda_k r),
    \end{equation}
    where $B_1=100n_0^{2n}N_0^2 n$ and $ B_2=10^3 n_0^{2n+1}N_0^2 n^2$.
    If both (\ref{lem4证明中的式子：Ek离k-spine近}) and (\ref{lem4:骨架上的点离Ek近}) hold, then we have $d_{x,\lambda_k r}(E_k ,L^k)< (B_1+n B_2) \varepsilon/ \lambda_k< C_5\varepsilon$, and Proposition \ref{lem4:用于投影的锥} follows.
     
    For (\ref{lem4证明中的式子：Ek离k-spine近}), fix $k$ and set $\sigma_s=\frac{s-m+1}{k-m+1} B_1$ for $m\leq s\leq k$. Then $\{\sigma_s\}_{s=m}^k$ is an increasing sequence with $\sigma_k=B_1$. We will prove the following claim by induction on $s$ from $m$ to $k$.
    \begin{equation}\label{y离Lk远，y不属于Es}
        \text{ For any } y\in E\cap B(x,\lambda_k r) \text{ and any } m\leq s\leq k, \text{ if } \dd(y,\LL^{k})\geq \sigma_s \varepsilon r, \text{ then } y\notin E_s.
    \end{equation}
    We first prove (\ref{y离Lk远，y不属于Es}) for $s=m$.  Assume (\ref{y离Lk远，y不属于Es}) is not true when $s=m$, that is, we can find $y\in E_m\cap B(x,\lambda_k r)$ such that $\dd(y,L^{k})\geq \sigma_m \varepsilon r$. Then $B(y,\sigma_m {\varepsilon r}/{n_0})\subset B(x,r) $ and  $n_0 B(y,\sigma_m {\varepsilon r}/{n_0} )\cap L^{k}=\emptyset$. By Proposition \ref{用次一级的锥替代}, we can find $Y\in \A(k+1)$ such that $Y\cap B(y, \sigma_m {\varepsilon r}/{n_0})=T\cap B(y,\sigma_m {\varepsilon r}/{n_0})$. At the same time, we have $B(y,\sigma_m {\varepsilon r}) $ does not meet $\ccup_{l=0}^{m-1} E_l$ since $B(x,5n_0^n r/2)\cap (\ccup_{l=0}^{m-1} E_l)=\emptyset$. Thus we can use Corollary \ref{推论：Em中的点满足am(x,r)很小，即使r很大} and get that $a_m(y,2\sigma_m {\varepsilon r}/ (5n_0^n) )<14 C_1\varepsilon$. That is, there exists $W_m\in \A(m)$ such that $d_{y, 2\sigma_m {\varepsilon r}/ (5n_0^n) }(W_m,E)<14C_1\varepsilon$ while $y\in L^{m}(W_m)$. By Lemma \ref{3个集合的dx,r}, $d_{y,\sigma_m {\varepsilon r}/ (5n_0^n) }(W_m,Y)<(14 C_1\varepsilon r+ 14 C_1\varepsilon\cdot( 2\sigma_m {\varepsilon r}/ (5n_0^n)) ) / ( \sigma_m {\varepsilon r}/ (5n_0^n))< (14 C_1+1)5 n_0^n n/ B_1<\delta_0$. 
    Since $Y\in \A(k+1)$, $W_m\in\A(m)$, $y\in L^m(W_m)$ and $k\geq m$, it leads to a contradiction to Proposition \ref{锥的dx,r设定}. As a consequence, the assumption that there exists $y\in E_m$ such that $\dd(y, L^k)\geq \sigma_m \varepsilon r$ is not true. Thus (\ref{y离Lk远，y不属于Es}) holds for $s=m$.
    
   Now assume that we have proved (\ref{y离Lk远，y不属于Es}) from $s=m$ to $t-1$, where $t$ is a number in $\{m+1,...,k\}$.
   We want to prove that 
        if $\dd(y, L^{k})\geq \sigma_t\varepsilon r$, then $y\notin E_{t}$ for every $ y\in E\cap B(x,\lambda_k r)$.
    Suppose not, then there exists $y\in E_{t}$ such that $\dd(y,L^{k})\geq \sigma_t \varepsilon r$. Consider the ball $B(y,\sigma_m\varepsilon r)$, where $\sigma_m=B_1/(k-m+1)$. For each $z\in B(y,\sigma_m\varepsilon r)$, the distance between $z$ and $L^{k}$ is greater than $ \sigma_{t-1}\varepsilon r$. By our hypothesis of induction, (\ref{y离Lk远，y不属于Es}) holds for $m,...,t-1$, hence $z\notin \ccup_{l=m}^{t-1} E_l$. Therefore, $B(y,\sigma_m \varepsilon r)\subset B(x,r)$ and does not meet $\ccup_{l=0}^{t-1} E_l$. By Corollary \ref{推论：Em中的点满足am(x,r)很小，即使r很大}, $a_{t}(y,2\sigma_m \varepsilon r/(5n_0^n)) <14 C_1\varepsilon$. By the same argument for $k=m$, it is impossible. We have thus proved claim (\ref{y离Lk远，y不属于Es}). Consequently, we get that if $y\in E_k\cap B(x,\lambda_k r)$, then $\dd(y, L^k)<B_1 \varepsilon r$. Thus (\ref{lem4证明中的式子：Ek离k-spine近}) holds for each $m\leq k\leq n-1$.

   For (\ref{lem4:骨架上的点离Ek近}), we first prove for the case when 
    \begin{equation}\label{lem4第一种情况}
        y\in L^{k}\cap B(x,\lambda_k r) \text{ and }\dd(y, L^{k-1})\geq B_2\varepsilon r
    \end{equation}
   for each $m\leq k\leq n-1$. Let $k$ be fixed. Note that $n_0 B(y,{B_2}\varepsilon r/{n_0})$ does not meet $L^{k-1}$, by Proposition \ref{用次一级的锥替代}, there exists $T_k\in \A(k)$ such that
    \begin{equation}
        T_k\cap B(y,{B_2}\varepsilon r/{n_0})=T\cap B(y,{B_2}\varepsilon r/{n_0}).
    \end{equation}
    Obviously, $y\in L^k(T_k)$. Since $y\in T$ and $d_{x,r}(E, T)<14 C_1 \varepsilon$, we can find $e\in E$ such that $|e-y|<14 C_1\varepsilon r$. 
    Let $Z=Z(e,B_2\varepsilon r)$ and $D=600n_0^{2n}N_0^2$, then $D$ is much larger than $14 C_1$ and is smaller than $B_2/n_0$. 
    By Lemma \ref{3个集合的dx,r}, we have $d_{y,D\varepsilon r}(Z,T_k)<(14 C_1+B_2\varepsilon)/D<40^{-1}n_0^{-(n+1)}\delta_0$.  
    Since $T_k\in \A(k)$ and $y\in L^k(T_k)$, we can find $Z_k\in \A(k)$ such that 
    \begin{equation}
        Z_k\cap B(y,60n_0^{n}N_0^2\varepsilon r)=Z\cap B(y,60n_0^{n}N_0^2\varepsilon r)
    \end{equation}
    by Lemma \ref{引理：两个锥dxr距离小，则另一个也是一样的type}. At the same time, we have $\dd(y,\LL^{k}(Z))<30 n_0^n N_0^2\varepsilon r$ by Lemma \ref{type m集合若dxr距离小于delta，则最高级spine距离受控}. Thus $\dd(y, L^k(Z_k))<30 n_0^n N_0^2 \varepsilon r$.
    Then we can find $l\in\LL^{k}(Z_k)$ such that $|y-l|<30n_0^n N_0^2\varepsilon r$ and it is clear that $l\in Z$ and $|l-e|<B_2\varepsilon r$. Therefore, there exists $e_1\in E$ such that $|e_1-l|<B_2\varepsilon^2 r$.  Let $\rho=10n_0^n N_0^2 \varepsilon r$, then we have $B(e_1,\rho)\subset B(e, B_2\varepsilon r)$ since $ |e_1-e|\leq (30n_0^n N_0^2+14C_1+1)\varepsilon r$, and $B(e_1,\rho)\subset B(y,60 n_0^n N_0^2\varepsilon r) $ since $|e_1-y|<(30n_0^n N_0^2+1)\varepsilon r$. Set $Z'=Z_k+(e_1-l)$, then $a_k(e_1,\rho)\leq d_{e_1,\rho}(Z',E)
    <C_3\varepsilon$, where $C_3=100(n+n_0)^{2n}$. By Lemma \ref{Em的m-flatness}, $d_{e_1,0.99\rho}(E_k, \LL^{k}(Z'))<0.99^{-1} C_4\varepsilon$. Thus, $\dd(e_1,E_k)<C_4\varepsilon\rho$ and $\dd(y,E_k)\leq |y-e_1|+\dd(e_1,E_k)<(30n_0^n N_0^2+2)\varepsilon r<B_1\varepsilon r$. That is, (\ref{lem4:骨架上的点离Ek近}) holds for each $y$ in (\ref{lem4第一种情况}).  We have thus proved (\ref{lem4:骨架上的点离Ek近}) for each $m\leq k\leq n-1$ when (\ref{lem4第一种情况}) holds.

    Note that when $k=m$, $L^{m-1}=\emptyset$, thus $\dd(y, L^{m-1})\geq B_2 \varepsilon r$ is valid for every $y\in L^m\cap B(x,\lambda_m r)$. So we end the proof of (\ref{lem4:骨架上的点离Ek近}) and  get that $d_{x,\lambda_m r}(E_m,L^m)<B_1\varepsilon/\lambda_m<C_5\varepsilon$. If $m=n-1$, Proposition \ref{lem4:用于投影的锥} for $n-1$ follows. Thus we only have to suppose $m<n-1$ and $k>m$.

    We will prove for the general case when $y\in L^{k}\cap B(x,\lambda_k r)$ that
    \begin{equation}\label{y到Ek的距离小于B1+(k-m)B2}
        \dd(y,E_k)<(B_1+(k-m)B_2)\varepsilon r
    \end{equation}
    for each $m<k\leq n-1$.
    Let $k$ be fixed. We aim to prove that $\dd(y, E_k)<(B_1+({k-m})B_2)\varepsilon r$ for each $y\in L^k\cap B(x,\lambda_k r)$ and $\dd(y, L^{k-1})<B_2\varepsilon r$, since it already holds for $y\in L^k\cap B(x,\lambda_k r)$ such that $\dd(y,L^{k-1})\geq B_2\varepsilon r$, after the above discussion.  Fix $y$ and let 
    \begin{equation}
        j=\min\{ t:t\geq m\text{ and for each } t\leq i\leq k-1, \dd(y, L^i)<(k-i) B_2\varepsilon r\}.
    \end{equation}
    It is obvious that $j\leq k-1$. When $j=m$, we can find $l_m\in L^m$ such that $|y-l_m|<(k-m)B_2\varepsilon r$. Therefore, $l_m\in L^{m}\cap  B(x,\lambda_m r)$. Furthermore, we can find $e_m\in E_m$ such that $|e_m-l_m|<B_1\varepsilon r$. 
    
    {  We aim to show that $e_m$ is a limit point of $E_k$. Indeed, for any radius $r' \in (0, r/100n)$, the ball $B(e_m, 5 n_0^n r'/2)$ is contained in $B(0, 1.99)$ and is disjoint from $\cup_{i=0}^{m-1} E_i$. Thus, the pair $(e_m, r')$ satisfies the hypotheses of Proposition \ref{lem4:用于投影的锥}. Consequently, the results obtained in the preceding arguments also hold for $(e_m, r')$. In particular, there exists a cone $T':= T(e_m, r') \in \mathcal{A}(m)$ with $e_m \in L^m(T')$. Recall that in the proof of case \eqref{lem4第一种情况}, we showed that points on the $k$-spine bounded away from the $(k-1)$-spine must be close to $E_k$. Applying this to the current scale $r'$, we can choose a point $l \in L^k(T') \cap B(e_m, r')$ such that $\dd(l, L^{k-1}(T')) \geq B_2 \varepsilon r'$. It follows that $\dd(l, E_k) < B_1 \varepsilon r'$. Thus, $ \dd(e_m, E_k) \leq |e_m - l| + \dd(l, E_k) < (1 + B_1 \varepsilon) r'$. 
    Letting $r' \to 0$, we conclude that $e_m \in \overline{E_k}$. Finally, combining this with the previous estimate, we obtain that} $\dd(y, E_k)\leq |y-e_m|<(B_1+({k-m})B_2)\varepsilon r$.

    When $j>m$, we have $\dd (y,L^{j-1})\geq (k-j+1)B_2\varepsilon r$ and $\dd (y,L^{j})<(k-j)B_2\varepsilon r$. Thus, there exists $l\in L^{j}$ such that $|l-y|<(k-j)B_2\varepsilon r$. At the same time, $\dd(l,L^{j-1})>B_2\varepsilon r$. And $|l-x|\leq |l-y|+|y-x|<(k-j) B_2\varepsilon r+\lambda_k r<\lambda_{j}r$, since $j\leq k-1$. Then $l$ satisfies the condition (\ref{lem4第一种情况}), which implies that $\dd(l,E_j)<B_1\varepsilon r$.
    Thus we can find $e_j \in E_j$ such that $|e_{j}-l|<B_1\varepsilon r$. Moreover, we have $|e_{j}-y|<(B_1+(k-j)B_2)\varepsilon r$ and therefore, $e_{j}\in E_{j}\cap B(x,\lambda_{j}r)$. By (\ref{lem4证明中的式子：Ek离k-spine近}), $E_{j-1}$ is contained in the $B_1\varepsilon r$ neighborhood of $L^{j-1}$ in $B(x,\lambda_{j-1}r)$. Thus we get that 
    \begin{equation}
        \begin{aligned}
        \dd(e_{j}, E_{j-1})&\geq \dd(l,E_{j-1})-|l-e_{j}|
        \geq B_2\varepsilon r -2B_1\varepsilon r>0
        \end{aligned}
    \end{equation}
    Set $I=B_2/2-B_1$. Since the spine of lower dimension is contained in the spine of higher dimension, we conclude that $B(e_{j},I\varepsilon r)$ does not meet $\ccup_{i=m}^{j-1} E_i$. At the same time, $B(e_{j},I\varepsilon r)$ does not meet $\cup_{i=0}^{m-1} E_i$ because $B(x,r)\cap (\ccup_{i=0}^{m-1} E_i)=\emptyset$. Thus $B(e_{j},I\varepsilon r) \cap (\cup_{i=0}^{j-1} E_i)=\emptyset$.
    Recall that we have assumed that Proposition \ref{lem4:用于投影的锥} holds for $n,...,m+1$ and $j>m$. Thus we can find $T_{j}\in \A(j)$ such that $e_{j}\in L^{j}(T_{j})$ and $d_{e_{j}, \lambda_k 2I\varepsilon r/(5n_0^n)}(L^{k}(T_{j}), E_k)<C_5\varepsilon$. Since $e_{j}\in L^{k}(T_{j})$, we have $\dd(e_{j}, E_k)<2C_5\lambda_k I \varepsilon^2 r/(5 n_0^n)$. Therefore,
    \begin{equation}
    \begin{aligned}
    \dd(y,E_k)&\leq |y-e_{j}|+\dd(e_{j}, E_k)
    <(B_1+(k-m)B_2)\varepsilon r.
    \end{aligned}
    \end{equation}
   Together with (\ref{lem4第一种情况}), we have (\ref{y到Ek的距离小于B1+(k-m)B2}). This ends the proof of (\ref{lem4:骨架上的点离Ek近}). Thus Proposition \ref{lem4:用于投影的锥} follows.  
\end{proof}

\subsection{Relationship between $E$ and $E_m$}

Now, we will show that $E$ is a disjoint union of $\{E_m\}_{m=0}^n$.

\begin{proposition}\label{E是Ei的不交并}
    \begin{equation}\label{式子：E是Ei的不交并}
        E\cap B(0,1.99)=(\ccup_{m=0}^n E_m) \cap B(0,1.99)\text{ (a disjoint union)}.
    \end{equation}
\end{proposition}


Before proving Proposition \ref{E是Ei的不交并}, we will first show that for every $x\in E\cap B(0,2)$, there exists $0\leq m\leq n$ and $r_k\to 0$ such that $a_m(x,r_k)<C_2$ for each $k$ in Lemma \ref{对每个x，都有任意小的r使得a_m(x,r)小于beta}. We will also show that if in the meanwhile  $a_t(x,r)\geq C_2$ when $r>0$ is small enough and for each $ t\in \{0,...,m-1\}$ (when $m=0$, $\{0,...,m-1\}=\emptyset$), then $x\in E_m$ in Lemma \ref{任意小半径am(x,r)<C2且对小于m的充分小大于C2}.

\begin{lemma}\label{对每个x，都有任意小的r使得a_m(x,r)小于beta}
    For every $x\in E\cap B(0,2)$, we can find $0\leq m\leq n$ such that there exists $r_k\to 0$ when $k\to\infty$ such that $a_m(x,r_k)<C_2$ for each $k$.
\end{lemma}

\begin{proof}
    Assume that we can find $x\in E\cap B(0,2)$ and $\rho>0$ such that for all $r\in(0,\rho)$ and all $0\leq m\leq n$, $a_m(x,r)\geq C_2$. Pick any $r\in( 0, \rho) $. Then we have  $x\in Z(x,r)$ and $d_{x,r}(E,Z(x,r))<\varepsilon$. The problem is that the spine of $Z(x,r)$ might not contain $x$. Suppose $Z(x,r)\in\A(s)$ and consider $Z(x,r)$ in a scale small enough.

    Set $\sigma_k=600 N_0\cdot(400 N_0 n_0)^{n-k}$ for $k=-1,...,n$. Then $\{\sigma_k\}_{k=-1}^n$ is a decreasing sequence with $\sigma_k=400 N_0 n_0\sigma_{k+1}$. We define $L^{-1}(Z(x,r))=\emptyset$ in addition. Denote by $t$ the smallest number satisfying the following:
    \begin{equation}
            L^{t-1}(Z(x,r))\cap B(x,\sigma_{t-1}\varepsilon r)=\emptyset\text{ and }L^{t}(Z(x,r))\cap B(x,\sigma_t \varepsilon r)\neq\emptyset.
    \end{equation}
    Since $x\in Z(x,r)$, we have $L^n(Z(x,r))\cap B(x,\sigma_n \varepsilon r)\neq\emptyset$. And it is evident that $L^{s-1}(Z(x,r))$ is an empty set. Therefore, such $t\geq s$ exists. Since $L^{t-1 }(Z(x,r))\cap B(x,\sigma_{t-1} \varepsilon r)=\emptyset$, by Proposition \ref{用次一级的锥替代}, we can find $W_t\in \A(t)$ such that 
    \begin{equation}
        W_t\cap B(x,\frac{\sigma_{t-1}}{n_0}\varepsilon r )=Z(x,r)\cap B(x,\frac{\sigma_{t-1}}{n_0}\varepsilon r).
    \end{equation}
    At the same time, we can find $l\in L^{t}(Z(x,r))$ such that $|x-l|<\sigma_t\varepsilon r$. Therefore, $l\in L^{t}(W_t)$ because $\sigma_t$ is much smaller than $\sigma_{t-1}/n_0$. Let $W=W_t+(x-l)$, then $W$ is a set in $\A(t)$ with its $t$-spine passing through $x$. By Lemma \ref{3个集合的dx,r}, $ d_{x,300N_0 \sigma_t\varepsilon r}(W,E)<\frac{(\sigma_t+1)\varepsilon r}{300N_0\sigma_t \varepsilon r}<C_2.$  
    This implies that $a_t(x,300 N_0 \sigma_t \varepsilon r)<C_2$, which contradicts our assumption that $a_t(x,\tau )\geq C_2$ for all $\tau\in (0,\rho) $. Therefore, fix $x\in E\cap B(0,1.99)$, for each $\rho>0$ small enough, there exists $r\in (0,\rho)$ and a number $m(r)\in \{0,...,n\}$ such that $a_{m(r)}(x,r)<C_2$, where $m(r)$ is related to $x$ and $r$. Since the set $\{0,...,n\}$ is finite, we can find at least one $m$ satisfying that there are arbitrarily small $r>0$ such that $a_m(x,r)<\beta$, and Lemma \ref{对每个x，都有任意小的r使得a_m(x,r)小于beta} follows.
\end{proof}

\begin{lemma}\label{任意小半径am(x,r)<C2且对小于m的充分小大于C2}
    { 
    
    Assume that $0 \leq m \leq n$ and let $x \in E \cap B(0, 1.99)$. Assume additionally in the case that $m > 0$ that there exists $\rho_x > 0$ such that $a_t(x, r) \geq C_2$ for each $t \in \{0, \dots, m-1\}$ and $0 < r < \rho_x$. Then if there exists a sequence $r_k \to 0$ such that $a_m(x, r_k) < C_2$ for each $k$, then it follows that $x \in E_m$.
    
    }
\end{lemma}

\begin{proof}
    We will prove by induction from $m=0$ to $m=n$. Fix $x\in E\cap B(0,1.99)$, first suppose that there is a decreasing series of radius $\{r_k\}_{k=1}^{\infty}$ such that $\lim_{k\to\infty}r_k= 0$ and $B(x,20 C_1 r_k)\subset B(0,2)$, in addition, $a_0(x,r_k)<C_2$ for all $k$. Then we can check that $x\in E_0$. Actually, for all $\rho>0$ small enough, we can find $l_k\geq 2$ such that $(8N_0)^{l_k-1}r_k\leq \rho<(8N_0)^{l_k}r_k$ for all $k$ big enough. We have $l_k$ tends to $\infty$ when $k$ tends to $\infty$. Then by Lemma \ref{lem:小影响大}, we have $a_0(x,(8N_0)^{l_k}r_k)<2^{-l_k} C_2 + (\sum_{t=0}^{l_k-1} 2^{-t})\cdot 7 C_1 \varepsilon$. Therefore, $a_0(x,\rho)<[2^{-l_k}C_2 +(\sum_{t=0}^{l_k-1} 2^{-t})\cdot 7C_1\varepsilon]\cdot 8N_0$. Let $k$ tends to $\infty$, then we have $a_0(x,\rho)\leq 112n_0^n N_0^2\varepsilon$ for all $\rho$ small enough, thus $x\in E_0$.

    Assume we have proved that Lemma \ref{任意小半径am(x,r)<C2且对小于m的充分小大于C2} holds for $0,...,m-1$, where $m\geq 1$. Let us show it holds for $m$. By hypothesis, we can find a decreasing series $r_k\to 0$ such that $a_m(x,r_k)<C_2$. First we show that $x\notin \ccup_{t=0}^{m-1} E_t$. 
    Otherwise, suppose that $x\in E_t$ for some $0\leq t\leq m-1$. 
    By the definition of $E_t$, we know that when $r>0$ is small enough, $a_t(x,r)<C_0\varepsilon$, where $C_0=10^8(n+n_0)^{3n}N_0$. Pick a $k$ large enough such that $a_t(x,r_k)<C_0\varepsilon$. Thus we can find $W\in \A(t)$, whose $t$-spine $L^{t}(W)$ passes through $x$ and $d_{x,r_k}(E,W)<C_0\varepsilon$. At the same time, $a_m(x,r_k)<C_2$. So we can find $Y\in \A(m)$ such that $x\in L^{m}(Y)$ and $d_{x,r_k}(Y,E)<C_2$. 
    Then $d_{x,r_k/2}(W,Y)<2(C_2 +C_0\varepsilon)<\delta_0$. It leads to a contradiction because $t<m$ and $x\in L^t(W)$. Thus,  $x\notin \ccup_{t=0}^{m-1} E_t$.

    Next we continue to prove that $x$ is not in the closure of $\cup_{t=0}^{m-1} E_t$. Suppose not, set $h=\min\{t:0\leq t\leq m-1$, there are infinitely many points in $E_t$ that converge to $x\}$. It is obvious that $h>0$ because $x\in B(0,1.99)$ and there is at most one point in $E_0\cap B(0,1.99)$. Choose $\{y_j\}_{j=1}^{\infty}$ in $E_h$ such that $\lim_{j\to\infty}y_j=x$.  By the minimality of $h$, there is $r>0$  such that $B(x,5n_0^n r)\subset B(0,2)$ and $B(x,5n_0^n r)\cap (\ccup_{t=0}^{h-1} E_t)=\emptyset$. And there is $r_k<r$ such that $a_m(x,r_k)<C_2$. So we can find $X\in \A(m)$ such that $x\in L^{m}(X)$ and $d_{x,r_k}(E,X)<C_2$. Fix $r$ and $r_k$. 
    Then pick $y_j$  such that $|y_j-x|<r_k/100$. Therefore, $B(y_j,\frac{5}{2}n_0^n r_k)\cap (\cup_{t=0}^{h-1} E_t)=\emptyset$.  
    Since $y\in E_h$, we have $a_h(y_j,r_k)<14C_1\varepsilon$ by Corollary \ref{推论：Em中的点满足am(x,r)很小，即使r很大}. So we can find $Y\in \A(h)$ such that $y_j\in L^{h}(Y)$ and $d_{y_t,r_k}(Y,E)<14C_1\varepsilon$. Then $d_{y_j,r_k/2}(Y,X)<\delta_0$. Since $h<m$ and $y_j\in L^{h}(Y)$, it is impossible. So the assumption is not true. Hence, $x\notin\overline{\cup_{t=0}^{m-1} E_t}$. 

    By the discussion above, we can find $r'>0$ such that $B(x,r')\cap (\cup_{t=0}^{m-1} E_t)=\emptyset$. At the same time, there are arbitrarily small $r>0$ such that $a_m(x,r)<C_2$. Using the same argument as for the case for $0$, we can show that $x\in E_m$.
\end{proof}

\begin{proof}[\rm \textbf{Proof of Proposition \ref{E是Ei的不交并}.}] 

{  
This is a direct corollary of Lemma \ref{对每个x，都有任意小的r使得a_m(x,r)小于beta} and Lemma \ref{任意小半径am(x,r)<C2且对小于m的充分小大于C2}. By Lemma \ref{对每个x，都有任意小的r使得a_m(x,r)小于beta}, for any $x \in E \cap B(0, 1.99)$, there exists at least one dimension $m$ and a sequence of radius $r_k \to 0$ such that $a_m(x, r_k) < C_2$ for all $k$.  Let $m_0$ be the smallest such dimension.  If $m_0=0$, then $x\in E_0$ by  Lemma \ref{任意小半径am(x,r)<C2且对小于m的充分小大于C2}. If $m_0>0$, then for all  $t<m_0$, we must have $a_t(x, r)\geq C_2$ for all sufficiently small $r>0$ (otherwise $m_0$ would not be the smallest). Lemma \ref{任意小半径am(x,r)<C2且对小于m的充分小大于C2} then implies that $x \in E_{m_0}$.

Thus, every point in $E \cap B(0, 1.99)$ belongs to some $E_m$. Since $\{E_m\}_{m=0}^n$ are disjoint by Lemma \ref{Em不交}, Proposition \ref{E是Ei的不交并} follows.

}
\end{proof}

At last, we show that $E_m$ is contained in the closure of $E_{m+1}$ for each $0\leq m<n$.

\begin{proposition}\label{每一层Em的关系}
    $\ccup_{t=0}^m E_t$ is closed in $\overline{B(0,1.989)}$ for each $0\leq m\leq n$. 
    And in $\overline{B(0,1.988)}$,  $E_{m}\subset \overline{E_{m+1}}$ for all $0\leq m< n$.
\end{proposition}

\begin{proof}
    We prove by induction on $m$. When $m=0$, $E_0\cap B(0,1.99)$ has at most one point. So it is closed. Now assume that $\ccup_{t=0}^{m-1} E_t$ is closed in $\overline{B(0,1.989)}$, let us show that $\ccup_{t=0}^{m} E_t$ is closed. Suppose not, there exists $x\in \overline{\ccup_{t=0}^{m} E_t}\backslash \ccup_{t=0}^m E_t$ and $x$ is contained in $\overline{B(0,1.989)}$. Then $x\in \ccup_{t=m+1}^{n}E_t$. Suppose that $x\in E_k$ for some $m+1\leq k\leq n$. According to the  induction hypothesis, there is $r>0$ small enough such that $B(x,r)\cap (\ccup_{t=0}^{m-1} E_t)=\emptyset.$ Since $x$ is contained in the closure of $\ccup_{t=0}^m E_t$, we can obtain that $x$ is a limit point of $E_m$. Fix $r>0$ such that  $B(x,5n_0^n r)\subset B(0,1.99)$ and $B(x,5n_0^n r)\cap (\ccup_{t=0}^{m-1} E_t)=\emptyset$, in addition, $r$ is sufficiently small so that $a_k(x,r)<C_0\varepsilon$. Since $x$ is a limit point of $E_m$, we can find $y\in E_m$  such that $|y-x|<r/100$. Thus $B(y,5n_0^n r/2)\subset B(x,5 n_0^n r)$. Then $B(y,5n_0^n r/2)\subset B(0,1.99)$ and does not meet $\cup_{t=0}^{m-1} E_t$. By Proposition \ref{lem4:用于投影的锥}, we can find $W_m\in \A(m)$ such that $y\in L^{m}(W_m)$ and $d_{y,r}(W_m,E)<14 C_1\varepsilon$. At the same time, we have $a_k(x,r)<C_0\varepsilon$. Thus there is $W_k\in \A(k)$ such that $d_{x,r}(E,W_k)<C\varepsilon$. Then $d_{y,r/10}(W_m,W_k)<140 C_1\varepsilon+10 C_0\varepsilon$. It contradicts the fact that $W_m$ is a set in $\A(m)$ whose $m$-spine passes through $y$ and $k>m$.  Therefore, the assumption is wrong and we have proved that $\ccup_{t=0}^m E_t$ is closed, and it is valid for all $m$.

    Next we prove that $E_{m}\subset \overline{E_{m+1}}$. By Proposition \ref{E是Ei的不交并}, $E_m\cap (\cup_{t=0}^{m-1} E_t )=\emptyset$. Since  $\cup_{t=0}^{m-1} E_t$ is closed, then for all $x\in E_m\cap \overline{B(0,1.988)}$, we can find $r>0$ such that $B(x,5 n_0^n r/2)\subset B(0,1.99)$ and does not meet $\cup_{t=0}^{m-1} E_t$. Here $r$ is related to $x$. Let $\{r_k\}_{k=1}^{\infty}$ be a series of radius smaller than $r$ and tend to 0. By Proposition \ref{lem4:用于投影的锥}, we can find $T_k\in \A(m)$ such that $x\in L^{m}(T_k)$ and $d_{x,0.99r_k}(E_{m+1},L^{m+1}(T_k))<2C_5\varepsilon$ for each $k$. Thus, there exists $y_k\in E_{m+1}$ such that $|y_k-x|<0.99\cdot 2C_5\varepsilon r_k$. Thus $y_k\in E_{m+1}$ goes to $x$ and $x\in \overline{E_{m+1}}$.
\end{proof}

\section{Covers, partitions of unity and some estimates}


{ 
In this section, we prepare the tools to build the parameterization. We first show that it is enough to prove Theorem \ref{主定理1} for a specific standard case. Then, in Section 4.1, we construct the covers and partitions of unity. In Section 4.2, we study the relationship between the approximating cones of nearby balls.

It suffices to prove Theorem \ref{主定理1} under the specific assumption that
\begin{equation}\label{要证的情况}
    E_0\cap B(0,1.99)=\{0\}, Z(0,2)\in \A(0) \text{ and is centered at 0.}
\end{equation}
In this case, for all $0\leq m\leq n$, using the same method as in Proposition \ref{lem4:用于投影的锥}, we have
\begin{equation}\label{1.98尺度}
    d_{0,1.98}(E_m, L^m(Z(0,2)))<C_6\varepsilon,
\end{equation}
where $C_6$ depends on $n_0,\delta_0,n$. Note that the case where $Z(0,2)$ is of type 0 is the most complex one, since it has the most singular geometry. Proving the theorem for this case gives us the foundation to handle all other cases.

It is necessary to remark that the hypothesis of Proposition \ref{lem4:用于投影的锥} fails for the case $x=0$ and $r=2$. Thus, Proposition \ref{lem4:用于投影的锥} cannot be used directly to derive  \eqref{1.98尺度}. However, condition \eqref{要证的情况} guarantees that $a_0(0, 2) \le d_{0,2}(E, Z(0, 2)) < \varepsilon$. By substituting this estimate directly into the proof of Proposition \ref{lem4:用于投影的锥}, we can verify that the entire argument remains valid in this context. Consequently, we obtain  \eqref{1.98尺度}.

Before discussing the general case, we comment on where the results in Section 3 apply. Although the sets $E_m$ and their properties were defined within $B(0, 2)$ for simplicity, the arguments depend only on the local approximation hypothesis. Therefore, these results extend naturally to any larger ball $B(0, T_0)$. In the discussion below, we apply the results of Section 3 within a sufficiently large ball $B(0, T_0)$ (e.g., $T_0 = 10 \cdot (10 n_0^n)^{n+1}$) to ensure we include all relevant geometric features. 

To motivate the classification below, consider the case where the approximating cone $Z(0, 2)$ is of type 0. If its center lies just outside $B(0, 2)$, it still dictates the internal geometry. Since we build the parameterization from the lowest-dimensional spine, we must anchor the construction from this center. This forces us to work on a larger domain to include such points.  Therefore, we need to classify the cases based on the sets $E_m$ within a larger ball.

Let $\{ R_m \}_{m=0}^n$ be a decreasing sequence of radii defined by $R_n = 2$ and the condition  $R_m = 10 n_0^n R_{m+1}$ for $m \in \{0, \dots, n-1\}$. That is, $R_m = 2 \cdot (10 n_0^n)^{n-m}$. Note that we can choose $T_0 > 10 n_0^n R_0$ to ensure enough space for the construction.

We proceed by identifying the effective starting dimension. Since $0 \in E$, the set $E \cap B(0, R_0)$ is non-empty. Let $m$ be the smallest integer in $\{0, \dots, n\}$ such that
\begin{equation}
    E_m \cap B(0, R_m) \neq \emptyset.
\end{equation}
If $m = n$, then  $E$ does not meet $\cup_{i=0}^{n-1} E_i$ in $B(0,R_{n-1})$, reducing the problem to the standard Reifenberg flat situation.

Now suppose $m < n$. Let $x_m$ be a point in $E_m \cap B(0, R_m)$. Let us verify that the hypotheses of Proposition \ref{lem4:用于投影的锥} are satisfied by the pair $(x_m, 2R_m)$. Observe that we have
\begin{equation}
\frac{5n_0^n}{2} B(x_m, 2R_m) \subset B(0, 5 n_0^n R_m + |x_m|) \subset B(0, (5n_0^n + 1) R_m) \subset B(0, R_{m-1}).
\end{equation}
(For $m=0$, the inclusion holds trivially within $B(0, T_0)$). Since $m$ is the smallest dimension, we know that $(\cup_{i=0}^{m-1} E_i) \cap B(0, R_{m-1}) = \emptyset$. Therefore, the conditions of Proposition \ref{lem4:用于投影的锥} are satisfied.

Consequently, Proposition \ref{lem4:用于投影的锥} yields a set $Z = T(x_m, 2R_m)$ of type $m$ satisfying the properties 1) and 2). Since the distance between the origin and $Z$ is small (bounded by $2 R_m C_5 \varepsilon$), we may assume that $0 \in Z$ (by adjusting $\varepsilon$ slightly). Note that in this context, we have an estimate similar to \eqref{1.98尺度} within $B(x_m, 1.98 R_m)$, which is sufficient for the construction.

We then carry out the construction of the parameterization within the ball $B(x_m, 1.98 R_m)$. In this process, we define the initial maps based on the dimension $m$.

If $m=0$, we define the first step as a global translation that maps the center $z_0$ of $Z$ to $x_m$.  By Proposition \ref{lem4:用于投影的锥} and the assumption $0\in Z$, this translation moves points no more than $4 C_5  R_m \varepsilon$. Since this translation is small and $E_0$ contains only one point in this region, it is easy to verify that this simple translation satisfies all the necessary geometric constraints (specifically, conditions (M1)-(M4) in Section 5). The rest of the construction then proceeds exactly as in the standard case \eqref{要证的情况}. 

If $m>0$, we set $f^0 = \dots = f^{m-1} = id$ (effectively treating $\cup_{i=0}^{m-1} E_i$ as the empty set) and note that the construction does not require the spine $L^m(Z)$ to pass through $x_m$. Finally, since $B(x_m, 1.98 R_m)$ completely contains  $B(0, 2)$, restricting the resulting map $f$ to $B(0, 2)$ completes the proof of the main theorem.

\ 

Throughout the rest of the paper, we say that a constant $C$ is a geometric constant if it depends only on $n, n_0, \delta_0$ and $\alpha$. 
}

\subsection{Covers and  partitions of unity}

Fix $k\geq 0$, we  will construct a cover of $E$. Let $E_{n+1}=\RR^N\backslash E$. 
First we cover $E_0\cap B(0,1.98)=\{0\}$ by $B(0,n_0^{-n}\cdot 2^{-k-10^2})$. Let $\{0\}=\{x_{i_0}\}_{i_0\in I_0(k)}$  and $r_{i_0}=n_0^{-n}\cdot 2^{-k-10^2}$. Let $B_{i_0}=B(x_{i_0},r_{i_0})$.

Suppose we have constructed the cover of $\cup_{t=0}^{m-1} E_t$, where $m$ is a number in $\{1,...,n+1\}$. 
{  Specifically, assume that we have defined balls $\{B_{i_t}\}_{{i_t} \in I_t(k)}$ for all $t\in \{0,\dots,m-1\}$, such that the union of these balls (suitably enlarged) covers $\cup_{t=0}^{m-1} E_t$ within the relevant domain.
}
Then we continue to construct a cover of $E_m$. Set
\begin{equation}
    E_m'=E_m\cap B(0,1.98-\frac{m}{1000n})\backslash \bigcup_{t=0}^{m-1} \bigcup_{i_t\in I_t(k)}(2-\frac{1}{2^{m-t+1}})B_{i_t}.
\end{equation}
Then pick a  maximal subset $\{x_{i_m}\}_{i_m\in I_m(k)}$ of $E_m'$ such that 
\begin{equation}\label{构造覆盖时点分离}
    |x_i-x_j|\geq n_0^{-n(m+1)}\cdot2^{-k-10^{m+2}}
\end{equation}
for every $i\neq j$ in $I_m(k)$. Let $r_{i_m}=n_0^{-n(m+1)}\cdot2^{-k-10^{m+2}}$ and $B_{i_m}=B(x_{i_m},r_{i_m})$. Then we have
\begin{equation}\label{2Bi构成Em覆盖}
    E_m\cap B(0,1.98-\frac{m}{1000n})\subset \left(\bigcup_{t=0}^{m-1}\ \bigcup_{i_t\in I_t(k)}(2-\frac{1}{2^{m-t+1}})B_{i_t} \right)\bigcup \left(\bigcup_{i_m\in I_m(k)}B_{i_m} \right).
\end{equation}
Let $I(k)=\ccup_{m=0}^{n+1} I_m(k)$ and $I_m=\ccup_{k=0}^{\infty}I_m(k)$. Then we have 
\begin{equation}
    E\cap B(0,1.97)\subset 
    \bigcup_{t=0}^n\bigcup_{i\in I_t(k)} 2B_i\text{ and }B(0,1.97)\subset
      \bigcup_{i\in I(k)} 2B_i.
\end{equation}
By the definition of the cover, we know the balls 
centered at points of greater types have a 
{  large} distance from points of smaller types. That is, given $i\in I_m(k)$ for some $0 < m\leq n+1$ and $j\in I_{t}(k)$, where $t<m$, then we have 
\begin{equation}
    \dd(10^3 n_0^n B_i, (2-\frac{1}{2^{m-t}})B_j)\gg  10^4 r_i.
\end{equation}
Since $\ccup_{t=0}^{m-1} E_t$ is contained in the union  $\ccup_{t=0}^{m-1}\cup_{i_t\in I_t(k)}(2-2^{t-m})B_{i_t}$, we have 
\begin{equation}\label{级别低的球离级别高的点远}
    \dd (10^3n_0^n B_i,\bigcup_{t=0}^{m-1} E_t)\gg 10^4 r_i.
\end{equation}
Now we continue to define  the partitions of unity. Fix $k\geq 0$, for any $i\in I(k)$, let $\tilde{\theta}_i$ be a bump  function such that $\tilde{\theta}_i=1$ on $2B_i$ and $\tilde{\theta}_i=0$ outside $3B_i$. In $3B_i\backslash 2B_i$, the value of $\tilde{\theta}_i$ is between $0$ and $1$. 
We also ask that {   $|\nabla ^p \tilde{\theta}_i|<\tilde{C}(p) 2^{pk}\text{ }\text{for each integer } p\geq 1$, where $\tilde{C}(p)$ is a geometric  constant.} 
{ Since, for each $m$, the points in $\{x_{i_m}\}_{i_m \in I_m(k)}$ are well-separated (as in \eqref{构造覆盖时点分离}) and the set $E\cap B(0,2)$ is bounded, the index set $I(k)$ is finite.}  For each $x\in B(0,1.97)$, by (\ref{2Bi构成Em覆盖}), $\sum_{j\in I(k)}\tilde{\theta}_j(x)\geq 1$. Moreover note that $ \sum_{j\in I(k)}\tilde{\theta}_j(x)\leq C<\infty$, where $C$ depends only on $N$, because the choice of balls ensures that they only overlap a bounded number of times.
Thus we set 
\begin{equation}
    \theta_i(x)=\frac{\tilde{\theta}_i(x)}{\sum_{j\in I(k)}\tilde{\theta}_j(x)}\text{ for }x\in B(0,1.97)
\end{equation}
and get that {  
\begin{equation}\label{单位分解的导数}
    |\nabla^p \theta_i|<C(p) 2^{p k},
\end{equation}
where  $C(p),p\geq 1$ are  geometric constants.} By (\ref{级别低的球离级别高的点远}),  if $x\in E$ is such that $\dd (x,E_m)<r_{i}/2$ for some $i\in I_{m}(k)$, then  $\theta_{j}(x)=0$ for all $j\in \ccup_{t=m+1}^{n+1} I_{t}(k)$.

\subsection{Similarity of cones corresponding to close balls}

{  Let $k \ge 0$ be the step index.} For each $i\in I(k)$, we have $500 n_0^n B_i\subset B(0,1.99)$. For each $m>0$ and each $i\in I_m(k)$, we can get that $500 n_0^n B_i$ does not meet $\ccup_{t=0}^{m-1} E_t$ by (\ref{级别低的球离级别高的点远}). By Proposition \ref{lem4:用于投影的锥}, there exists $W_i\in\A(m)$ such that 
\begin{equation}
    d_{x_i,200r_i}(E,W_i)<14 C_1\varepsilon.
\end{equation}
In the meanwhile, for each $m\leq t\leq n$, let $L^t_i$ be the $t\mbox{-}$spine of $W_i$, then $x_i\in L^m_i$ and
\begin{equation}\label{取出来的锥每一层都接近}
    d_{x_i,100r_i}(E_t, L^t_i)<C_5\varepsilon.
\end{equation}



{  For the spine $L^{t}_i$, we consider its decomposition into branches $L^{t}_i = \cup_{l} L^{t,l}_i$.

First, we define the projections associated with each branch. For each $l$, let $P^{t,l}_i$ denote the $t$-plane containing the branch $L^{t,l}_i$. We denote by $\overline{\pi}^{t,l}_i$ the orthogonal projection onto the plane $P^{t,l}_i$, and by $\pi^{t,l}_i$ the nearest point projection onto the branch $L^{t,l}_i$ itself (which is well-defined since each branch is convex).

Next, we consider the specific case where $L_{i}^{t}$ consists of a single branch. In this case, we simplify the notation by dropping the index $l$. We denote by $P^{t}_i$ the unique $t$-plane containing $L^{t}_i$, by $\overline{\pi}^{t}_i$ the orthogonal projection onto $P^{t}_i$, and by $\pi_{i}^{t}$ the nearest point projection onto $L_{i}^{t}$.

Note that we do not define a global orthogonal projection onto $L_{i}^{t}$ when it has multiple branches, as the set is not convex.
}

Denote by $b^t_i$ the number of branches of $L^t_i$. 
Let 
\begin{equation}
    \text{$Z_0=Z(0,2)$ and $L^m=L^m(Z_0)$}.
\end{equation}

First we show some similarities of the cones when they are close. 

\begin{lemma}\label{面的dx,r决定夹角和距离}
   Let $d>0$ be an integer and $C>0$. Let $x\in\RR^N$ and radius $r>0$. If two planes $P_1,P_2$ of dimension $d$ in $\RR^N$ are such that $\dd(x,P_1)\leq r/2$ and $d_{x,r}(P_1,P_2)<C\varepsilon$, {  then} we can find a geometric constant $C(d)$ depending only on $d$ such that 
   \begin{equation}
       |D\pi_1-D\pi_2|<C(d)C\varepsilon,\text{ } |\pi_1(y)-\pi_2(y)|<C(d)C\varepsilon(|y-x|+r) \text{ for }y\in \RR^N.
   \end{equation} 
   {  Here, $\pi_{1}$ and $\pi_{2}$ denote the orthogonal projections onto the  $d$-planes $P_{1}$ and $P_{2}$. The operators $D\pi_{1}$ and $D\pi_{2}$ denote the orthogonal projections onto the linear subspaces parallel to $P_{1}$ and $P_{2}$, respectively. Note that for $i=1, 2$, since $\pi_i$ is an affine map, its derivative $D\pi_i$ is a constant linear map, and is precisely the orthogonal projection onto the linear subspace parallel to $P_i$.}
\end{lemma}

\begin{proof}
     Pick $z\in P_1$ such that $|z-x|\leq r/2$. Then we can find $z'\in P_2$ such that $|z-z'|<C\varepsilon r$. 
     Let $\{e_i\}_{i=1}^d$ be an  orthonormal basis of $P_1-z$. For each $i$,  set $p_i=z+\frac{r}{2}e_i$. Then $p_i\in P_1\cap B(x,r)$. And we can find $q_i\in P_2$ such that $|p_i-q_i|<C\varepsilon r$. Denote by $f_i=2(q_i-z')/r$. Then $|f_i-e_i|<4C\varepsilon$ for each $i$ and $\{f_i\}_{i=1}^d\subset P_2$ is almost an  orthonormal basis. That is, for each $i$, $ 1-4C\varepsilon<|f_i|<1+4C\varepsilon$ and for each $i\neq j$, the inner product $|f_i\cdot f_j|<9C\varepsilon$. 
     Since $D\pi_1$ is the orthogonal projection to the subspace of $\RR^N$ that is parallel to $P_1$, we can show that $|D\pi_1-D\pi_2|<C(d) C\varepsilon$ as a result of that $f_i$ is close to $e_i$,  where $C(d)$ is a geometric constant depending only on $d$. 
     For every $y\in \RR^N$, let $\gamma(t)=(1-t)z+ty,t\in[0,1]$. Then we have
     \begin{equation}
         \begin{aligned}
             |\pi_1(y)-\pi_2(y)|&\leq |\pi_1(z)-\pi_2(z)|+\int_0^1 |D\pi_1-D\pi_2|\cdot |\gamma'|dt\\
             &\leq |z-z'|+C(d)C\varepsilon|y-z|
             <C(d)C\varepsilon (|y-x|+r).
         \end{aligned}
     \end{equation}
 \end{proof}

\begin{lemma}\label{lem:球近则对应的锥接近}
    Let $i, j \in I(k) \cup I(k+1)$. Suppose that $i \in I_s$, $j \in I_t$ and $0 \leq t \leq s \leq n$. If $10B_i \cap 10B_j \neq \emptyset$, {  then for each $d \in \{s, \dots, n\}$,  we have:
    \begin{equation}\label{球近则锥近的式子1}
         d_{x_i,21r_i}(L^d_i,L^d_j) < 10 C_5\varepsilon.
    \end{equation}
    Moreover, let $\B_i^d$ (resp. $\B_j^d$) denote the set of branches of $L_i^d$ (resp. $L_j^d$) intersecting $20B_i$. Then $\# \B_i^d = \# \B_j^d$, and there exists a bijection $\Phi_d: \B_i^d \to \B_j^d$ such that for every branch $L \in \B_i^d$, the corresponding branch $\Phi_d(L) \in \B_j^d$ satisfies:
    \begin{equation}\label{球近则锥近的式子2}
        d_{x_i,20r_i}(L,\Phi_d(L)) < C_7\varepsilon,
    \end{equation}
    where $C_7=2^n 60 n C_5$. }
\end{lemma}

\begin{proof}
    Since $10B_i\cap 10B_j\neq\emptyset$ and $t\leq s$, we have $21B_i\subset 21B_i\cap 72B_j$. By (\ref{取出来的锥每一层都接近}), we have $d_{x_i,21r_i}(L^d_i,L^d_j)<10 C_5\varepsilon$. Hence, (\ref{球近则锥近的式子1}) follows. 
    
    For (\ref{球近则锥近的式子2}), we first aim to show that $W_j$ coincides with a set of type $s$ in $22B_i$. That is, there exists $Z\in \A(s)$ such that 
    \begin{equation}\label{W(j)在22B_i中和type s重合}
    W_j\cap 22B_i=Z\cap 22B_i
    \end{equation}
    If $t=s$, then $W_j\in \A(s)$, so this identity is immediate. 
    If $t<s$, condition (\ref{级别低的球离级别高的点远}) holds, which implies that $10^3 n_0 B_i$ is far away from $E_{s-1}$. At the same time, since $d_{x_j,100r_j}(E_{s-1},L^{s-1}_j)<C_5\varepsilon$ and $10^3 n_0 B_i\subset 100 B_j$, it follows that $10^3 n_0 B_i\cap L^{s-1}_j=\emptyset$. Since $x_i\in E_s\cap 100B_j$, we have $\dd(x_i,L^{s}_j)<100 C_5\varepsilon r_j$. Hence, {   by Lemma \ref{引理：两个锥dxr距离小，则另一个也是一样的type}}, there exists $Z\in \A(s)$ such that $Z\cap 22B_i=W_j\cap 22B_i$, while $\dd(x_i, L^s(Z))<100 C_5\varepsilon r_j$. {Thus,} the identity (\ref{W(j)在22B_i中和type s重合}) follows. 

     Now we prove the existence of the bijection $\Phi_d$ and the estimate (\ref{球近则锥近的式子2}) by induction on the dimension $d$ from $s$ to $n$. Specifically, we aim to show that for each $d \in \{s, \dots, n\}$, there exists a bijection $\Phi_d: \B_i^d \to \B_j^d$ such that for any branch $L \in \B_i^d$,
    \begin{equation}\label{用kd归纳假设的式子}
    d_{x_i, (21-\frac{d-s}{100n})r_i}(L, \Phi_d(L)) < K_d \varepsilon.
    \end{equation}
    Here, $\{K_d\}_{d=s}^n$ is a strictly increasing sequence of constants defined as follows:
    \begin{equation}
        K_s = 10 C_5, \quad K_{s+1} = (10+42(s+1))C_5,
    \end{equation}
    and for $d > s+1$, we set $K_d = 2 K_{d-1}$. With these definitions, we have $\frac{21}{19}K_d \le C_7$ for all $d \le n$. 

    When $d=s$, $L^s_i$ is an $s$-plane and $x_i\in L^s_i$. Since $Z\in \A(s)$,  $L^s (Z)$ is also an $s$-plane. By (\ref{W(j)在22B_i中和type s重合}), we have $L^s(Z)\cap 22B_i=L^s_j\cap 22B_i$. This implies that $L^s_j$ has only one branch intersecting with $22B_i$, denoted as $L^{s,l}_j$. Consequently, both $\B_i^s$ and $\B_j^s$ are singletons.  We define the bijection $\Phi_s$ by $\Phi_s(L^s_i)=L^{s,l}_j$. Furthermore, we have $d_{x_i,21r_i}(L^s_i,L^{s,l}_j)=d_{x_i,21r_i}( L^s_i,L^s_j )<10 C_5 \varepsilon$ based on (\ref{球近则锥近的式子1}). Thus, the inductive hypothesis holds for $d=s$.
    
    Assume the inductive hypothesis holds for dimension $d-1$. That is, there exists a bijection $\Phi_{d-1}$ satisfying the distance estimate (\ref{用kd归纳假设的式子}) with constant $K_{d-1}$. We now construct $\Phi_d$ and prove the estimate for dimension $d$. The proof proceeds in three steps: first, we define the map $\Phi_d: \B_i^d \to \B_j^d$; second, we prove the distance estimate (\ref{用kd归纳假设的式子}); finally, we verify that $\Phi_d$ is a bijection by showing it is both injective and surjective.

    \textbf{Step 1:} We define a map $\Phi_d$ from $\B_i^d$ to $\B_j^d$. Fix a branch $L \in \B_i^d$. Since the family of cones $\B$ (and thus $\A(\B)$) is finite modulo isometry, there exists a geometric constant $c(\B,d) \in (0,1)$, depending only on $\B$, such that for any branch $L$ of dimension $d$, we can find a point $y \in L$ with $|y-x_i|=10r_i$ and 
   $\dd(y, L_i^{d-1}) \ge c(\B,d) \cdot r_i$. Let $\rho = c(\B,d) \cdot r_i$. Then the ball $B(y,\rho)\cap L_i^{d-1}=\emptyset$. Consequently, within $B(y, \rho/n_0)$, the spine $L_i^d$ coincides with the single branch $L$.

    By the inductive hypothesis, the Hausdorff distance between $L_i^{d-1}$ and $L_j^{d-1}$ is small (less than $K_{d-1}\varepsilon$). Since $\varepsilon$ is sufficiently small, the ball $B(y, \rho/2n_0)$ is also disjoint from $L_j^{d-1}$.  Specifically, $L^d_j$ intersects $B(y,\rho/{(2n_0)})$ in exactly one branch, denoted as $L'$. Furthermore, we get that $d_{y,\rho/ (3n_0)}(L,L')<\frac{630 C_5 n_0 }{c(\B,d)}\varepsilon$ by (\ref{球近则锥近的式子1}). We define the mapping $\Phi_d(L) = L'$.

{ 
    \textbf{Step 2:} Next, we prove the distance estimate (\ref{用kd归纳假设的式子}) for $\Phi_d$. If $d=s+1$, we can immediately obtain $d_{x_i,(21-\frac{1}{100n})r_i}(L, \Phi_d(L))<(10+42d) C_5\varepsilon=K_{s+1}\varepsilon$ by Lemma \ref{面的dx,r决定夹角和距离}. 

    Now suppose that $d>s+1$. For each $(d-1)$-boundary $S\subset L$,  there exists a unique element $S'=\Phi_{d-1}(S)\in\B_j^{d-1}$ such that (\ref{用kd归纳假设的式子}) holds (by the inductive hypothesis). Let us show that $S'$ is also a boundary of $L'$. 
     
    We first show that a large part of $L$ is close to $L'$. Define a subset $U$ of $L$ by 
    \begin{equation}
        U=\{x\in L : \dd(x, \partial L)> M \varepsilon r_i /\sin\alpha\}\cap B(x_i, (21-\frac{d-s}{100n})r_i),
    \end{equation}
    where $M= 220(n_0+1) C_5$. Then $U$ is a connected set.  Recall that $\alpha = \alpha(\B)$ is the minimum angle between different branches of a cone in $\mathcal{A}$, see Definition \ref{角的定义} and \eqref{alpha定义}. By Lemma \ref{分支分离性}, for all $x\in U$, $\dd(x, L^{d-1}_i)>M\varepsilon r_i$.
    Let $A$ be the subset of $U$ defined by
    \begin{equation}
        A=\{x\in U: \dd(x,L')<210C_5\varepsilon r_i\}.
    \end{equation}

    We claim that $A=U$. Since $y \in A$ (from Step 1), $A$ is non-empty. Also, $A$ is open in $U$ by definition. We now show that $A$ is closed in $U$ by contradiction. Suppose $x\in U$ is a limit point of $A$ but  $x\not\in A$. By continuity, $\dd(x, L')=210 C_5\varepsilon r_i$. On the other hand, by \eqref{球近则锥近的式子1}, we can find another branch $L''$ of $L^d_j$ such that $\dd(x,L'')<210 C_5 \varepsilon r_i$.  This implies that the ball $B(x, 220 C_5\varepsilon r_i)$ meets both $L'$ and $L''$. Consequently, $L^d_j$ does not coincide with a single $d$-plane in $B(x, 220 C_5\varepsilon r_i)$, which implies that $n_0 B(x,220 C_5\varepsilon r_i)  \cap L^{d-1}_j\neq \emptyset$ by Lemma \ref{用次一级的锥替代}. Specifically, $\dd(x, L^{d-1}_j)<220 C_5\varepsilon r_i$. 
    
    However, since $x\in U$, we have $\dd(x, L^{d-1}_i)>M\varepsilon r_i$. By \eqref{球近则锥近的式子1} and the fact that $x\in U\subset B(x_i,(21-\frac{d-s}{100n})r_i)$, we have
    \begin{equation}
        \dd(x, L^{d-1}_j)=\dd(x,L^{d-1}_j\cap 21B_i )\geq \dd(x,L^{d-1}_i)-210 C_5\varepsilon r_i >(M-210 C_5)\varepsilon r_i> 220 n_0 C_5\varepsilon r_i.
    \end{equation}
     This leads to a contradiction. Thus, our assumption that $A$ is not closed in $U$ is false. Since $U$ is connected, and $A$ is a non-empty subset that is both open and closed, it follows that $A=U$.

    Finally, we address the boundary correspondence. Pick a point $z\in S$ such that $|z-x_i|=10 r_i$ and $ \dd(z, L_i^{d-2})>c(\B,d-1)r_i$, analogous to the choice in Step 1. Consider the ball $B(z, \rho')$, where $\rho' = \frac{c(\B,d-1)}{n_0}r_i$. Inside this ball, $L^{d-1}_i$ coincides with a $(d-1)$-plane, and  $L$ coincides with a $d$-dimensional half-plane. Thus, $\dd(z, U )=M\varepsilon r_i$, and it follows from $A=U$ that $\dd(z, L')<(M+210 C_5)\varepsilon r_i$. 
    
    Recall that $L^{d-2}_i$ and $L^{d-2}_j$ are contained in the $210 C_5\varepsilon r_i$-neighborhood of each other by \eqref{球近则锥近的式子1}. Since $S'=\Phi_{d-1}(S)$, there exists a point $z'\in S'$ such that $|z-z'|<(21-\frac{d-s}{100n})K_{d-1}\varepsilon r_i$. Furthermore, $n_0 B(z',\rho'/2)\cap L^{d-2}_j=\emptyset$, which implies that $W_j$ coincides with a set of type $d-1$ in $ B(z',\rho'/2)$. Therefore, $S'$ is a $(d-1)$-boundary of every branch of $L^d_j$ that meets $ B(z',\rho'/2)$. Given that $\dd(z, L')<(M+210 C_5)\varepsilon r_i$ and $ |z-z'|<(21-\frac{d-s}{100n})K_{d-1}\varepsilon r_i$, we conclude that $L'$ meets $ B(z',\rho'/2)$. Thus, $S'$ is a boundary of $L'$.
}
    
   Hence, there is a one-to-one correspondence between the $(d-1)$-boundaries of $L$ and $L'$, satisfying (\ref{球近则锥近的式子2}). Therefore, we have
   \begin{equation}
   \begin{aligned}
       d_{x_i,(21-\frac{d-s}{100n})r_i}(L,\Phi_d(L))\leq &2\max
       \{ d_{x_i,(21-\frac{d-1-s}{100n})r_i}(S,\Phi_{d-1}(S)): 
         S\in \B_i^{d-1}\text{ and }S\subset L\}.
    \end{aligned}
   \end{equation}
   Thus $d_{x_i,(21-{d-s})r_i\cdot({100n)^{-1}}}(L,\Phi_d(L))<2K_{d-1}\varepsilon$. 

   \textbf{Step 3:} Finally, we verify that $\Phi_d$ is a bijection by showing it is both injective and surjective. Suppose $L_1,L_2\in \B_i^d$ are distinct branches.  If $\Phi_d(L_1)=\Phi_d(L_2)$, then we can use (\ref{球近则锥近的式子2}) to obtain that $d_{x_i,19r_i}(L_1, L_2)<2 C_7 \varepsilon$. However, since the angle between $L_1$ and $L_2$ with respect to their common boundary is greater than $\alpha$, $d_{x_i,19r_i}(L_1, L_2)$ is significantly larger than $2 C_7 \varepsilon$, which leads to a contradiction. Thus, $\Phi_d$ is injective, which implies $\# \B_i^d \leq \# \B_j^d$.

   By a symmetric argument, we can also construct an injective map from $\B_j^d$ to $\B_i^d$. This is possible because $W_j$ coincides with $Z \in \mathcal{A}(s)$ in $22B_i$, allowing us to repeat the argument in the reverse direction. This implies $\# \B_j^d \leq \#  \B_i^d$. Therefore, we must have $\#  \B_i^d = \#  \B_j^d$, and $\Phi_d$ is a bijection.

   This completes the induction for all $d \in \{s, \dots, n\}$. Finally, the main estimate (\ref{球近则锥近的式子2}) involving $C_7$ follows directly from the inductive estimate (\ref{用kd归纳假设的式子}) since we defined the sequence $\{K_d\}$ such that $\frac{21}{19}K_d \le C_7$ for all $d$.
\end{proof}

\begin{remark}
    1. Lemma \ref{lem:球近则对应的锥接近} also holds for $Z(0,2)$ by the same argument. That is, given $i\in I(0)$, then we can replace $W_j$ by $Z(0,2)$ and get the same conclusions between $W_i$ and $Z(0,2)$.

    2. Note that since $i \in I_s$, we have $x_i \in L^s_i$. Thus, for each $d \in \{s, \dots, n\}$, every branch of $L^d_i$ contains $x_i$ and consequently intersects $20 B_i$. As a result, $\B_i^d$ coincides with the set of all branches of $L^d_i$. Therefore, if $t=s$, we have $b^d_i = b^d_j$ for all $d \in \{s, \dots, n\}$.
\end{remark}

\section{Construction of $f^m:L^m\to E_m$}

We say $C$ is a geometric constant if $C$ only depends on $n,n_0,\delta_0,\alpha$.

We aim to construct a parameterization $f^m$ of a big part of $E_m$ by $f^m : \Gamma^m\cap B(0,\rho^m_0-n_0^{-n}2^{-10})\to E_m\cap B(0,\rho^m_0-2n_0^{-n}2^{-10})$ for each $0\leq m\leq n$. We shall only care about $f^m(z)$ when $z$ lies in the set
\begin{equation}\label{Gamma^m定义}
    \Gamma^m= L^m\cap B(0,\rho^m_0)=\cup_l L^{m,l}\cap B(0,\rho^m_0),
\end{equation}
where $L^m$ is the $m$-spine of $Z(0,2)$, $L^{m,l}$ is a branch of $L^m$ and 
\begin{equation}\label{rho^m_k}
    \rho^m_k=1.95+2^{-10}+ n_0^{-n}2^{-10}\cdot(2-2m-\sum_{t=0}^k 2^{-t-1}) \text{ for }k\geq 0,
\end{equation}
where we say $k$ is the step. We want to construct $f^m$ by induction from $m=0$ to $m=n$, and for each $m$, get $f^m$ as the limit of mappings $f^m_k$, where $f^m_0=id$ and
\begin{equation}\label{fk,gk}
    f^m_{k+1}=g^m_k\circ f^m_k,\enspace g^m_k=\sum_{i\in I(k)} \theta_i\cdot \psi^m_i,
\end{equation}
where the $\psi_i^m$, $i\in I(k)$ are suitable deformations and will be defined soon. For each $m$ and each $k\geq 0$, let
\begin{equation}
        \Gamma^m_k=f^m_k(\Gamma^m),\enspace \Gamma^{m,l}_k=f^m_k(\Gamma^{m,l}).
\end{equation}

When $m=0$, since $E_0=\{0\}$ and $L^0=\{0\}$, let $f^0=f^0_k=id$ for $k\geq 0$. And we end the definition of $f^0$. Let $m>0$ be fixed. Assume that we have defined $f^t$ and $f^t_k$ for $t=0,...,m-1$. Furthermore, assume we have proved (M1)-(M4) for each dimension $0\leq t\leq m-1$ and each step $k\geq 0$ (when the dimension $t=0$, (M1)-(M4) hold trivially since $E_0=\{0\}$ and $\Gamma^0_k=\{0\}$ for all $k\geq 0$): 

(M1) $\dd(x,E_t)<C_{t,1}\varepsilon 2^{-k}$, for all $x\in \Gamma^t_k\cap B(0,\rho^t_k)$.

(M2) When $i\in I_t (k)$ and $x_i\in B(0,\rho^t_k)$, {  the ball  $5B_i$ is far from $\cup_{s=0}^{t-1} E_s$, due to \eqref{级别低的球离级别高的点远}. Consequently, $W_i$ is a set of type $t$ and $L_i^t$ coincides with a $t$-plane passing through $x_i$ in $5B_i$.} There is only one branch $\Gamma^{t,l}_k$ of $\Gamma^t_k$ meeting $5B_i$, and there exists a $C_{t,2}\varepsilon$-Lipschitz graph $G^t_i$ over $L^t_i$, such that
\begin{equation}\label{M2的式子}
    \Gamma^{t,l}_k\cap 5B_i=G^t_i\cap 5B_i.
\end{equation}
In addition, we have $\Gamma^{t,l}_k\cap B(x_i,C_{t,3}\varepsilon 2^{-k})\neq\emptyset$. {   To unify the notation with the branching case discussed later, we denote the plane $L_i^t$ itself as the unique branch $L_i^{t,l}$.} 

(M3) When $t\geq 1$, $i\in I_h(k)$ for some $0\leq h\leq t-1$ and $x_i\in B(0,\rho^t_k)$, {  $W_i$ is a set of type $t$}. There are $b^t_i$ branches of $\Gamma^{t}_k$ meeting $ 5B_i$. Denote by
$\{\Gamma^{t,l}_k\}$ these branches of $\Gamma^t_k$.  For each branch $\Gamma^{t,l}_k$ intersecting $5B_i$, there exists a unique branch of the spine $L^t_i$ that is sufficiently close to serve as its planar approximation. We label this specific branch as $L^{t,l}_i$. Specifically, there exists a $C_{t,4}\varepsilon$-Lipschitz graph $G^{t,l}_i$ defined over $D^{t,l}_i$, such that
\begin{equation}\label{M3的式子}
    \Gamma^{t,l}_k\cap 5B_i=G^{t,l}_i\cap 5B_i,
\end{equation}
where the closed domain $D^{t,l}_i\subset P^{t,l}_i$ is such that
\begin{equation}\label{Dtl和Ltl很近}
    d_{x_i,(5+1/500)r_i}(D^{t,l}_i,L^{t,l}_i)<C_{t,5}\varepsilon.
\end{equation}
This correspondence also holds for intersections. If two branches $\Gamma^{t,l}_k$ and $\Gamma^{t,l'}_k$ intersect at a lower-dimensional branch $\Gamma^{d,s}_k$, then the intersection of their corresponding spine branches, $L^{t,l}_i \cap L^{t,l'}_i$, is exactly the branch $L^{d,s}_i$, the unique branch of $L_i^d$ close to $\Gamma^{d,s}_k$. That is, $L^{t,l}_i \cap L^{t,l'}_i = L^{d,s}_i$.
Similarly, the boundaries of $L^{t,l}_i$ correspond to the boundaries of $\Gamma^{t,l}_k$. That is, if the boundaries of $\Gamma^{t,l}_k$ are $\{\Gamma^{t-1,l''}_k\}$, then the boundaries of $L^{t,l}_i$ are exactly the branches $\{L^{t-1,l''}_i\}$ that approximate them.
\begin{remark}
    Note that the branch indices of $L_i^t$ follow the global indices of $\Gamma^t_k$. Therefore, the indices $\{l\}$ for $\{L^{t,l}_i\}$ may not be consecutive. For example, suppose that the branches of $\Gamma_k^t$ meeting $5B_i$ are exactly $\Gamma^{t,3}_k$ and $\Gamma^{t,5}_k$. Then we label the branches of $L_i^t$ as $L^{t,3}_i$ and $L^{t,5}_i$ rather than $L^{t,1}_i$ and $L^{t,2}_i$.
\end{remark}

(M4) The restriction of $f^m_k$ to $\Gamma^m\cap B(0, \rho^{m+1}_0+n_0^{-n} 2^{10})$ is continuous. And the restriction of $f^m_k$ to $\Gamma^m\backslash \Gamma^{m-1} \cap B(0, \rho^{m+1}_0+n_0^{-n} 2^{10})$ is of class $C^1$, with a derivative that does not vanish. Moreover, for each $x\in \Gamma^{m-1}\cap B(0, \rho^{m+1}_0+n_0^{-n} 2^{10})$, we have $f^m_k(x)=f^{m-1}_k(x)$.

To ensure the constants are well-defined and to avoid circular dependencies, we clarify the hierarchy of constants. Let $\{C_{t,e}\}_{1\le e\le 5}$ denote the constants associated with dimension $t$. When $e \neq 3$, the constant $C_{t,e}$ depends on the global geometric parameters $n, n_0, \delta_0, \alpha$ and the full set of constants from lower dimensions, $\{C_{u,e}\}_{0\le u\le t-1,1\leq e\leq 5}$. When $e=3$, the constant $C_{t,3}$ depends on the constants $C_{t,1}$ and $C_{t,2}$ from the current dimension $t$.

Next we aim to show that (M1)-(M4) hold for dimension $m$. Before that, we should define the deformations $\psi_i^m$ for $i\in I(k)$, which are mentioned in (\ref{fk,gk}).

\subsection{Construction of $\psi^m_i$}

\subsubsection{Partitioning a neighborhood of $W_{i}$ into open sets $O(w)$}\label{构造O(w)}

We define $\psi^m_i$ for $i\in I(k)$. Before providing the detailed definitions, we briefly outline the strategy for constructing $\psi_i^m$. Suppose that $i\in I_t$. The map $\psi_i^m$ is defined as the composition of two mappings: $\psi_i^m = h_i \circ \eta_i^{m-1}$. Here, the map $\eta_i^{m-1}$ serves to align all lower-dimensional surfaces $\Gamma^{u}_k$ onto the local spines $L_i^u$ for every integer $u \in \{t, \dots, m-1\}$. Subsequently, the map $h_i$ acts as a piecewise orthogonal projection of the aligned $\Gamma^m_k$ onto $L_i^m$.

Motivated by the geometric scenario where the inductive hypotheses (M1)-(M4) hold for dimension $m$ at step $k$, the construction varies depending on the dimension $t$ relative to $m$.
When $t > m$, it follows from \eqref{级别低的球离级别高的点远} that the ball $B_i$ is far from $E_m$. Combined with the inductive hypothesis (M1), which implies $\Gamma^m_k$ is close to $E_m$, we deduce that $B_i$ is also far from $\Gamma^m_k$. Consequently, no modification is required in this neighborhood, and we simply set $\psi_i^m = h_i = \eta_i^{m-1} = id$. 
When $t = m$, similarly,  $\Gamma^{m-1}_k$ is far from $B_i$. In this region, $L_i^m$ coincides with a single $m$-plane. Therefore, we only need to project $\Gamma^m_k$ orthogonally onto $L_i^m$, which corresponds to setting $\eta_i^{m-1} = id$ and $h_i = \pi_i^m$. 
When $t < m$, the situation is more involved. We construct $\eta_i^{m-1}$ inductively to sequentially align $\Gamma^t_k$ to $L_i^t$, $\Gamma^{t+1}_k$ to $L_i^{t+1}$, and so on, up to aligning $\Gamma^{m-1}_k$ to $L_i^{m-1}$. After this alignment, we achieve the property that each $C^1$ branch of $\Gamma_k^m$ intersecting $B_i$ corresponds to a unique branch of the spine $L_i^m$. This allows us to define $h_i$ as the piecewise orthogonal projection of these branches onto their corresponding spine branches.

We now provide the detailed definitions of $\psi_i^m$, $\eta_i^{m-1}$, and $h_i$.

When $i\in \cup_{t=m+1}^{n+1} I_t (k)$, let $\eta_i^{m-1}=h_i=id$ and $\psi_i^m=h_i\circ \eta_i^{m-1}=id$.  

When $i\in I_m(k)$, $W_i\in \A(m)$ and $L^m_i$ is an $m$-plane passing through $x_i$. Let $\eta_i^{m-1}=id$ and $\psi^m_i=\pi^m_i\circ \eta_i^{m-1}=\pi_i^m$.

When $i\in \ccup_{t=0}^{m-1} I_t(k)$, we define $\eta^{m-1}_i$ inductively. Suppose that $i\in I_t(k)$ for some $0\leq t\leq m-1$, then $W_i\in \A(t)$ has spines of dimension from $t$ to $n$. First, we shall define $\eta^t_i,...,\eta^{m-1}_i$ in order to map $\Gamma^t_k\to L^t_i$, ..., $\Gamma^{m-1}_k\to L^{m-1}_i$ by induction from $u=t$ to $u=m-1$.

{For the base case $u=t$, $L_i^t$ coincides with a $t$-plane passing through $x_i$. By induction hypothesis  of (M2) for dimension $t$, $\Gamma^t_k$ is a $C_{t,2}\varepsilon$-Lipschitz graph over $L^t_i$ which is $C_{t,3}\varepsilon 2^{-k}$-close to $x_i$. Denote by $\varphi_{i}^{t}: P_i^t \to (P_i^t)^\perp$ this $C_{t,2}\varepsilon$-Lipschitz map. Here $P_i^t$ is the $t$-plane that contains $L_i^t$. Then $\Gamma^t_k\cap 5B_i\subset $graph$(\varphi_i^t):=\{x+\varphi_i^t(x):x\in P_i^t\}$.  For each $x\in \RR^N$, set
\begin{equation}
    \eta^t_i(x)=x-\varphi^t_i\circ \pi^t_i (x).
\end{equation}
Note that if $t=0$, $\Gamma^0_k=\{0\}$ and $L_i^{0}=\{0\}$, so $\eta^0_i\equiv id$. 
Consequently, for each $x \in \Gamma_k^t \cap 5B_i$, we have $x = \pi_i^t(x) + \varphi_i^t(\pi_i^t(x))$, which implies $\eta_i^t(x) = \pi_i^t(x) \in P_i^t\cap 5B_i= L_i^t\cap 5B_i$. Thus, $\eta_{i}^{t}$ maps points on $\Gamma_{k}^{t} \cap 5B_i$ into $L_{i}^{t}$. That is, 
\begin{equation}
    \eta^t_i( \Gamma^t_k\cap 5B_i)\subset L^t_i.
\end{equation}

Then we estimate the derivative of $\eta_i^t$. Since $\varphi_i^t$ is Lipschitz, $\eta_i^t$ is continuous. Furthermore, by the induction hypothesis (M4), $\Gamma_{k}^{t}$ is a $C^1$ manifold (away from $\Gamma_k^{t-1}$), which implies that $\varphi_{i}^{t}$ is actually a $C^1$ map on $\pi_i^t(\Gamma_k^t)$.  We can estimate the derivative of $\eta_i^t$ by $|D\eta_{i}^{t}(x)-I|\leq |D\varphi_i^t|\cdot 1 \le C_{t,2}\varepsilon$ for each $x\in 5B_i$ satisfying $\pi_i^t(x)\in \pi_i^t(\Gamma^t_k\cap 5B_i)$. The condition (M2) in dimension $t$ ensures that $\Gamma_k^t$ is close to $x_i$ and that the Lipschitz constant of $\varphi_i^t$ is sufficiently small. This implies that for every $x\in 4.9 B_i$, the inclusion $\pi_i^t(x)\in \pi_i^t(\Gamma^t_k\cap 5B_i)$ holds. Therefore, we have the following estimate in $4.9 B_i$:
\begin{equation}
   |D\eta^t_i -I|<C_{t,2}\varepsilon.
\end{equation}

Finally, we estimate the displacement of $\eta_i^t$. Let $p\in \Gamma_k^t\cap B(x_i,C_{t,3}\varepsilon 2^{-k})$, then for each $x\in 4.9B_i$, we have $|\eta_i^t(x)-x|=|\varphi_i^t(\pi_i^t(x))|\leq |\varphi_i^t(\pi_i^t(x))-\varphi_i^t(\pi_i^t(p))|+|\varphi_i^t(\pi_i^t(p))|\leq 5C_{t,2}\varepsilon r_i+C_{t,3}\varepsilon 2^{-k}<(5 C_{t,2}+ C_{t,3})\varepsilon 2^{-k}$. The last inequality holds because $r_i$ is smaller than $2^{-k}$, see below \eqref{构造覆盖时点分离}.
}

Now fix  $u\in \{ t+1,...,m-1\}$ and assume that we have defined $\eta^{u-1}_i$ on $\RR^N$, which maps $\Gamma^t_k \to  L^t_i$, ..., $\Gamma^{u-1}_k \to L^{u-1}_i$ in $4.9B_i$. In addition, $\eta^{u-1}_i$ is continuous on $\RR^N$ and $C^1$ in $4.9 B_i\backslash L^{u-2}_i$, with $|D\eta^{u-1}_i-I|<C\varepsilon$, here $C$ depends only on $C_{t,4},...,C_{u-1,4},\alpha$. Then we are ready to define $\eta^u_i$ in $\RR^N$.
We will divide $\RR^N$ into finite parts with respect to $W_i$ and define $\eta^u_i$ on each part separately. 

{  

\medskip

\noindent\textbf{Convention on Notation.}
    Throughout the construction in this subsection (up to the end of Section \ref{对齐映射构造}), we fix the indices $m, k, t$ and the ball index $i \in I_t(k)$. Specifically, $m$ is the dimension of the set $\Gamma^m_k$, $k$ denotes the current inductive step, and $i \in I_t(k)$ is the index of the ball $B_i$ where the associated cone $W_i$ is of type $t$. To avoid cumbersome notation, we will suppress the index $k$ in the definition of the map $\eta_i^u$ in the sequel. However, the dependence on $k$ is implicit, as these maps are constructed relative to the ball $B_i$ and $i\in I_t(k)$. Consequently, from this point until the end of Section \ref{构造O(w)}, we free the symbol $k$ from its role as the inductive step in defining $f^m$. Instead, we shall use $k$ to denote the index of the digits of the words defined in \eqref{Lambda定义}.
    
    From now on, we omit $i$ by replacing $L^h_i$ with $L^h$ for convenience. 

}

\medskip

Before giving the precise definitions of $O(w)$ in \eqref{O(w)的归纳定义}, we briefly explain the idea behind the sets $O(w)$. These sets form a hierarchical decomposition of the neighborhood of $L^u$. Our goal is to define the map $\eta^u_i$ near the spine $L^u$. The spine $L^u$ has a complicated structure, consisting of various branches that intersect at different angles along lower-dimensional spines. This geometric complexity makes it difficult to construct a globally unified map.
Therefore, we partition the neighborhood of $L^u$ into smaller regions based on the lower-dimensional spines $L^t, \dots, L^{u-1}$, such that within each region, the structure of $L^u$ is simple (effectively resembling a flat $u$-plane).

The reason for this split is geometric. Close to $L^u$ but far from $L^{u-1}$, $W_i$ looks like a set of type $u$, and $L^u$ itself looks like a flat $u$-plane. Here, we simply map $\eta_i^{u-1}(\Gamma^u_k)$ onto this single plane. However, near $L^{u-1}$ but away from $L^{u-2}$, $W_i$ looks like a set of type $u-1$. So $L^u$ splits into multiple $u$-dimensional half-planes meeting at a $(u-1)$-plane. These branches meet at angles larger than $\alpha$. This allows us to separate them. For each branch, we define a cone-shaped neighborhood starting from $L^{u-1}$. Inside such a cone, we see only one branch. So, we map $\eta_i^{u-1}(\Gamma^u_k)$ onto that branch. We use this same rule for lower dimensions.

To organize this, we use binary words to label the regions. A word acts like a address. It records our position relative to the spines. For dimension $j$, the digit $\delta_j$ tells us whether we are inside a conical neighborhood of $L^j$. If $\delta_j = 1$, we are inside a conical neighborhood of $L^j$. Here, we focus on a specific branch of $L^j$ where this branch looks like a flat $j$-plane. If $\delta_j = 0$, we are outside these conical neighborhoods. Here, we stay away from  $L^j$. We remain in a zone with simpler structure.

We now give the formal definitions.

For every two integers $b_1,b_2$ such that  $t\leq b_1< b_2 \leq n$ and for every $a>0$, let
\begin{equation}
\begin{split}
     &U^{b_1}_{b_2}(a)=\{x\in\RR^N:\dd(x,L^{b_2}) <a\cdot \dd(x,L^{b_1})\},\\
     &U^{b_1}_{b_2,s}(a)=\{x\in\RR^N:\dd(x,L^{b_2,s}) <a\cdot \dd(x,L^{b_1})\},\\
     &U^{b_1,l}_{b_2,s}(a)=\{x\in\RR^N:\dd(x,L^{b_2,s})<a\cdot \dd(x,L^{b_1,l})\},\\
     &U^{b_1}_F(a)=\{x\in\RR^N:\dd(x,F)<a\cdot \dd(x,L^{b_1})\}\text{, where $F$ is a set in $\RR^N$},\\
     &U^{b_1,l}_F(a)=\{x\in\RR^N:\dd(x,F)<a\cdot \dd(x,L^{b_1,l})\}\text{, where $F$ is a set in $\RR^N$}.\\
\end{split}
\end{equation} 
Given $a>0$, since $L^{b_2}=\cup_s L^{b_2,s}$, for each $b_1<b_2$ we have
\begin{equation}\label{Ub1b2拆解}
    U^{b_1}_{b_2}(a)=\cup_s U^{b_1}_{b_2,s}(a).
\end{equation}
We define the set $\Lambda$ of binary words with indices in ascending order as follows:
\begin{equation}\label{Lambda定义} 
    \Lambda = \{ \delta_{t}\dots\delta_{s} : t \le s \le u, \ \delta_{t}=1, \ \delta_{d} \in \{0,1\} \text{ for } t < d \le s, \text{ and } \delta_{s}=1 \text{ if } s=u \}.
\end{equation}
For any $s\in \{t,\dots,u\}$, let $\Lambda_s \subset \Lambda$ denote the set of words ending at index $s$ (with $\Lambda_t=\{1\}$). For a word $w=\delta_t \dots \delta_s \in \Lambda$, its length is defined as $s-t+1$. If a word $v=\delta_t' \dots \delta_{s'}' \in \Lambda$ satisfies $s' \geq s$ and $\delta_j = \delta_j'$ for $t \leq j \leq s$, we write $v=w \delta_{s+1}' \dots \delta_{s'}'$. 

For $w=\delta_t\dots\delta_s$, let $\mathcal{N}(w) = \max\{j : \delta_j=1, t\le j \le s\}$ be the index of the last non-zero digit.

Let us define three sequences of angles which are all small enough.
For $s\geq 1$, let $ \theta_{1,1}=\frac{\alpha}{100}$ 
and
\begin{equation}\label{theta,beta}
    \beta_s=\frac{\theta_{s,1}}{10},\text{ }\theta_{s,2}=\frac{\theta_{s,1}}{10^3},\text{ }\theta_{s+1,1}=\frac{\alpha\theta_{s,2}}{10}\text{ for }s\geq 1
\end{equation}
Then the constants $\{\theta_{s,1}\}_s$, $\{\theta_{s,1}\}_s$ and $\{\beta_s\}_s$ are such that $\theta_{s,2}\leq \beta_s/10$. 
In addition, we have $\theta_{s,1}>\beta_s>\theta_{s,2}>\theta_{s+1,1}$ for all $s\geq 1$. 

{We will define $O(w)$ depending on the word $w=\delta_t\dots\delta_k\in \Lambda$ so that for each $j\in \{t,\dots,k\}$, $O(w)$  is a neighborhood of $L^j$ (or a subset of $L^j$) if $\delta_j=1$ and is not a neighborhood of $L^j$ if $\delta_j=0$. Let us do this now.} 

{ We define the sets $O(w)$ recursively.}
Set $O(1)=\RR^N$. Let
\begin{equation}\label{O(11)O(10)}
    \begin{split}
    &O(\delta_t\delta_{t+1})=O(11)=O(1)\cap U^t_{t+1}(\sin\theta_{1,1})
    \text{ when }\delta_{t+1}=1,\\
    &O(\delta_t\delta_{t+1})=O(10)=O(1)\cap U^t_{F(10)}(\sin\theta_{1,2})\text{ when }\delta_{t+1}=0,
    \end{split}
\end{equation}
where
\begin{equation}\label{F(10)}
    \begin{aligned}
    F(\delta_t\delta_{t+1})&=F(10)=\{x\in L^u:\dd(x,L^{t+1})\geq \sin\beta_1\cdot \dd(x,L^t)\}. 
    \end{aligned}
\end{equation}
Assume we have defined $O(\delta_t...\delta_k)$ for all words $\{\delta_t...\delta_k: \delta_t=1, \delta_j=0,1\text{ for }t<j\leq k\}$, where $k$ is a number in $\{t+1,...,u-1\}$. We continue to define $O(\delta_t...\delta_{k+1})$. Let
\begin{equation}\label{O(w)的归纳定义}
\begin{split}
&O(\delta_t...\delta_{k+1})=O(\delta_t...\delta_{k}1)=O(\delta_t...\delta_k)\cap U^{\mathcal{N}(\delta_t...\delta_k)}_{k+1}(\sin\theta_{k+1-t,1})\text{ when }\delta_{k+1}=1,\\
&O(\delta_t...\delta_{k+1})=O(\delta_t...\delta_k 0)=O(\delta_t...\delta_k)\cap U^{\mathcal{N}(\delta_t...\delta_k)}_{F(\delta_t...\delta_k 0)}(\sin\theta_{k+1-t,2})\text{ when }\delta_{k+1}=0,
\end{split}
\end{equation}
where
\begin{equation}\label{F(w0)}
    F(\delta_0...\delta_k 0)
    =\{x\in L^u:\dd(x,L^{k+1})\geq \sin\beta_{k+1-t}\cdot\dd(x,L^{\mathcal{N}(\delta_t...\delta_k)})\}.
\end{equation}
By repeating the construction, we get $O(w)$ for all words $w\in \Lambda$. {   For a given word $w=\delta_t\dots\delta_k\in \Lambda$ and $\lambda>0$,  we define the scaled sets $O^{(\lambda)}(w)$ by replacing every occurrence of $\sin(\cdot)$ arguments in definitions (\ref{O(w)的归纳定义}) with $\sin(\lambda \cdot)$. That is, $ O^{(\lambda)}(w1)=O(w)\cap U^{\mathcal{N}(w)}_{k+1}(\sin(\lambda \theta_{k+1-t,1}))$ and $O^{(\lambda)}(w0)=O(w)\cap U^{\mathcal{N}(w)}_{F(w0)}(\sin(\lambda \theta_{k+1-t,2}))$.}

{  
We refine the open sets according to branches. Consider a word  $w=\delta_t\dots\delta_k\in \Lambda$. 

If $\delta_k=1$, let $\{L^{k,s}\}_s$ be the set of branches of $L^k$. If $k>t$, we define
\begin{equation}\label{O(w1,s)定义式}
O(w, s) = O(w) \cap U_{k, s}^{\mathcal{N}(w')}(\sin \theta_{k-t, 1}),
\end{equation}
where $w'=\delta_t\dots\delta_{k-1}$.
Observe that $O(w,s)\subset O(w)$ and $O(w) = \cup_s O(w, s)$. 
If $k=t$ (where $L^t$ is a $t$-plane), we define $O(1, l) = \mathbb{R}^N$, where $l$ denotes the unique branch of $\Gamma^{t}_k$ intersecting $5B_i$, see \eqref{M2的式子}. Similarly, we define the scaled refined subset $O^{(\lambda)}(w,s)$ by replacing the sine argument in \eqref{O(w1,s)定义式} with $\sin(\lambda \theta_{k-t,1})$.

If $\delta_k=0$, for a specific branch $L^{\mathcal{N}(w), l}$, let $F(w,l)$ be the region on $u$-branches containing $L^{\mathcal{N}(w),l}$ that is far from their $k$-boundaries containing $L^{\mathcal{N}(w),l}$:
\begin{equation}\label{F(w,l)}
\begin{aligned}
    F(w,l) = \bigcup_{\substack{L \supset L^{\mathcal{N}(w), l} \\ L \text{ is a } u\text{-branch}}} 
    \bigg\{ 
        x \in L : \operatorname{dist}(x, L^{k,s}) \ge (\sin \beta_{k-t}) \operatorname{dist}(x, L^{\mathcal{N}(w), l}), \\
        \text{ for all }L^{k,s}\text{ such that } L^{\mathcal{N}(w),l}\subset L^{k,s} \subset L \bigg\}.
\end{aligned}
\end{equation}
Note that the global set $F(w)$ is not simply the union of these local sets $F(w, l)$. However, in the localized region $O(w) \cap O(\delta_t\dots\delta_{\mathcal{N}(w)}, l)$, the set $F(w)$ coincides exactly with $F(w, l)$. Specifically, on every $u$-branch containing $L^{\mathcal{N}(w), l}$, we effectively remove the conical neighborhoods of its $k$-boundaries that contain $L^{\mathcal{N}(w), l}$. We will prove this equivalence formally in Lemma \ref{O(w)的性质}.

Since the different branches of the spines of $W_i$ are separated by angles larger than $\alpha$, and the angles defined in \eqref{theta,beta} are sufficiently small, the sets $\{O(w)\}_{w\in \Lambda}$ satisfy several crucial properties. We establish these properties in Lemma \ref{O(w)的性质}, which will be essential for understanding the geometric structure of these open sets.
}

\begin{lemma} \label{O(w)的性质}
Let $w = \delta_t \dots \delta_k$ be a word with $t < k \le u$. Let $w' = \delta_t \dots \delta_{k-1}$. To simplify notation, let $\nu = \mathcal{N}(w')$. We denote the branch decompositions of the relevant spines by $L^k = \bigcup_s L^{k,s}$ and $L^{\nu} = \bigcup_l L^{\nu,l}$. Then the set $O(w)$ satisfies the following properties:

\noindent \textbf{Case A. If $\delta_k = 1$:}
\begin{enumerate}[label={\rm(\roman*)}, leftmargin=*, nosep]
    \item\label{O(w1)等价定义} We have the decomposition
    \begin{equation}\label{O(w)等价定义1}
        O(w) = \bigcup_l \left( O(w') \cap O(\delta_t\dots\delta_{\nu}, l) \cap \bigcup_{\substack{L^{k,s} \supset L^{\nu,l}}} U_{k,s}^{\nu,l}(\sin\theta_{k-t,1}) \right).
    \end{equation}

    \item\label{O(w1)=不交并} The collection of sets $\{O(w,s)\}_s$ is mutually disjoint.

    \item\label{O(w1)内退化为type k} For each $s$, within the region $O^{(9)}(w,s)$, the set $W_i$ coincides with a set of type $k$. This implies that for any branch $L$ of $W_i$ of any dimension that does not contain $L^{k,s}$, and for any $x \in O(w,s)$, we have
    \begin{equation}\label{A(iii)距离不等式}
        \operatorname{dist}(x, L) > \operatorname{dist}(x, L^{k,s}).
    \end{equation}
    In particular, $\operatorname{dist}(x, L^k) = \operatorname{dist}(x, L^{k,s})$.
\end{enumerate}

\noindent \textbf{Case B. If $\delta_k = 0$:}
\newline 
Since $\delta_k = 0$, the reference spine does not update, i.e., $\mathcal{N}(w) = \nu$.
\begin{enumerate}[label={\rm(\roman*)}, leftmargin=*, nosep]
    \item\label{B(i)等价定义} We have the decomposition
    \begin{equation}\label{B(i)等价定义式子}
        O(w) = \bigcup_l \left( O(w') \cap O(\delta_t\dots\delta_{\nu}, l) \cap U_{F(w,l)}^{\nu, l}(\sin\theta_{k-t,2}) \right).
    \end{equation}

    \item\label{B(ii)估算点到L^k的距离} For any $x \in O(w)$, the following distance estimate holds:
    \begin{equation}\label{case B(ii)式子}
        \operatorname{dist}(x, L^k) \ge \sin(\beta_{k-t} - \theta_{k-t,2}) \cdot \operatorname{dist}(x, L^{\nu}).
    \end{equation}
\end{enumerate}
\end{lemma}

\begin{remark}
    Let us briefly explain the geometric meaning of Lemma \ref{O(w)的性质}. Recalling the definition (\ref{O(w)的归纳定义}), we see that the construction of $O(w)$ is essentially a dynamic process. As we extend the word, the digit $\delta_k$ at step $k$ determines whether we select the conical neighborhood of $L^k$, or remove a conical neighborhood of $L^k$ and take a neighborhood of the complement.

    Lemma \ref{O(w)的性质} ensures that this process is purely local. Specifically, Case A\ref{O(w1)等价定义} and Case B\ref{B(i)等价定义} show that to define $O(w)$, we do not need to consider the entire set $W_i$. Instead, we only need to focus on the region $O(w')$ and the ancestor region $O(\delta_t \dots \delta_{\nu})$, where the structure of $W_i$ simplifies to a cone of type $\nu$ (by Case A\ref{O(w1)内退化为type k}). Furthermore, Case A\ref{O(w1)=不交并} allows us to work independently within each connected component of $O(w)$, and Case A\ref{O(w1)内退化为type k} and Case B\ref{B(ii)估算点到L^k的距离} provide quantitative distance estimates in each region.

    Moreover, Lemma \ref{O(w)的性质} lays the foundation for defining the maps $\eta^{u-1}_i$ and $\psi_i^m$ later. It ensures that in every final open set $O(w)$ (for $w \in \Lambda_u$), there is exactly one unique $u$-branch that intersects it and behaves like a flat $u$-plane. This allows us to define orthogonal projections onto this $u$-plane in $O(w)$.
\end{remark}

\begin{proof}
We proceed by induction on the index $k$.

\noindent \textbf{Step 1. Base Case ($k=t+1$).} 

We verify the lemma for words of length 2.

\noindent \textbf{Case A ($\delta_{t+1} = 1$).} 
Here $w=11$, so $w'=1$ and $\nu=t$. And $L^t$ is a $t$-plane.

(i) Since $O(w') = \mathbb{R}^N$ and every branch $L^{t+1,s}$ contains $L^t$, the formula in \eqref{O(w)等价定义1} trivially reduces to $\cup_s U^{t}_{t+1,s}(\sin\theta_{1,1})$, which matches the definition of $O(11)$ in \eqref{O(w)的归纳定义}. 

\medskip

(ii) The disjointness follows from the angular separation of the branches. Let $x \in U_{t+1,s}^t(\sin\theta_{1,1})$ and let $z$ be the orthogonal projection of $x$ onto $L^t$. The condition implies that the angle between the vector $x-z$ and the branch $L^{t+1,s}$ is less than $\theta_{1,1}$. Since the angle between any distinct branches $L^{t+1,s}$ and $L^{t+1,s'}$ is at least $\alpha$, the triangle inequality implies that the angle between $x-z$ and the branch $L^{t+1,s'}$ is greater than $\alpha - \theta_{1,1} $. Since $\alpha > \theta_{1,1}=\alpha/10^2$, we have $\sin(\alpha - \theta_{1,1}) > \sin\theta_{1,1}$, which implies $x \notin U_{t+1,s'}^t(\sin\theta_{1,1})$.

\medskip

(iii) For each $s$, any branch $L$ of $W_i$ that does not contain $L^{t+1,s}$ forms an angle of at least $\alpha$ with $L^{t+1,s}$ relative to the spine $L^t$. However, the region $O^{(9)}(11,s)$ is defined by an opening angle of $9\theta_{1,1}$. Since $9\theta_{1,1} < \alpha$, the branch $L$ is strictly separated from this region. Thus, within this region, every branch of $W_i$ must contain $L^{t+1,s}$, which implies $W_i$ coincides with a set of type $t+1$. 

For any $x \in O(11,s)$, the point $x$ deviates from the axis $L^{t+1,s}$ by an angle of at most $\theta_{1,1}$. Since $L$ lies outside the region $O^{(9)}(11,s)$, its angular separation from the axis is at least $9\theta_{1,1}$. By the triangle inequality for angles, the angle between $x-\pi^t(x)$ and $L$ is at least $8\theta_{1,1}$. Consequently, $\operatorname{dist}(x,L) \geq \sin(8\theta_{1,1})\operatorname{dist}(x,L^t) > \sin\theta_{1,1}\operatorname{dist}(x,L^t) \ge \operatorname{dist}(x,L^{t+1,s}).$
Thus, \eqref{A(iii)距离不等式} holds.

\

\noindent \textbf{Case B ($\delta_{t+1} = 0$).} Here $w=10$, so $w'=1$ and $\nu=\mathcal{N}(w)=t$. And $L^t$ is a $t$-plane.

(i) Since the initial spine $L^t$ consists of a single branch, the union in \eqref{B(i)等价定义式子} reduces to a single term involving $F(10,l)$. Explicitly,
\begin{equation}
    F(10,l)=\bigcup \{x \in L : \dd(x, L^{t+1,s}) \ge (\sin \beta_{1}) \dd(x, L^{t}) \text{ for all } L^{t+1,s} \subset L,\ L\text{ is a } u\text{-branch} \}
\end{equation}
We now verify $O(10) = U_{F(10,l)}^t(\sin\theta_{1,2})$. The inclusion $O(10) \subset U_{F(10,l)}^t(\sin\theta_{1,2})$ is trivial since $F(10) \subset F(10,l)$. Conversely, for any $y \in F(10,l)$ on a $u$-branch $L$, the distance condition holds for all $(t+1)$-branches contained in $L$ by definition. For any $(t+1)$-branch $L'$ not contained in $L$, since $L\cap L'=L^t$, the angular separation $\alpha$ ensures that $\dd(y, L') \ge \sin\alpha \cdot \dd(y, L^t) > \sin\beta_1 \cdot \dd(y, L^t)$. Thus, $y$ satisfies the condition for all $(t+1)$-branches, implying $y \in F(10)$. The decomposition holds.

\medskip

(ii) For any $x \in O(10)$, there exists a reference point $y \in F(10)$ such that $x$ lies within an angular distance of $\theta_{1,2}$ from $y$ relative to the spine $L^t$. Since the definition of $F(10)$ ensures $y$ has an angular separation of at least $\beta_1$ from $L^{t+1}$, the triangle inequality implies that $x$ is separated from $L^{t+1}$ by an angle of at least $\beta_1 - \theta_{1,2}$. The distance bound \eqref{case B(ii)式子} follows immediately.

\

\noindent \textbf{Step 2. Inductive Step ($k>t+1$).}

Assume that Lemma \ref{O(w)的性质} holds for all words of length $k-t$, that is, words ending at index $k-1$. We now consider an arbitrary word $w = \delta_t \dots \delta_k$ of length $k-t+1$. Since $w'$ has length $k-t$, by the inductive hypothesis, the properties hold for $O(w')$. By Case A\ref{O(w1)等价定义} of the inductive hypothesis, $O(w')$ is the disjoint union of components $\{O(w') \cap O(\delta_t\dots \delta_{\nu}, l)\}_l$. So the decomposition of $O(w)$ and its geometric properties can be verified separately on each such component. 

Fix an index $l$. We restrict our attention to the specific  component $ O(w') \cap O(\delta_t \dots \delta_\nu, l)$. To simplify the notation for the rest of the proof, we define: 
\begin{equation} 
\Omega' := O(w') \cap O(\delta_t \dots \delta_\nu, l), \ Z := L^{\nu, l}.
\end{equation}
We now prove that these properties extend to $O(w)$.

\noindent \textbf{Case A: $\delta_k = 1$.} In this case, $w = w'1$. 

(i) Recall that $O(w) = O(w') \cap U^{\nu}_{k}(\sin\theta_{k-t,1})$. For each $x\in  \Omega'$, Case A\ref{O(w1)内退化为type k} ensures that $\operatorname{dist}(x, L^{\nu}) = \operatorname{dist}(x, Z)$ and that $\operatorname{dist}(x, L) > \operatorname{dist}(x, Z)$ for any $k$-branch $L \not\supset Z$. Thus, we have 
\begin{equation}  
\Omega' \cap U^{\nu}_k(\sin\theta_{k-t,1}) = \Omega' \cap \bigcup_{L^{k,s}\supset Z} U^{\nu,l}_{k,s}(\sin\theta_{k-t,1}). 
\end{equation} 
Taking the union over all $l$, we get \eqref{O(w)等价定义1}.

\medskip

(ii) We proceed to show that for distinct indices $s$, the sets $\{O(w, s)\}_s$ are mutually disjoint. By Case A\ref{O(w1)等价定义}, it suffices to prove disjointness in $\Omega'$. Specifically, we must show that the collection of sets
\begin{equation}\label{A(ii)同一个l下的不交邻域}
     \left\{ \Omega' \cap U^{\nu,l}_{k,s}(\sin\theta_{k-t,1}) \right\}_{s: L^{k,s} \supset Z}
\end{equation}
are mutually disjoint. Let $L^{k,s}$ and $L^{k,s'}$ be two branches that contain $Z$. We consider the two cases $\nu=k-1$ and $\nu < k-1$. 

If $\nu=k-1$, then the branches $L^{k,s}$ and $L^{k,s'}$ share $Z$ as a common boundary of codimension 1. The angle between them is at least $\alpha$, and since $2\theta_{k-t,1} < \alpha$, the sets in \eqref{A(ii)同一个l下的不交邻域} are disjoint.

If $\nu<k-1$, then the last letter $\delta_{k-1}$ of $w'$ is 0. We aim to show that for each $x\in\Omega' \cap U^{\nu,l}_{k,s}(\sin\theta_{k-t,1}) $, we have
\begin{equation}\label{A(ii)x离s_2的k分支远}
    \dd(x,L^{k,s'})>\sin\theta_{k-t,1} \cdot \dd(x,Z).
\end{equation}
If \eqref{A(ii)x离s_2的k分支远} holds, then $x$ cannot belong to $O(w, s')$, implying disjointness.

Pick $y \in L^{k,s}$ such that $|x-y| = \operatorname{dist}(x, L^{k,s})$. By Lemma \ref{分支分离性}, we have $ \dd(y,L^{k,s'})\geq \sin \alpha \cdot \dd(y,L^{k,s}\cap L^{k,s'})$. Since the intersection $L^{k,s} \cap L^{k,s'}$ is contained in $L^{k-1}$, we have $\dd(y,L^{k,s}\cap L^{k,s'})\geq \dd(y, L^{k-1})\geq \dd(x,L^{k-1})-|x-y|$. Since $x\in U^{\nu,l}_{k,s}(\sin\theta_{k-t,1})$, we have $|x-y|<\sin\theta_{k-t,1}\cdot \dd(x,Z)$.  Also, because $w'$ ends in 0, we can use Case B\ref{B(ii)估算点到L^k的距离} of the inductive hypothesis. This implies $\dd(x,L^{k-1})\geq \sin(\beta_{k-1-t}-\theta_{k-1-t,2})\cdot\dd(x,Z)$. Combining these estimates, we have
\begin{equation}
    \begin{aligned}
        \dd(x,L^{k,s'}) &\geq \dd(y,L^{k,s'}) - |x-y| \\
        &\geq \sin\alpha \cdot (\dd(x,L^{k-1}) - |x-y|) - |x-y| \\
        &= \sin\alpha \cdot \dd(x,L^{k-1}) - (1+\sin\alpha)|x-y| \\
        &\ge \left( \sin\alpha \sin(\beta_{k-1-t}-\theta_{k-1-t,2}) - (1+\sin\alpha)\sin\theta_{k-t,1} \right) \cdot \dd(x,Z) \\
        &> \sin\theta_{k-t,1} \cdot \dd(x,Z).
    \end{aligned}
\end{equation}
The last inequality holds due to the parameter choices in \eqref{theta,beta}. First, we have the lower bound $\sin \alpha \sin( \beta_{k-1-t}-\theta_{k-1-t,2} )>({2}/{\pi})^2 \cdot {99}\cdot{10^{-3}} \alpha \theta_{k-1-t,1}$. Second, we have the upper bound $ (2+\sin\alpha)\sin\theta_{k-t,1}<3\sin\theta_{k-t,1}<{3}\cdot{10^{-4}}\alpha \theta_{k-1-t,1}$. And the former lower bound is strictly larger than the latter upper bound. 
Thus,  we get \eqref{A(ii)x离s_2的k分支远}. This completes the proof of Case A\ref{O(w1)=不交并}.

\medskip

(iii) By definition, $O^{(9)}(w,s) = O(w') \cap U^{\nu}_{k,s}(\sin9\theta_{k-t,1})$. It suffices to prove that $W_i$ coincides with a set of type $k$ in $\Omega' \cap U^{\nu,l}_{k,s}(\sin9\theta_{k-t,1}) $. We need only show that any branch $L$ intersecting this region contains the spine $L^{k,s}$. By Case A\ref{O(w1)内退化为type k} of the inductive hypothesis, any branch intersecting $\Omega'$ must contain $Z$. Thus, both $L^{k,s}$ and $L$ contain $Z$. We consider the two cases $\nu=k-1$ and $\nu < k-1$.

If $\nu=k-1$, then $Z= L^{k-1,l}$ and $\delta_t\dots\delta_{\nu}=w'$. Suppose for contradiction that $L \not\supset L^{k,s}$. Then $L \cap L^{k,s} = Z$. Pick $x\in L \cap O(w',l) \cap U^{k-1,l}_{k,s}(\sin9\theta_{k-t,1})$. 
Let $y \in L^{k,s}$ and $z \in Z$ be the points closest to $x$ on $L^{k,s}$ and $Z$, respectively. Since $x \in U^{k-1,l}_{k,s}(\sin9\theta_{k-t,1})$, we have $|x-y| < \sin9\theta_{k-t,1}\cdot |x-z|$. This strict inequality implies that $y$ lies in the relative interior of $L^{k,s}$. Therefore, $x-y$ is orthogonal to the plane $P^{k,s}$ and $z$ is also the closest point on $Z$ to $y$, where $P^{k,s}$ is the affine span of $L^{k,s}$. Thus, $\angle xzy < 9\theta_{k-t,1}$. However, by Definition \ref{角的定义} and the separation condition \eqref{alpha定义}, we get $\angle xzy \ge \alpha$. This leads to a contradiction because $9\theta_{k-t,1} < \alpha$. Thus, $L \supset L^{k,s}$.

If $\nu<k-1$, then the last letter $\delta_{k-1}$ of $w'$ is 0. Suppose for contradiction that $L \not\supset L^{k,s}$. Then $L\cap L^{k,s}\subset L^{k-1}$. Pick $x\in L\cap \Omega' \cap U^{\nu,l}_{k,s}(\sin9\theta_{k-t,1})$. The neighborhood condition implies $\dd(x,L^{k,s})<\sin 9\theta_{k-t,1}\cdot \dd(x,Z)$. Since $x\in O(w')$, by Case B\ref{B(ii)估算点到L^k的距离} of the inductive hypothesis, $\dd(x,L^{k-1})\geq \sin (\beta_{k-1-t}-\theta_{k-1-t,2})\dd(x,Z)$. Then, by Lemma \ref{分支分离性}, we have
\begin{equation}
\begin{aligned}
    \dd(x,L^{k,s})&\geq \sin\alpha\cdot \dd(x,L\cap L^{k,s})\geq \sin\alpha\cdot \dd(x,L^{k-1})\\
    &\geq \sin\alpha \cdot \sin (\beta_{k-1-t}-\theta_{k-1-t,2} ) \dd(x,Z)\\
    &>\sin 9\theta_{k-t,1}\cdot \dd(x,Z).
\end{aligned}
\end{equation}
The last inequality holds because $\sin \alpha \sin( \beta_{k-1-t}-\theta_{k-1-t,2} )>({2}/{\pi})^2 \cdot 99\cdot{10^{-3}} \alpha \theta_{k-1-t,1}$ and $ \sin9\theta_{k-t,1}<9\sin\theta_{k-t,1}<{9}\cdot{10^{-4}}\alpha \theta_{k-1-t,1}$. And the former lower bound is strictly larger than the latter upper bound.
This leads to a contradiction. Thus, $L \supset L^{k,s}$.

Let us prove \eqref{A(iii)距离不等式}.  Fix $x \in \Omega' \cap U^{\nu,l}_{k,s}(\sin\theta_{k-t,1})$. Consider any branch $L \not\supset L^{k,s}$. If $L \not\supset Z$, then by Case A\ref{O(w1)内退化为type k} of the inductive hypothesis, $\dd(x,L) > \dd(x,Z) \ge \dd(x,L^{k,s})$. If $L \supset Z$, then $L$ cannot intersect the region $\Omega' \cap U^{\nu,l}_{k,s}(\sin9\theta_{k-t,1})$. By Case A\ref{O(w1)内退化为type k} of the inductive hypothesis, $W_i$ is a set of type $\nu$ with common boundary $Z$ in this region. Applying the triangle inequality for angular radii yields $\dd(x, L) \ge \sin(8\theta_{k-t,1})\cdot \dd(x, Z) > \sin\theta_{k-t,1}\cdot \dd(x, Z) > \dd(x, L^{k,s})$. In particular, $\dd(x, L^k) = \dd(x, L^{k,s})$.

\ 

\noindent \textbf{Case B ($\delta_k = 0$).} In this case, $w = w'1$. 

(i) By Case A\ref{O(w1)=不交并} and \ref{O(w1)内退化为type k}, we have the decomposition $O(w)=\bigcup_l O(w')\cap O(\delta_t\dots\delta_{\nu},l)\cap U^{\nu,l}_{F(w)}(\sin\theta_{k-t,2})$. 
By the definition of $F(w)$ and $F(w,l)$, $O(w)$ is contained in the right-hand side of \eqref{B(i)等价定义式子}. Fix the index $l$. Let $L$ be a branch of $W_i$, by Case A\ref{O(w1)内退化为type k} of the inductive hypothesis, the intersection $\Omega' \cap U_L^{\nu,l}(\sin\theta_{k-t, 2})$ is non-empty only if $L \supset Z$. Thus, it suffices to fix a $u$-branch $L \supset Z$, and show that for each point $x\in \Omega' \cap U_{F(w,l)\cap L}^{\nu,l}(\sin\theta_{k-t,2})$, we have $x\in O(w)$.

Pick a point $y\in L\cap F(w,l)$ such that $\dd(x, F(w,l)\cap L)=|x-y|$. Then $y\in O(\delta_t\dots\delta_{\nu},l)$ and $|x-y|<\sin \theta_{k-t,2}\cdot \dd(x,Z)$. We aim to prove that $y\in F(w)$. To do this, we must verify that for any $k$-branch $L'$, the inequality
\begin{equation}\label{B(i)验证y离其他k分支也远}
    \dd(y,L') \geq \sin\beta_{k-t} \cdot \dd(y,Z)
\end{equation}
holds. If $L'\supset Z$ and $L'\subset L$, \eqref{B(i)验证y离其他k分支也远} holds by the definition of $F(w,l)$. If $L'\not\supset Z$, then since $y\in O(\delta_t\dots \delta_{\nu},l)$, Case A\ref{O(w1)内退化为type k} implies $\dd(y,L')>\dd(y, Z)$, so \eqref{B(i)验证y离其他k分支也远} holds automatically. It remains to consider the case where $L'\supset Z$ but $L'\not\subset L$. We consider the two cases $\nu=k-1$ and $\nu < k-1$.

If $ \nu=k-1$, then $Z=L^{k-1,l}$ and $L\cap L'=Z$. By Lemma \ref{分支分离性}, we have
\begin{equation}
    \text{dist}(y,L')\geq\sin\alpha\cdot\text{dist}(y,L\cap L')=\sin\alpha\cdot\text{dist}(y,Z).
\end{equation}
Since $\sin\alpha > \sin\beta_{k-t}$, the inequality \eqref{B(i)验证y离其他k分支也远} holds.

If $ \nu<k-1$, then the last letter $\delta_{k-1}$ of $w'$ is 0 and $L\cap L'\subset L^{k-1}$. Since $x\in O(w')$,  $y$ is also contained in $O(w')$. By Case B\ref{B(ii)估算点到L^k的距离} of the inductive hypothesis, $\dd(y,L^{k-1})\geq \sin(\beta_{k-1-t}-\theta_{k-1-t,2})\cdot\dd(y,Z)$. Therefore, by Lemma \ref{分支分离性}, we have
\begin{equation}
\begin{aligned}
    \dd(y, L')&\geq \sin\alpha \cdot \dd(y,L\cap L')\geq \sin\alpha \cdot \dd(y,L^{k-1})\\
    &\geq \sin\alpha\cdot \sin(\beta_{k-1-t}-\theta_{k-1-t,2})\dd(y,Z)\\
    &>\sin\beta_{k-t}\cdot\dd(y,Z).
\end{aligned}
\end{equation}
The last inequality holds because $\sin \alpha \sin( \beta_{k-1-t}-\theta_{k-1-t,2} )>({2}/{\pi})^2 \cdot {99}\cdot {10^{-3}} \alpha \theta_{k-1-t,1}$ and $ \sin\beta_{k-t}<{10^{-5}}\alpha \theta_{k-1-t,1}$. And the former lower bound is strictly larger than the latter upper bound. Thus, we have \eqref{B(i)验证y离其他k分支也远}.

Consequently, $y\in F(w)$, and therefore $x\in O(w)$. This completes the proof of \eqref{B(i)等价定义式子}.

\medskip

(ii) It suffices to prove \eqref{case B(ii)式子} for $x\in \Omega' \cap U^{\nu,l}_{F(w,l)}(\sin\theta_{k-t,2})$.  Let $y \in F(w,l)$ be such that $|x-y| = \dd(x, F(w,l))$. Since $y \in F(w,l) \subset W_i$, there exists a $u$-branch $H$  such that $Z \subset H$ and $k < \dim H \le u$, with $y$ lying in the relative interior of $H$. Thus, the vector $x-y$ is orthogonal to the plane spanned by $H$.

Let $z \in Z$ be the point closest to $x$ on $Z$, i.e., $\dd(x,Z) = |x-z|$.
Since $Z \subset H$ and $x-y \perp H$, $z$ is also the closest point to $y$ on $Z$. Thus, we also have $\dd(y,Z) = |y-z|$.

Since $y \in F(w)$, the open cone with vertex $z$, axis $y-z$, and half-angle $\beta_{k-t}$ does not intersect $L^k$ within the region $O^{(9)}(\delta_t\dots\delta_{\nu},l)$. On the other hand, since $x \in U_{F(w,l)}^{\nu,l}(\sin\theta_{k-t,2})$, the vector $x-z$ forms an angle of at most $\theta_{k-t,2}$ with the vector $y-z$. Thus, for $k$-branches containing $Z$, the angular separation from $x-z$ is at least $\beta_{k-t} - \theta_{k-t,2}$. Since Case A\ref{O(w1)内退化为type k} ensures that the $j$-branches not containing $Z$ are sufficiently far away and thus do not minimize the distance to $x$, we conclude that
\begin{equation}
    \dd(x, L^k) = \dd(x, L^k \cap O(\delta_t\dots\delta_{\nu},l)) > \sin(\beta_{k-t}-\theta_{k-t,2}) \cdot \dd(x, Z). 
\end{equation}
This completes the proof.
\end{proof}

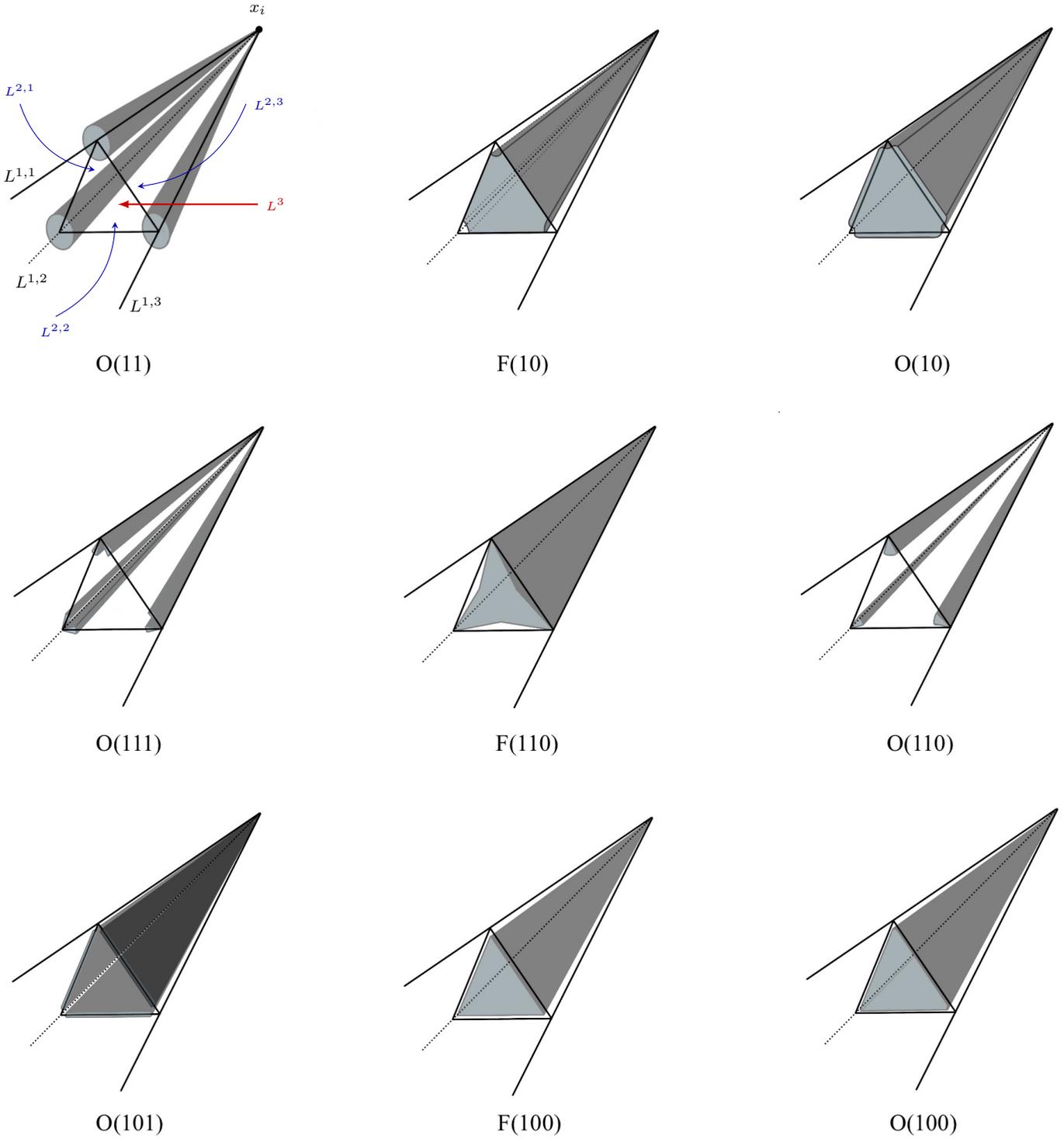
\begin{figure}[htbp]
    \centering
    \makebox[\textwidth][c]{ 
    \begin{tikzpicture}
        \node[anchor=south west, inner sep=0] (image) at (0,0) {
            \includegraphics[width=1.2\textwidth]{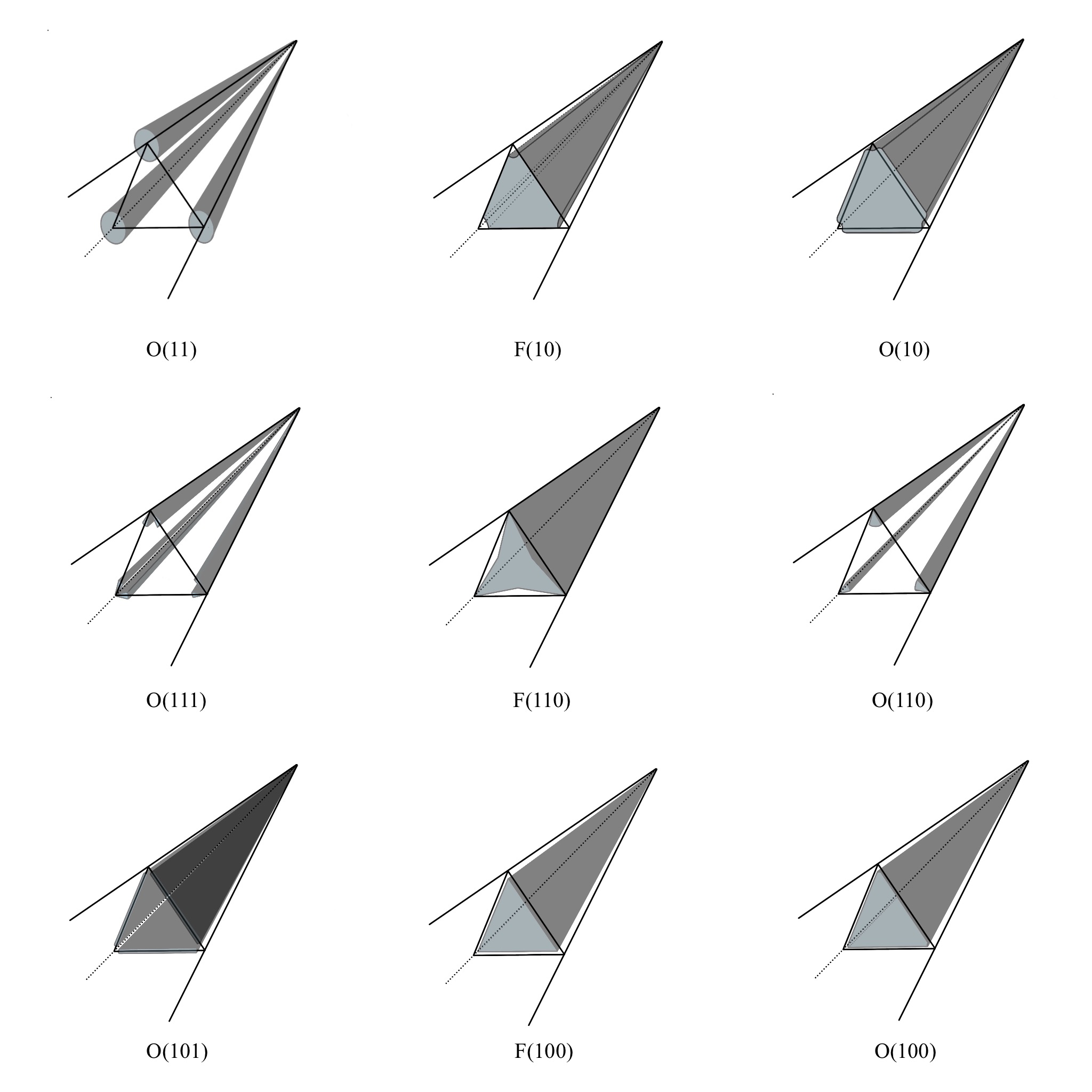}
        };

        \begin{scope}[x={(image.south east)},y={(image.north west)}]

            
            \filldraw (0.271, 0.963) circle (1.5pt);
            \node at (0.27, 0.98) {\scriptsize $x_i$};

            \node at (0.07, 0.84) {\scriptsize $L^{1,1}$};
            \node at (0.08, 0.75) {\scriptsize $L^{1,2}$};
            \node at (0.176, 0.73) {\scriptsize $L^{1,3}$};


             \draw[->, >=stealth, bend right, blue!70!black] (0.07, 0.9) node[above] {\tiny $L^{2,1}$} to (0.133, 0.844);
             \draw[->, >=stealth, bend right, blue!70!black] (0.1, 0.72) node[below] {\tiny $L^{2,2}$} to (0.15, 0.8);
             \draw[->, >=stealth, bend left, blue!70!black] (0.26, 0.9) node[right] {\tiny $L^{2,3}$} to (0.169, 0.83);

             \draw[->, >=latex,  red!80!black, thick, shorten >=3pt] (0.27, 0.815) node[right] {\tiny $L^3$} to (0.147, 0.815);

        \end{scope}
    \end{tikzpicture}
    }
    \caption[An example with $t=0$ and $u=3$]{\textbf{An illustrative example with $t=0$ and $u=3$.} 
    Here, $W_i$ is a 3-dimensional simple cone centered at $L^0=\{x_i\}$. The 1-spine $L^1=\cup_{l=1}^3 L^{1,l}$ consists of three rays from $x_i$, marked in black. The 2-spine $L^2=\cup_{l=1}^3 L^{2,l}$ consists of three branches; for instance, $L^{2,1}$ is the planar sector bounded by the rays $L^{1,1}$ and $L^{1,2}$, marked in blue. The 3-spine $L^3$ is the solid cone bounded by the planar sectors $L^{2,1}, L^{2,2},$ and $L^{2,3}$, marked in red. The gray regions represent the open sets $O(w)$ and $F(w)$ based on $W_i$ for $w \in \Lambda_2\cup \Lambda_3$. Note that these sets are unbounded cones, shown here truncated for visualization.}
\end{figure}

We will show in Proposition \ref{命题O(w)是覆盖} that the family $\{O^{(1/2)}(w)\}_{w\in \Lambda_u}$ indeed forms an open cover of the relevant conical neighborhood of $L^u$, where $\Lambda_u$ is the set of words ending with index $u$. 

\begin{proposition}\label{命题O(w)是覆盖}
    \begin{equation}\label{O(w)是覆盖}
    U^{u-1}_{u}(\sin\frac{\theta_{u-t,2}}{2})\subset \ccup_{w\in \Lambda_u} O^{(1/2)}(w).
    \end{equation}
\end{proposition}

\begin{proof}
    We first show that for each $k\in \{t,...,u-1\}$ and each $w=\delta_t...\delta_k\in\Lambda$,
    \begin{equation}\label{O(w)是覆盖step 1：局部覆盖结论}
        O(w) \cap U^{\mathcal{N}(w)}_u \left(\sin\frac{\theta_{k+1-t,2}}{2}\right) \subset O^{(1/2)}(w1) \cup O^{(1/2)}(w0).
    \end{equation}
    Take any $x$ in the left-hand side of \eqref{O(w)是覆盖step 1：局部覆盖结论}. If $x \in U^{\mathcal{N}(w)}_{k+1}(\sin({\theta_{k+1-t,1}}/{2}))$, then $x\in O^{(1/2)}(w1)$ by the definition. If $x \notin U^{\mathcal{N}(w)}_{k+1}(\sin({\theta_{k+1-t,1}}/{2}))$, then $\text{dist}(x, L^{k+1}) \geq (\sin{({\theta_{k+1-t,1}}/{2}})) \cdot \text{dist}(x, L^{\mathcal{N}(w)})$. At the same time, since $x\in U^{\mathcal{N}(w)}_u (\sin({\theta_{k+1-t,2}}/{2}))$, we have $\dd(x, L^u)<(\sin ({\theta_{k+1-t,2}}/{2}))\cdot \dd(x,L^{\mathcal{N}(w)})$.
    Let $z \in L^u$ be such that $\dd(x,L^u)=|x-z|$. Then $\dd(z,L^{\mathcal{N}(w)})\leq \dd(x,L^{\mathcal{N}(w)})+|x-z|<(1+\sin({\theta_{k+1-t,2}}/{2}))\cdot\dd(x,L^{\mathcal{N}(w)})$. Using the triangle inequality, we have $\dd(z, L^{k+1})\geq  \dd(x,L^{k+1})-|x-z|\ge ( \sin({\theta_{k+1-t,1}}/{2})-\sin({\theta_{k+1-t,2}}/{2}))\cdot ( 1+\sin({\theta_{k+1-t,2}}/{2}))^{-1}\cdot \dd(z,L^{\mathcal{N}(w)})>(\sin\beta_{k+1-t})\cdot\dd(z,L^{\mathcal{N}(w)})$.
   The last inequality holds due to the parameter choices in \eqref{theta,beta}. Thus $z \in F(w0)$, which implies $x \in O^{(1/2)}(w0)$. This concludes the proof of \eqref{O(w)是覆盖step 1：局部覆盖结论}.

    Using the local covering property \eqref{O(w)是覆盖step 1：局部覆盖结论}, we prove by induction that for each $k \in \{t, \dots, u-1\}$,
    \begin{equation}\label{O(w)是覆盖step 2:归纳假设}
    U^k_u\left(\sin \frac{\theta_{k-t+1,2}}{2}\right) \subset \bigcup_{w \in \Lambda_{k+1}} O^{(1/2)}(w).
    \end{equation}
    The base case $k=t$ follows directly from the local covering property \eqref{O(w)是覆盖step 1：局部覆盖结论} applied to $w=1$. Assume inductively that \eqref{O(w)是覆盖step 2:归纳假设} holds for the index $k-1$. That is,  $U^{k-1}_u\left(\sin ({\theta_{k-t,2}}/{2})\right) \subset \bigcup_{w \in \Lambda_{k}} O^{(1/2)}(w)$. 
    Intersecting both sides with $U^{k}_{u}(\sin ({\theta_{k+1-t,2}}/{2}))$, we get
    \begin{equation}\label{覆盖归纳假设的下一步}
         U^{k-1}_u(\sin \frac{\theta_{k-t,2}}{2})\cap U^{k}_{u}(\sin \frac{\theta_{k+1-t,2}}{2}) \subset \bigcup_{w\in \Lambda_k} \left( O^{(1/2)}(w)\cap U^{k}_{u}(\sin \frac{\theta_{k+1-t,2}}{2})\right),
    \end{equation}
    Since $L^{k-1} \subset L^k$ and $\theta_{k+1-t,2} < \theta_{k-t,2}$, we have $U^{k}_{u}(\sin ({\theta_{k+1-t,2}}/{2})) \subset U^{k-1}_u(\sin ({\theta_{k-t,2}}/{2}))$. Then the left-hand side in \eqref{覆盖归纳假设的下一步} is just $U^k_u(\sin({\theta_{k+1-t,2}}/{2}) )$. For the right-hand side, fix a word $w\in \Lambda_k$. Since $\mathcal{N}(w)\leq k$, we have $O^{(1/2)}(w) \cap U^{k}_{u}(\sin ({\theta_{k+1-t,2}}/{2})) \subset O(w) \cap U^{\mathcal{N}(w)}_u (\sin ({\theta_{k+1-t,2}}/{2})) $. And the local covering property \eqref{O(w)是覆盖step 1：局部覆盖结论} implies that $O(w) \cap U^{\mathcal{N}(w)}_u \left(\sin ({\theta_{k+1-t,2}}/{2}) \right) \subset O^{(1/2)}(w1) \cup O^{(1/2)}(w0)$. Taking the union over all $w \in \Lambda_k$, we obtain $U^{k}_{u}(\sin ({\theta_{k+1-t,2}}/{2})) \subset \bigcup_{w \in \Lambda_{k+1}} O^{(1/2)}(w)$. This confirms that the inductive claim holds for all integers $k$ from $t$ up to $u-2$. In particular, we have
    \begin{equation}
    U^{u-2}_u\left(\sin\frac{\theta_{u-1-t,2}}{2}\right) \subset \bigcup_{w\in \Lambda_{u-1}} O^{(1/2)}(w).
    \end{equation}
    Finally, we intersect both sides of this inclusion with $U^{u-1}_u(\sin ({\theta_{u-t,2}}/{2}))$. The left-hand side reduces to $U^{u-1}_u(\sin ({\theta_{u-t,2}}/{2}))$. For the right-hand side, for each $w$,  we have
    \begin{equation}
    O^{(1/2)}(w) \cap U^{u-1}_u(\sin \frac{\theta_{u-t,2}}{2}) \subset O(w) \cap U^{\mathcal{N}(w)}_u (\sin\frac{\theta_{u-t,2}}{2}) \subset O^{(1/2)}(w1).
    \end{equation}
    Combining these, we obtain the desired global covering \eqref{O(w)是覆盖}.
\end{proof}

\subsubsection{Define maps in $O(w)$ to get $\psi^m_i$}\label{对齐映射构造}

Now we begin to define a series of maps in $O(w)$ for $w\in \Lambda_u$ in a descending  binary order to construct $\eta^u_i$ and extend them in a proper way.  We denote by $1^s$ the word $1...1$ that has $s$ letters and composed by 1, and the same for 0. We shall use $\zeta^{u,\delta_t...\delta_s}_1$ to represent a map defined in $O(\delta_t...\delta_s)$. And we shall use $\zeta^{u,\delta_t...\delta_s}$ to represent a map which is an extension of $\zeta^{u,\delta_t...\delta_s}_1$ to  $O(\delta_t...\delta_{s-1})$.

Before proceeding to the specific construction, we briefly outline the method used to define these maps. As seen from the definition of $\{O(w)\}_{w\in \Lambda}$, the relationship between these sets presents a tree structure (refer to Figure \ref{fig tree}). The root node $1\in\Lambda_t$ corresponds to the set $O(1)=\mathbb{R}^N$. From $\Lambda_t$ to $\Lambda_{u-2}$, each element $w$ in the word set has two leaves: a left leaf $1$ and a right leaf $0$, corresponding to the words $w1, w0$ and open sets $O(w1), O(w0)$ respectively. Words in $\Lambda_{u-1}$ have only one leaf $1$, and words in $\Lambda_u$ have no leaf. We define the maps in the open sets corresponding to this tree from bottom to top and from left to right. This definition process is a linear flow, meaning there is a strict sequential order between any two steps.

We first describe the rules of definition without precise formulations.
For $w\in \Lambda_u$, we define a map on $O(w)$, denoted by $\tau_1^{u,w}$. The construction of this map relies on the family of maps $\{\tau_1^{u,\tilde{w}}: \tilde{w}\in \Lambda_u, \tilde{w}>w \text{ in binary}\}$. The explicit definition of $\tau_1^{u,w}$ will be provided subsequently.
When $w\notin \Lambda_{u}$, we denote the map to be defined in $O(w)$ as $\zeta_1^{u,w}$. When we need to define $\zeta_1^{u,w}$, we assume that the maps in the open sets corresponding to the leaves of $w$ have already been defined. When $w\in \Lambda_{u-1}$, by assumption, $\tau_1^{u,w1}$ has been defined. We define the map $\zeta_{1}^{u,w}$ in $O(w)$ as the extension of $\tau_1^{u,w1}$ to $O(w)$.  
When $w\in \Lambda_s$ with $t\leq s<u-1$, we define the map $\zeta_1^{u,w}$ in $O(w)$ as the composite map $\zeta^{u,w0}\circ \zeta^{u,w1}$, where $\zeta^{u,w\delta}$ ($\delta \in \{0,1\}$) is the extension of the map $\zeta_1^{u,w\delta}$ from $O(w\delta)$ to $O(w)$.

We initiate the entire construction by defining $\tau_1^{u,1^{u-t+1}}$ on the open set $O(1^{u-t+1})$, which corresponds to the leftmost and bottommost node $1^{u-t+1}$. By iterating this procedure, we ultimately obtain the global map $\zeta_1^{u,1}$. We then define $\eta_i^u := \zeta_1^{u,1}\circ \eta_i^{u-1}$. Given the finiteness of $\Lambda$, this process terminates in finitely many steps, progressively aligning $\eta_i^{u-1}(\Gamma^u_k)$ onto $L^u$.

\textbf{For $\bm{w=1^{u-t}}$.}

Pick the first word $\delta_t...\delta_{u-1}=1^{u-t}$ which is such that $\delta_t=...=\delta_{u-1}=1$. We construct $\zeta^{u,1^{u-t}}$ by first defining a map $\zeta^{u,1^{u-t}}_1$ on $O(1^{u-t})$ and then extending it to $O(1^{u-t-1})$.

\noindent\textbf{Definition of $\bm{\zeta^{u,1^{u-t}}_1}$ in $\bm{O(1^{u-t})}$.} As a preparation, we define a map 
$\tau^{u,1^{u-t+1}}_1$ in $O(1^{u-t+1})$ to map points on $\eta^{u-1}(\Gamma^u_k)$ to $L^u$ in $O(1^{u-t+1})$. Recall we have assumed that $\eta^{u-1}$ maps $\Gamma^{t}_k\to L^t$,...,$ \Gamma^{u-1}_k \to L^{u-1}$ in $4.9 B_i$ with $|D\eta^{u-1}-I|<C\varepsilon$ in $4.9B_i\backslash \Gamma^{u-1}_k$ , where $C$ depends on $\alpha,C_{t,2}$, if $u>t+1$, also on $C_{t+1,4},..., C_{u-1,4}$. Since $ u<m$, the induction hypothesis of (M3) for $u$ says that each $\Gamma^{u,l}_k$ is a $C_{u,4}\varepsilon$-Lipschitz graph over $D^{u,l}$ in $5B_i$, therefore,  $\eta^{u-1} (\Gamma^{u,l}_k)$ is also a $C\varepsilon$-Lipschitz graph over $ L^{u,l}$ in $4.99B_i$, where $C$ depends on $C_{t,2},C_{t+1,4},...,C_{u,4},\alpha$. Denote by $ \tilde{\varphi}^{u,l,1^{u-t+1}} :L^u\to (L^u)^\perp $ this $C\varepsilon$-Lipschitz map.  By Case A\ref{O(w1)内退化为type k} in Lemma \ref{O(w)的性质}, every connected component of $O(1^{u-t+1})$ only meets one branch of $L^u$. Precisely speaking, $O(1^{u-t+1})\cap U^{u-1}_{u,l}(\sin\theta_{u-t,1})\cap L^{u,l}\neq\emptyset$ and $O(1^{u-t+1})\cap U^{u-1}_{u,l}(\sin\theta_{u-t,1})\cap L^{u,l'}=\emptyset $ for every $l'\neq l$.
Let $\tilde{\varphi}^{u,1^{u-t+1}} = \tilde{\varphi}^{u,l,1^{u-t+1}} $ and $\pi^u=\pi^{u,l}$ in each $ O(1^{u-t+1})\cap U^{u-1}_{u,l}(\sin\theta_{u-t,1})$. 
Set
\begin{equation}
    \tau^{u,1^{u-t+1}}_1 (x)=x-\tilde{\varphi}^{u,1^{u-t+1}} (\pi^u(x)) \text{ in } O(1^{u-t+1}).
\end{equation}
Then $\tau^{u,1^{u-t+1}}_1$ maps $\eta^{u-1} (\Gamma^u_k)$ to $L^u$ in $O(1^{u-t+1})$. 

Now we are ready to define $\zeta^{u,1^{u-t}}_1$ by extending $\tau^{u,1^{u-t+1}}_1$ to $O(1^{u-t})$.
Let $V(1^{u-t+1})=O(1^{u-t})\backslash \overline{U}^{u-1}_u(\sin\frac{\theta_{u-t,1}}{2})$
and set ${\tau}^{u,1^{u-t+1}}_2=id$ on $V(1^{u-t+1})$, then $\{O(1^{u-t+1}),V(1^{u-t+1})\}$ is an open cover of $O(1^{u-t})\backslash L^{u-1}$ and there is a partition of unity subordinated to $\{O(1^{u-t+1}),$ $V(1^{u-t+1})\}$. We call them $\{\mu^{1^{u-t+1}}_1,\mu^{1^{u-t+1}}_2\}$. Precisely speaking, $ \sum_{j=1}^2 \mu^{1^{u-t+1}}_j=1$ on $O(1^{u-t})\backslash L^{u-1}$. Moreover, we have supp$(\mu^{1^{u-t+1}}_1)\subset O(1^{u-t+1})$ and supp$(\mu^{1^{u-t+1}}_2)\subset V(1^{u-t+1})$. We can ask that $\mu^{1^{u-t+1}}_j$ is $C^1$ in $O(1^{u-t})\backslash L^{u-1}$ for $j=1,2$. By Case A\ref{O(w1)内退化为type k} in Lemma \ref{O(w)的性质}, for each $x\in O(1^{u-t})\backslash L^{u-1}$ and $j=1,2$, we can ask that $|\nabla\mu^{1^{u-t+1}}_j(x)|<C\cdot\dd(x, L^{u-1})^{-1}$, here $C$ depends only on $\alpha,n$, thus, it is a geometric constant. Let 
$   \zeta^{u,1^{u-t}}_1=\sum_{j=1}^2 \mu^{1^{u-t+1}}_j\cdot \tau^{u,1^{u-t+1}}_j
$ in $ O(1^{u-t})\backslash L^{u-1}$ 
and let $\zeta^{u,1^{u-t}}_1=id$ on $O(1^{u-t})\cap L^{u-1}$. Then we have a continuous map $\zeta^{u,1^{u-t}}_1$ in $O(1^{u-t})$ which is such that 
\begin{gather}\label{导数接近1}
    |D\zeta^{u,1^{u-t}}_1-I|<C\varepsilon \text{ in } 4.9 B_i\cap O(1^{u-t})\backslash L^{u-1},\\
    \label{移动距离小}
    |\zeta^{u,1^{u-t}}_1(x)-x|< C\varepsilon \dd(x,L^{u-1})\enspace\text{for all }x\in O(1^{u-t}),
\end{gather}
where the first $C$ depends on $C_{t,2},C_{t+1,4},...,C_{u,4},\alpha$ and the second $C$ depends on $C_{t,3},$ $C_{t+1,5},...,C_{u,5}$.
Furthermore, $\zeta^{u,1^{u-t}}_1$ maps $\eta^{u-1} (\Gamma^u_k)$ to $L^{u}$ in $O(1^u)\cap 4.9 B_i$.

\noindent\textbf{An extension of $\bm{\zeta^{u,1^{u-t}}_1}$ to $\bm{O(1^{u-t-1})}$.} Then we extend $\zeta^{u,1^{u-t}}_1$ to $O(1^{u-t-1})$ to get $\zeta^{u,1^{u-t}}$ by the same method as for extending $\tau^{u,1^{u-t+1}}_1$ to get $\zeta^{u,1^{u-t}}_1$.
At the same time, similar properties as (\ref{导数接近1}) and (\ref{移动距离小}) still hold, that is, $|D\zeta^{u,1^{u-t}}-I|<C\varepsilon$ in $4.9 B_i\cap O(1^{u-t-1})\backslash L^{u-2}$ and $ |\zeta^{u,1^{u-t}}(x)-x|<C\varepsilon \dd(x,L^{u-2})$, where the first $C$ depends on $C_{t,2},C_{t+1,4},...,C_{u,4},\alpha$ and the second $C$ depends on $C_{t,3},C_{t+1,5},...,C_{u,5}$. We end the process for $1^{u-t}$.

\textbf{For $\bm{w=1^{u-t-1}}0$.}

Next we consider $1^{u-t}-1=1^{u-t-1}0$. What we need to do is to define $\zeta^{u,1^{u-t-1}0}_1$ in $O(1^{u-t-1}0)$ and extend it to $O(1^{u-t-1})$. Next, we composite $\zeta^{u,1^{u-t-1}0}$ and $\zeta^{u,1^{u-t-1}1}$ in $O(1^{u-t-1}) $ to get a new map and proceed to extend it.
As a preparation, we define $\tau^{u,1^{u-t-1}01}_1$ in $O(1^{u-t-1}01)$ and extend it to $O(1^{u-t-1}0)$. 
 By Case A\ref{O(w1)等价定义} and \ref{O(w1)内退化为type k} in Lemma \ref{O(w)的性质}, every component $O(1^{u-t-1}01)\cap U^{\mathcal{N}(1^{u-t-1}0)}_{u,l}(\sin\theta_{u-t,1})$ only meets $L^{u,l}$. 
Since $|D\zeta^{u,1^{u-t}}-I|<C\varepsilon$, $\zeta^{u,1^{u-t}}\circ \eta^{u-1} (\Gamma^{u,l}_k)$ is still a $C\varepsilon$-Lipschitz graph over $L^{u,l}$ in $O(1^{u-t-1}01)\cap U^{\mathcal{N}(1^{u-t-1}0)}_{u,l}(\sin\theta_{u-t,1})$, where $C$ depends on  the constants $C_{t,2},C_{t+1,4},\dots,C_{u,4},\alpha$. 
Denote by  $\tilde{\varphi}^{u,1^{u-t-1}01}$ this Lipschitz map
and let $\pi^u=\pi^{u,l}$ in every $ O(1^{u-t-1}01)\cap U^{\mathcal{N}(1^{u-t-1}0)}_{u,l}(\sin\theta_{u-t,1})$ and let
\begin{equation}
    \tau^{u,1^{u-t-1}01}_1(x)=x-\tilde{\varphi}^{u,1^{u-t-1}01}(\pi^u (x)).
\end{equation} 
Then let 
$
    V(1^{u-t-1}01)=O(1^{u-t-1}0)\backslash \overline{U}^{u-2}_{F(1^{u-t-1}0)}(\sin\frac{\theta_{u,1}}{2})
$
and let $\eta^{u,1^{u-t-1}01}_2=id$ on $V(1^{u-t-1}01)$. Therefore, $\{O(1^{u-t-1}01),V(1^{u-t-1}01)\}$ is an open cover of $O(1^{u-1}0)\backslash L^{u-2}$. By using the same method as for defining $\zeta^{u,1^{u-t}}_1$, we get $\zeta^{u,1^{u-t-1}0}_1$ defined in $O(1^{u-t-1}0)$ with 
\begin{gather}
        |D\zeta^{u,1^{u-t-1}0}_1-I|<C\varepsilon \text{ in } 4.9 B_i\cap O(1^{u-t-1}0)\backslash L^{u-2}\\
        | \zeta^{u,1^{u-t-1}0}_1(x)-x|<C\varepsilon\dd(x,L^{u-2})\text{ for all }x\in O(1^{u-1}0) ,
\end{gather}
where the first $C$ depends on $C_{t,2},C_{t+1,4},...,C_{u,4},\alpha$ and the second $C$ depends on $C_{t,3}$, $C_{t+1,5},...,C_{u,5}$. Furthermore, $\zeta^{u,1^{u-t-1}01}$ maps $\zeta^{u,1^{u-t}}\circ\eta^{u-1}(\Gamma^{u}_k)$ to $L^u$ in $O(1^{u-t-1}0)\cap 4.9B_i$. In the same way, we can extend $\zeta^{u,1^{u-t-1}0}_1$ to $O(1^{u-t-1})$ and then get $\zeta^{u,1^{u-t-1}0}$ defined in $O(1^{u-t-1})$. Now we have defined $\zeta^{u,1^{u-t}}$ and $\zeta^{u,1^{u-t-1}0}$ in $O(1^{u-t-1})$. Let 
\begin{equation}
    \zeta^{u,1^{u-t-1}}_1=\zeta^{u,1^{u-t-1}0}\circ \zeta^{u,1^{u-t}}.
\end{equation}
It has similar properties as in (\ref{导数接近1}) and (\ref{移动距离小}). Then we extend it to $O(1^{u-t-2})$ to get $\zeta^{u,1^{u-t-1}}$. And we turn to next case for $1^{u-t-1}0-1=1^{u-t-2}01$.

\textbf{For $\bm{w=\delta_t...\delta_{u-1}}$.}

Set $w=\delta_t...\delta_{u-1}$ and assume that we need to define $\zeta^{u,w}_1$ in $O(w)$ now. That is, for all $w'>w$ in binary, suppose that $w'=\delta_t'...\delta_{u-1}'$ and set $s'=\mathcal{N}(w')$, we have end the construction of $\zeta^{u,w'}$ in $O(\delta_t'...\delta_{u-2}')$, $\zeta^{u,\delta_t'...\delta_{u-2}'}$ in $O(\delta_t'...\delta_{u-3}')$,... and $\zeta^{u,\delta_t'...\delta_{s'}'}$ in $O(\delta_t'...\delta_{s'-1}')$. Then we are ready to  define $\zeta^{u,w}_1$ in $O(w)$. 
Let $e_1=\max\{j:\delta_j=0,w=\delta_t...\delta_{u-1}\}$ and $w_1=\delta_t...\delta_{e_1-1}1$. Let $e_2=\max\{j:\delta_j=0,w_1=\delta_t...\delta_{e_1-1}1\}$ and $w_2=\delta_t...\delta_{e_2-1}1$.
Since $w$ has finitely many letters, we can end the process with $e_v$ and $w_v$. And $\zeta^{u,w_1}\circ ...\circ \zeta^{u,w_v} (\Gamma^u_k)$ is still a $C\varepsilon$-Lipschitz graph of $L^u$ in $O(\delta_t\dots \delta_{e_1-1})$, where $C$ depends on $C_{t,2},C_{t+1,4},...,C_{u,4},\alpha$. Denote by $\tilde{\varphi}^{u,w1}$ this $C\varepsilon$-Lipschitz map. Then  for each $x\in O(w1)$, let
\begin{equation}
    \tau^{u,w1}_1(x)=x-\tilde{\varphi}^{u,w1}(\pi^u (x)).
\end{equation}
Then we extend $\tau^{u,w1}_1$ to $O(w)$ and get $\zeta^{u,w}_1$. We proceed to extend $\zeta^{u,w}_1$ to $O(\delta_t...\delta_{u-2})$ (where $w=\delta_t...\delta_{u-1}$) and get $\zeta^{u,w}$.

If $\delta_{u-1}=1$, then we end the process for $w=\delta_t...\delta_{u-1}$ and begin to consider $\zeta^{u,w-1}$. If $\delta_{u-1}=0$, let $\zeta^{u,\delta_t...\delta_{u-2}}_1$ equals to $\zeta^{u,w}\circ \zeta^{u,\delta_t...\delta_{u-2}1}$. Then extend $\zeta^{u,\delta_t...\delta_{u-2}}_1$ to $O(\delta_t...\delta_{u-3})$ to get $ \zeta^{u,\delta_t...\delta_{u-2}}$. And consider $\delta_{u-2}=1$ or 0 and repeat the operation as what we do when consider $\delta_{u-1}=1$ or 0. Continue this process and it will terminates within a finitely many steps. Then we continue to consider $\zeta^{u,w-1}$, where the subtraction in $w-1$ is in binary.

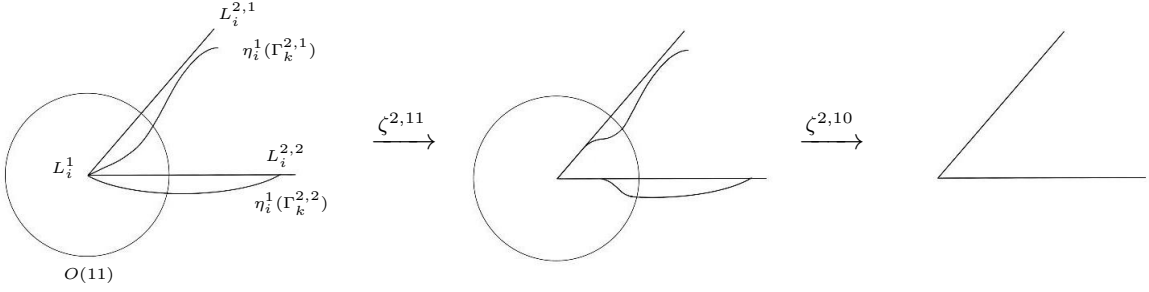
\begin{figure}[H]
    \centering
    \begin{overpic}[scale=0.5]{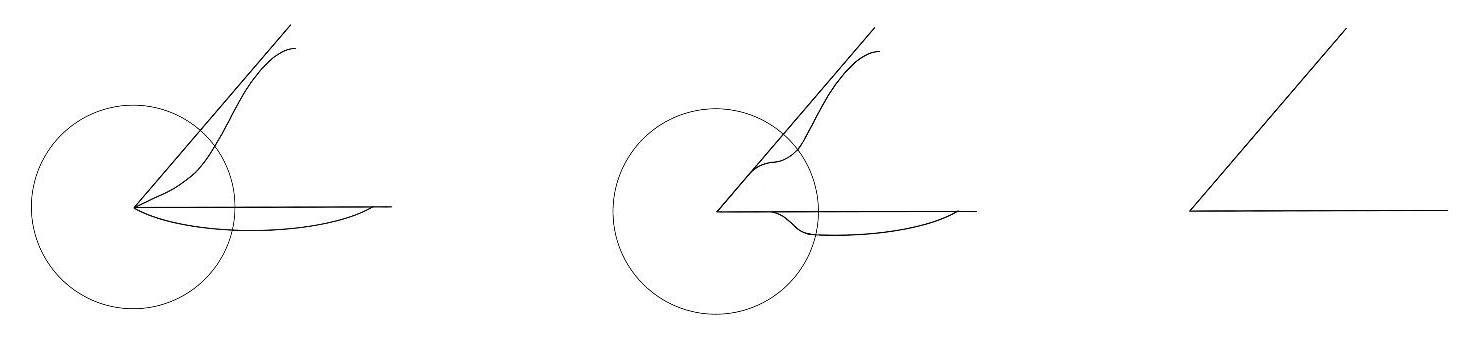}
    \put(6,9){\tiny $L^1_i$}
    \put(20,22){\tiny $L^{2,1}_i$}
    \put(22,19){\tiny $\eta^1_i(\Gamma^{2,1}_k)$}
    \put(24,10){\tiny $L^{2,2}_i$}
    \put(23,6){\tiny $\eta^1_i(\Gamma^{2,2}_k)$}
    \put(7,0){\tiny $O(11)$}
    \put(32,11){ $\xrightarrow{\zeta^{2,11}}$ }
    \put(68,11){ $\xrightarrow{\zeta^{2,10}}$ }
    \end{overpic}
    \caption{Examples of $\zeta^{u,w}$ when $t=0$, $u=2$ and $w=11,10$ (assume that we are looking in the direction orthogonal to $L^1_i$)}
    \label{fig2}
\end{figure}

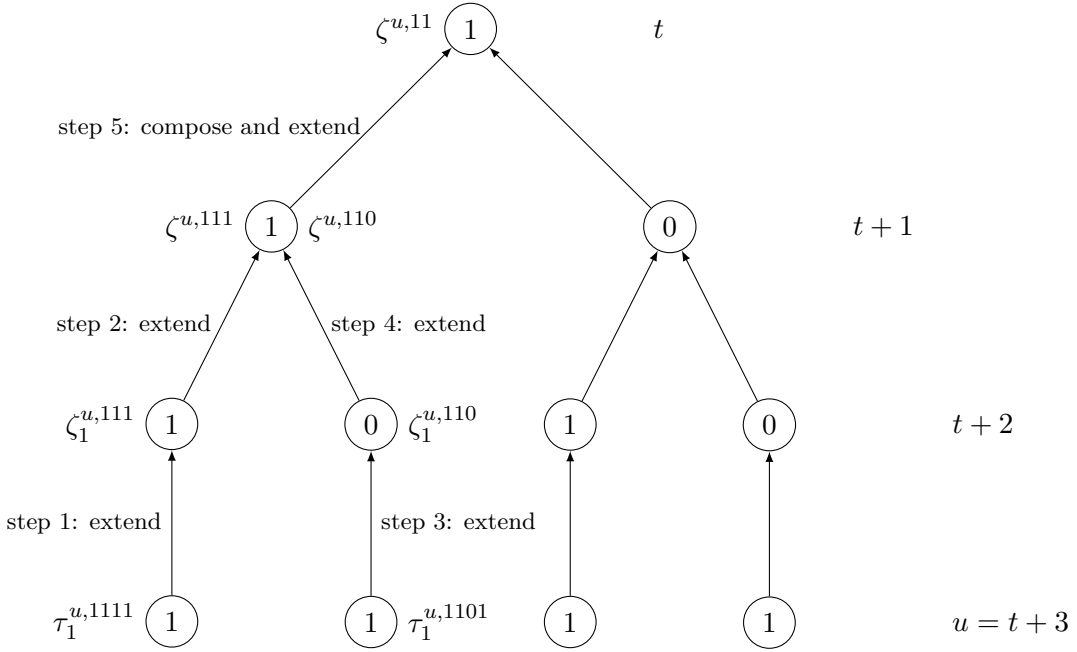
\begin{figure}[H]
\centering
\begin{forest}
for tree={
  circle, 
  draw, 
  edge={<-,>=latex}, 
  s sep=5em,  
  l sep=5em,  
  edge label/.wrap pgfmath arg={%
    node [midway, right, font=\scriptsize, xshift=1pt] {\text{#1}}
  }{content()},
}
[1, label={left:$\zeta^{u,11}$}, label={[xshift=5em]right:$t$}
    [1,label={left:$\zeta^{u,111}$},label={right:$\zeta^{u,110}$}, edge label={node[midway, left, font=\footnotesize]{step 5: compose and extend}} 
        [1, label={left:$\zeta^{u,111}_1$}, edge label={node[midway, left, font=\footnotesize]{step 2: extend}}
            [1, label={left:$\tau_1^{u,1111}$}, edge label={node[midway, left, font=\footnotesize]{step 1: extend}}]
        ]
        [0, label={right:$\zeta_1^{u,110}$}, edge label={node[midway, right, font=\footnotesize]{step 4: extend}}
            [1, label={right:$\tau_1^{u,1101}$}, edge label={node[midway, right, font=\footnotesize]{step 3: extend}}]
        ]
    ]
    [0, label={[xshift=5em]right:$t+1$}
        [1
            [1]
        ]
        [0, label={[xshift=5em]right:$t+2$}
            [1, label={[xshift=5em]right:$u=t+3$} ]
        ]
    ]
]
\end{forest}
\caption{An example of the construction of $\zeta^{u,w}$ in binary when $u=t+3$}
\label{fig tree}
\end{figure}

At last, we get $\zeta^{u,1}_1$ in $\RR^N$ and then we set $\eta^u_i=\zeta^{u,1}_1\circ \eta^{u-1}_i$. According to the discussion above, we have the properties that $\eta^u_i$ maps $\Gamma^t_k \to L^t_i$,..., $\Gamma^u_k \to L^u_i$ in $4.9 B_i$. Furthermore, $\eta^u_i$ is continuous in $\RR^N$ and is $C^1$ in $4.9 B_i \backslash L^{u-1}_i$ with $|D\eta^u_i-I|<C\varepsilon$. When  {$u=m-1$}, 
we get the map  $\eta^{m-1}_i$ for $i\in I_t(k)$, which is such that
\begin{equation}\label{eta的4个性质}
    \begin{split}
        &|\eta^{m-1}_i(x)-x|<C\varepsilon r_i \text{ for }x\in \RR^N,\\
        &\eta^{m-1}_i \text{ is continuous in }\RR^N,  \\
        &\eta^{m-1}_i\text{ is }C^1 \text{ in }4.9 B_i\backslash L^{m-1}_i\text{ with }|D\eta^{m-1}_i - I|<C\varepsilon,\\
        &\eta^{m-1}_i(\Gamma^{u}_k)=L^u_i\text{ in }4.9 B_i\text{ for }t\leq u\leq {m-1},
    \end{split}
\end{equation}
where $C>0$ depends on $C_{t,2}, C_{t,3}, \{ C_{ w,e}\}_{t\leq w\leq u,e=4,5},\alpha$. 

And then {   we define the map $h_i$ analogously to the construction of $\zeta_1^{u,1}$. Specifically, let $\Lambda=\cup_{k=t}^m \Lambda_k$ be defined as in \eqref{Lambda定义} with $u=m$, and let $\{O(w)\}_{w\in \Lambda}$ be the corresponding open sets defined as in \eqref{O(w)的归纳定义} and \eqref{F(w0)}. For each $w\in \Lambda_m$, Lemma \ref{O(w)的性质} Case A\ref{O(w1)=不交并} implies that $O(w)$ is the disjoint union of the refined sets $\{O(w,s)\}_s$.  Furthermore, by Case A\ref{O(w1)内退化为type k}, $L^m\cap O(w,s)=L^{m,s}\cap O(w,s) = P^{m,s}\cap O(w,s)$ for each index $s$. Accordingly, for each $w\in \Lambda_m$, we define the local map $h_1^w : O(w)\to L^m$ by setting  $h_1^w(x)=\pi^{m,s}(x)=\ppi^{m,s}(x)$ for each $x\in O(w,s)$ and each index $s$. Having established the maps on the terminal sets $\{O(w)\}_{w\in \Lambda_m}$, we extend and compose them to obtain the global map $h_i$ on $O(1)=\RR^N$ finally.  By construction, $h_i$ is continuous in $\RR^N$ and $C^1$ in $\RR^N\backslash L^{m-1}$. Moreover, if a point $x$ is not contained in  $O(w)$ for any $w\in \Lambda_m$, then $h_i(x)=x$. In particular, $h_i=id$ on $L^{m-1}$. 
By Proposition \ref{O(w)是覆盖}, it follows that for each $w\in \Lambda_m$ and for each index $s$,}
\begin{equation}\label{h的定义}
    h_i(x)=\pi^{m,s}_i(x)\enspace\text{when } \dd(x,L^{m,s}_i)<(\sin\frac{\theta_{m-t,2}}{2})\cdot \dd(x,L^{m-1}_i).
\end{equation}
Finally, set $\psi_i^m: 4.9 B_i\to \RR^N$ by
\begin{equation}
    \psi^m_i=h_i\circ \eta^{m-1}_i.
\end{equation}
It is clear that 
\begin{equation}\label{psi^m-1和psi^m在m-1维graph上相等}
    \psi^{m-1}_i(x)=\psi^m_i(x)\text{ for each }x\in \Gamma^{m-1}_k\cap 3 B_i.
\end{equation}
Also, by (\ref{eta的4个性质}), (\ref{h的定义}), the induction hypothesis of (M2),(M3),(M4), we have
\begin{equation}\label{psi的正则性}
    \begin{split}
    &\psi^m_i\text{ is continuous on }\Gamma^m_k\cap 4.9 B_i,\\
    &\psi^m_i\text{ is }C^1 \text{ on }(\Gamma^{m}_k\backslash \Gamma^{m-1}_k)\cap 4.9 B_i.
    \end{split}
\end{equation}

\subsection{Proof of (M1)-(M4) for the base case $k=0$ in dimension $m$}

{  For each $x\in \RR^N$, let $I_x=\{i\in I(k):\theta_i(x)\neq 0\}$.}

By (\ref{fk,gk}), we can estimate that
\begin{equation}\label{g^m_k粗略估计}
    |g^m_k(x)-x|<n_0^{-n}\cdot 2^{-n-90}\text{ for }x\in \RR^N.
\end{equation}
This is because $|g^m_k(x)-x|\leq \max_{i\in I_x} |\psi^m_i(x)-x|\leq |\psi^m_i(x)-\eta^{m-1}_i(x)|+|\eta^{m-1}_i(x)-x|<
8r_i$ and $r_i\leq n_0^{-n}\cdot 2^{-k-100}$. Now we are ready to check that (M1)-(M4) hold.

First we prove that (M1)-(M4) hold for $k=0$. Since $\Gamma_0^m=\Gamma^m= L^m\cap B(0,\rho^m_0)$ by (\ref{Gamma^m定义}) and $d_{0,1.98}(L^m,E_m)<C_6\varepsilon$ by (\ref{1.98尺度}), for each $x\in \Gamma^m_0\cap B(0,\rho^m_0)$, there is $y\in E_m$ such that $|x-y|<1.98C_6\varepsilon$. Set $C_{m,1}\geq 1.98C_6$, then (M1) holds for $k=0$.

When $i\in I_m(0)$, (\ref{取出来的锥每一层都接近}) implies that $W_i\in \A(m)$ and $L^m_i$ is an $m$-plane passing through $x_i$ such that $d_{x_i,100r_i}(L^m_i,E_m)<C_5\varepsilon$. Also, (\ref{级别低的球离级别高的点远}) implies that $\dd(10^3n_0 B_i,E_{m-1})\gg 10^4 r_i$. Thus $10^3 n_0 B_i$ does not meet $L^{m-1}$ and we can find $Z\in \cup_{s=m}^{n} \A(s)$ such that $Z\cap 10^3 B_i=Z_0\cap 10^3 B_i$ by Proposition~\ref{用次一级的锥替代}. On the other hand, since $x_i\in E_m\cap B(0,1.98)$, we have $\dd(x_i,L^m)<1.98 C_6\varepsilon$ by (\ref{1.98尺度}), which means that $L^m \cap B_i\neq \emptyset$. As a consequence, $Z\in \A(m)$ and $L^m(Z)\cap 10^3 B_i=L^m\cap 10^3 B_i$. Since $L^m(Z)$ is an $m$-plane, there is only one branch of $L^{m}$ intersecting with $10^3 B_i$. Denote by $L^{m,l}$ this branch, we get that $\Gamma^m_0\cap 10^3 B_i=\Gamma^{m,l}_0\cap 10^3 B_i=L^{m,l}\cap 10^3 B_i$ and $\dd(x_i,\Gamma^{m,l}_0)<1.98C_6\varepsilon$. Set $C_{m,3}>1.98 C_6 $, then we get that $\Gamma^{m,l}_0$ is the unique branch meeting $5B_i$ and $\Gamma^{m,l}_0\cap B(x_i,C_{m,3}\varepsilon)\neq\emptyset$. In addition, we have $d_{x_i,20r_i}(L^m_i,L^{m,l})<(\frac{C_6}{10 r_i}+5C_5)\varepsilon$. Let $C_{m,2}>\frac{C_6}{10 r_i}+5C_5$ and let $G^m_i=L^{m,l}$, then (M2) holds for $k=0$.

When $i\in I_t(0)$, where $ 0\leq t\leq m-1$, we have $x_i\in E_t$. When $t=0$, $x_i=0$. When $t\geq 1$, (\ref{级别低的球离级别高的点远}) implies that  $\dd(10^3n_0 B_i, E_{t-1})\gg 10^4 r_i$. By (\ref{取出来的锥每一层都接近}), when $t\geq 1$, $10^3 n_0 B_i$ does not meet $L^{t-1}_i$ and there is $Z\in \ccup_{s=t}^{n} \A(s)$ such that $Z\cap 10^3 B_i=Z_0\cap 10^3 B_i$ by Proposition \ref{用次一级的锥替代}. Since $x_i\in E_t$, $L^t$ meets $B_i$ and therefore $Z\in \A(t)$. By remark in Lemma \ref{lem:球近则对应的锥接近}, the number of branches of $L^m$ that meets $20B_i$ is the same with the number of branches of $L^m_i$. And there is a constant $C>0$ depending on $C_5$ and $C_6$ such that, for each $L^{m,l}$ that meets $20B_i$, there is a unique branch  $L^{m,l}_i \subset  L^m_i$ such that $d_{x_i,20r_i}(L^{m,l},L^{m,l}_i)<C\varepsilon$. 
{  In addition, there exists a one-to-one correspondence between the $(m-1)$-boundaries of $L^{m,l}$ and the $(m-1)$-boundaries of $L_i^{m,l}$, as in (\ref{球近则锥近的式子2}).} 
Let $G^{m,l}_i=L^{m,l}$ and $D^{m,l}_i=\ppi^{m,l}_i(L^{m,l})$, where $\ppi^{m,l}_i$ is the orthogonal projection onto the $m$-plane that contains $L^{m,l}_i$, then we have
\begin{equation}\label{k=0时的M3}
\begin{split}
d_{x_i,(5+1/500)r_i}(D^{m,l}_i,L^{m,l}_i)<4C\varepsilon
\end{split}
\end{equation}
for some geometric constant $C$. By Lemma \ref{面的dx,r决定夹角和距离}, (\ref{Gamma^m定义}) and (\ref{k=0时的M3}), (M3) for $k=0$ holds.

 Since $f^m_0=id$, the case when $k=0$ for (M4) is clear.

\subsection{Proof of (M1)-(M4) for the inductive step $k+1$ in dimension $m$}

Now suppose that (M1)-(M4) hold from step $0$ to $k$. We continue to prove that they hold for $k+1$. For each $x\in\Gamma^m_k\cap B(0,\rho^m_k)$, let $I_x=\{i\in I(k):\theta_i(x)\neq 0\}$. By induction hypothesis of (M1) for dimension $m$ and step $k$, $\dd(x,E_m)<C_{m,1}\varepsilon 2^{-k}\ll r_j/2$ for $j\in I_m(k)$. 
{Combining} with (\ref{级别低的球离级别高的点远}), 
we have $I_x\subset \ccup_{t=0}^m I_t(k)$.

We first show that different maps for a fixed point are 
{  sufficiently close} in Lemma \ref{lem for m:对同一个点映射之后距离相差很小}.

\begin{lemma}\label{lem for m:对同一个点映射之后距离相差很小}
    Fix $x\in\Gamma^m_k$ and suppose that $x_i\in B(0,\rho^m_k)$ for all $i\in I_x$. Then we have
    \begin{equation}\label{for m:对同一个点映射之后距离相差很小}
        |\psi^m_i(x)-\psi^m_j(x)|<C_8\varepsilon 2^{-k},\text{ for all }i,j\in I_x
    \end{equation}
    where $C_8>0$ is a constant only depending on $n_0,\delta_0,n,\alpha$ and $\{C_{u,e}\}_{0\leq u\leq m-1,1\leq e\leq 5}$.
\end{lemma}

\begin{proof}
    Let $ J_x=\min\{s:I_x\cap I_s(k)\neq \emptyset\}$ and suppose that $x\in\Gamma^{m,l}_k$ for some $l$.

    When $J_x=m$, $ I_x\subset I_m(k)$. Fix an index $j_x\in I_m(k)\cap I_x$, then for all $i\in I_x$, we have $3B_i\cap 3B_{j_x}\neq\emptyset$ and therefore, $d_{x_i,20r_i}(L^m_i,L^m_{j_x})<C_7 \varepsilon$ by Lemma \ref{lem:球近则对应的锥接近}. Since $x_i\in L^m_i$, we get that $|\pi^m_i(x_i)-\pi^m_{j_x}(x_i)|<20 C_7\varepsilon r_i$ and $|\pi^m_i (x)-\pi^m_{j_x}(x)|<100 n C_7\varepsilon r_i$ by Lemma \ref{面的dx,r决定夹角和距离}. In $3B_i$,  $\psi^m_i=\pi^m_i$,  so $|\psi^m_i(x)-\psi^m_j(x)|\leq |\psi^m_i(x)-\psi^m_{j_x}(x) |+|\psi^m_j(x)-\psi^m_{j_x}(x)|<200n C_7\varepsilon r_i$. Let  $C_8>0$ be such that  $200n C_7\varepsilon r_i<C_8\varepsilon 2^{-k}$, then $C_8$ only depends on $n,n_0,\delta_0$ and (\ref{for m:对同一个点映射之后距离相差很小}) follows.

    When $J_x<m$, pick $j_x\in I_x\cap I_{J_x}(k)$ and consider $|\psi^m_i(x)-\psi^m_{j_x}(x)|$ for each $i\in I_x$ such that $i\neq j_x$.
    There is a unique branch of $L^m_{j_x}$, denoted as $L^{m,l}_{j_x}$, such that $\Gamma^{m,l}_k$ is a $C_{m,4}\varepsilon$-Lipschitz graph of $D^{m,l}_{j_x}$ by induction hypothesis of (M3) for step $k$. In addition, since $|D\eta^{m-1}_{j_x}-I|<C\varepsilon$, $\eta^{m-1}_{j_x}(\Gamma^{m,l}_k)$ is also a $C\varepsilon$-Lipschitz graph of $L^{m,l}_{j_x}$ with $\eta^{m-1}_{j_x}(\Gamma^{m-1}_k)\subset L^{m-1}_{j_x}$ in $4.9B_{j_x}$, where $C$ is a geometric constant.
    Thus we have $\dd(\eta^{m-1}_{j_x}(x),L^{m,l}_{j_x})\leq C\varepsilon \dd(\eta^{m-1}_{j_x}(x),L^{m-1}_{j_x})$. 
    {  This estimate implies that $\eta_{j_x}^{m-1}(x) \in O^{(1/2)}(w,l)$ for some $w \in \Lambda_m$. Recall that the family of conical neighborhoods $\{O(w)\}_w$ here is associated with the $m$-spine of $W_{j_x}$. Consequently, by \eqref{h的定义}, we have} 
    \begin{equation}\label{psi_jx}
        \psi^{m}_{j_x}(x)=\pi^{m,l}_{j_x}\circ \eta^{m-1}_{j_x}(x).
    \end{equation} 
    On the other hand, $L^{m,l}_{j_x}$ meets $20B_i$. Actually, $\dd(x,L^{m,l}_{j_x})<C\varepsilon 2^{-k} $ for some geometric constant $C>0$ by the fact that $\Gamma^{J_x}_k\cap B(x_{j_x},C_{J_x,3}\varepsilon 2^{-k})\neq \emptyset$ in (M2), $\Gamma^{m,l}_k$ is a $C_{m,4}\varepsilon$-Lipschitz graph over $D^{m,l}_{j_x}$ in (M3) and $x\in 3 B_{j_x}$.
    By Lemma \ref{lem:球近则对应的锥接近}, there is a unique branch of $L^m_i$, denoted as $L^{m,l}_i$, such that
    $
        d_{x_i,20r_i}(L^{m,l}_i,L^{m,l}_{j_x})<C_7\varepsilon.
    $
    We continue to prove that 
    \begin{equation}\label{式子：psi_i定义}
        \psi^m_i(x)=\pi^{m,l}_i \circ \eta^{m-1}_i(x).
    \end{equation}
    Recall that, as defined at the beginning of Section 5.1.1, we have $\eta_i^{m-1}=id$ for $i\in I_m(k)$. Consequently, in this case, since $L^{m}_i$ is an $m$-plane, (\ref{式子：psi_i定义}) follows immediately. 
    When $i\notin I_m(k)$, for each branch $L^{m,l'}_i\subset L^{m}_i$, $ \Gamma^{m,l'}_k$ is a $C\varepsilon$-Lipschitz graph of a unique $D^{m,l'}_i$ by induction hypothesis for (M3) of step $k$. And $\eta^{m-1}_i$ maps $\Gamma^{m-1}_k$ to $L^{m-1}_i$ with $|D\eta^{m-1}_i-I|<C\varepsilon$ by \eqref{eta的4个性质}. Therefore, there is $C>0$ depending on $\{C_{u,e}\}_{0\leq u\leq m,1\leq e\leq 5}$ such that
    \begin{equation}\label{eta映射后Gamma的像在骨架里面}
        \eta^{m-1}_i(\Gamma^{m,l'}_k)\subset \{z: \dd(z,L^{m,l'}_i)\leq C\varepsilon \dd(z,L^{m-1}_i)\}\text{ in } 4.9 B_i,
    \end{equation}
     and points on the same $\eta^{m-1}_i (\Gamma^{m,l'}_k)$ should be 
     {   mapped by the orthogonal projection} to the $m$-plane that contains the same branch of $L^{m}_i$. 
     Furthermore, the projection of $\eta^{m-1}_i(\Gamma^{m,l'}_k)$ is contained in $L^{m,l'}_i$. Thus, once we know the image of any point in $\Gamma^{m,l}$, we know which branch of $L^m_i$ should $\eta^{m-1}_i(x)$ be projected to. 
     For this, we just pick a point $y\in \Gamma^{m,l}_k\cap 3 B_i$ such that the distance between $y$ and $L^{m-1}_i$ is greater than $C(\alpha,n) r_i$, where $C(\alpha,n)$ is a constant depending only on $\alpha,n$. By (\ref{alpha定义}), such $y$ exists. And we can get that $y$ is $C\varepsilon 2^{-k}$ close to $L^{m,l}_i$ by (\ref{eta映射后Gamma的像在骨架里面}). As a result, there is a geometric constant $C>0$ such that $\dd(\eta^{m-1}_i(y),L^{m,l}_i)<C\varepsilon \dd(\eta^{m-1}_i(y) , L^{m-1}_i)$. So $\psi^m_i(y)=\pi^{m,l}_i \circ \eta^{m-1}_i(y)$ and (\ref{式子：psi_i定义}) follows.
    
    By (\ref{psi_jx}) and (\ref{式子：psi_i定义}), we have $|\psi^m_i(x)-\psi^m_{j_x}(x)|\leq |\pi^{m,l}_i\circ \eta^{m-1}_i (x)-\eta^{m-1}_i(x)|+ |\pi^{m,l}_{j_x}\circ \eta^{m-1}_{j_x}(x)-\eta^{m-1}_{j_x}(x)|+|\eta^{m-1}_i(x) - \eta^{m-1}_{j_x}(x)|$. By \eqref{eta的4个性质} and (\ref{eta映射后Gamma的像在骨架里面}), there is $C_8>0$ depending only on $n_0,\delta_0,n,\alpha$ and $\{C_{u,e}\}_{0\leq u\leq m-1,1\leq e\leq 5}$ such that $|\psi^m_i(x)-\psi^m_{j_x}(x)|\leq C_8\varepsilon 2^{-k}$ and Lemma \ref{lem for m:对同一个点映射之后距离相差很小} follows.
   
\end{proof}

\subsubsection{Proof of (M1) for the inductive step $k+1$} 

Let  $x\in\Gamma^m_k\cap B(0,\rho^m_k)$ be such that $x_i\in B(0,\rho^m_k )$ for all $i\in I_x$. Then, by Lemma \ref{lem for m:对同一个点映射之后距离相差很小}, $|g^m_k(x)-\psi^m_i(x)|<C_8\varepsilon 2^{-k}$. Since $\psi^m_i(x)\in L^m_i\cap 4B_i$, we can get that $\dd(\psi^m_i(x),E_m)<C\varepsilon 2^{-k}$ by (\ref{取出来的锥每一层都接近}) and $\dd(g^m_k(x),E_m)<C\varepsilon 2^{-k}$ by (\ref{fk,gk}), where $C>0$ depends on $C_5,C_8$ and some other geometric constants. 
Now fix $x\in\Gamma^{m}_{k+1}\cap B(0,\rho^m_{k+1})$, there is $y\in \Gamma^m_k$ such that $x=g^m_k(y)$ and $|x-y|<n_0^{-n}\cdot 2^{-k-90}$ by (\ref{g^m_k粗略估计}). By (\ref{rho^m_k}), we have $\rho^m_{k+1}=\rho^m_k-n_0^{-n}\cdot 2^{-k-12}$. Therefore, $x_s\in B(0,\rho^m_k) $ for all $s\in I_y$.  
And there exists $C_{m,1}>0$ depending only on $C_5,C_8$ and $n_0,\delta_0,n,\alpha$ such that $ \dd(x, E_m)=\dd(g^m_k(y), E_m)<C_{m,1}\varepsilon 2^{-k-1}$. Thus (M1) holds for $k+1$.  

Let us check that there is $C>0$ depending on $C_5,C_8,\{C_{u,e}\}_{0\leq u\leq m-1,1\leq e\leq 5}$ and $C_{m,1}$ such that
\begin{equation}\label{g^m_k移动点y很近}
    |g^m_k(y)-y|<C\varepsilon 2^{-k} \text{ when }y\in \Gamma^{m}_k\cap B(0,\rho^m_k).
\end{equation}
Since $g^m_k$ is an average of $\psi^m_s$ for all $s\in I_y$, we only consider $|y-\psi^m_s(y)|$ for one $s\in I_y$. It is clear that $x_s\in B(0,\rho^m_k)$. And $|\psi^m_s(y)-y|\leq |\eta^{m-1}_s(y)-y|+\dd(\eta^{m-1}_s(y), L^m_s)\leq  2|\eta^{m-1}_s(y)-y|+\dd(y, L^m_s)$. By (M1) and (\ref{取出来的锥每一层都接近}),  $\dd(y, L^m_s)<C\varepsilon 2^{-k}$. Then by \eqref{eta的4个性质} and Lemma \ref{lem for m:对同一个点映射之后距离相差很小}, (\ref{g^m_k移动点y很近}) follows.

\subsubsection{Proof of (M2) for the inductive step $k+1$} 
We begin to prove (M2) for $k+1$. Let $i\in I_m(k+1)$ such that $x_i\in B(0,\rho^m_{k+1})$. Then $x_i\in E_m$ and there is $j\in \ccup_{t=0}^m I_t(k)$ such that $x_i\in 2B_j$ by (\ref{2Bi构成Em覆盖}). So let us choose $j$ such that
\begin{equation}\label{M2取j}
    j\in \ccup_{t=0}^m I_t(k)\text{ and }x_i\in 2B_j.
\end{equation}
First we check that
\begin{equation}\label{式子for m:5.5Bi中的点来自5Bj}
    \Gamma^{m,l}_{k+1}\cap 5.5 B_i=g^m_k(\Gamma^{m,l}_k\cap 5B_j)\cap 5.5 B_i
\end{equation}
for each $l$.
It is obvious that the right-hand-side is contained in the left-hand-side. Then we consider the converse.
Fix $l$, if $\Gamma^{m,l}_{k+1}\cap 5.5 B_i\neq \emptyset$, then for each $x\in \Gamma^{m,l}_{k+1}\cap 5.5B_i$, there exists $y\in\Gamma^{m,l}_k$ such that $x=g^m_k(y)$. Then $y\in B(0,\rho^m_k)$ and by (\ref{g^m_k移动点y很近}), $|x-y|<C\varepsilon 2^{-k}$. Thus $ y\in 5B_j$.  If $ \Gamma^{m,l}_{k+1}\cap 5.5 B_i=\emptyset$, then $g^m_k(\Gamma^{m,l}_k)\cap 5.5B_i=\emptyset$, and the right-hand-side is also empty. Then (\ref{式子for m:5.5Bi中的点来自5Bj}) follows.

We continue to prove that 
\begin{equation}\label{M2只有一个分支和5Bi交}
    \text{$L^m_j$ coincides with an $m$-plane in $5.5 B_i$ and only one branch of  $\Gamma^m_{k+1}$ meets $5.5B_i$. }
\end{equation}
And there exists $x\in\Gamma^{m,l}_k$ such that 
\begin{equation}\label{M2中x的存在性}
    x\in\frac{1}{200}B_i \text{ and } g^m_k(x)\in \frac{1}{100}B_i.
\end{equation}
If $j\in I_m(k)$, there is only one branch of $\Gamma^{m}_k$ meeting $5B_j$ by induction hypothesis of (M2) for step $k$, and  (\ref{式子for m:5.5Bi中的点来自5Bj}) implies that there is only one branch of $\Gamma^{m}_{k+1}$ meeting $5.5 B_i$. Furthermore, since  $x_i\in E_m\cap 2B_j $, we have $\dd(x_i,L^m_j)<C\varepsilon 2^{-k}$ by (\ref{取出来的锥每一层都接近}), where $C$ depends on $C_5$. Since $\Gamma^{m,l}_k $ is a $C_{m,2}\varepsilon$-Lipschitz graph of $L^{m}_j$ in $5B_j$ by (M2) of step $k$, we get that $\dd(x_i,\Gamma^{m,l}_k)<C\varepsilon 2^{-k}$ and $\dd(x_i,\Gamma^{m,l}_{k+1})<C\varepsilon 2^{-k}$, where $C$ depends on $C_{m,2}, C_{m,3}, C_5$. Then we get $x$ in (\ref{M2中x的存在性}).
If $j\notin I_m(k)$, then $10^3 n_0 B_i\subset 3B_j$ and $10^3 n_0 B_i\cap L^{m-1}_j=\emptyset$. Also, we have $\dd(x_i,L^m_j)<C\varepsilon 2^{-k}$ by (\ref{取出来的锥每一层都接近}). By Proposition \ref{用次一级的锥替代}, there is $Z\in \A(m)$ such that $Z\cap 100B_i=W_j\cap 100B_i$. Since $L^m(Z)$ is an $m$-plane, only one branch of $L^m_j$ meets $100B_i$, denoted as $L^{m,l}_j$.
Then we aim to show that the distance between $x_i$ and {   $\Gamma^{m,l}_j$} 
is no more than $C\varepsilon 2^{-k}$. For this, suppose $j\in I_t(k)$, then we have $\Gamma^t_k\cap B(x_j,C_{t,3}\varepsilon 2^{-k})$ by induction hypothesis of (M2) for dimension $t$. Furthermore, by (M3), $\Gamma^{m,l}_k$ is a $C_{m,4}\varepsilon$-Lipschitz graph of $D^{m,l}_j$ with (\ref{Dtl和Ltl很近}) and $\Gamma^t_k$ is a boundary of $\Gamma^{m,l}_k$, then we have
$\dd(y,\Gamma^{m,l}_k)<C\varepsilon 2^{-k}$ for all $y\in L^{m,l}_j\cap 4B_j$, where $C$ depends on $C_{t,3}, C_{t+1,4},...,C_{m,4}$. Since $\dd(x_i,L^m_j)=\dd(x_i,L^{m,l}_j)$, we have $\dd(x_i,\Gamma^{m,l}_k)<C\varepsilon 2^{-k}$ and $\dd(x_i,\Gamma^{m,l}_{k+1})<C\varepsilon 2^{-k}$, where $C$ depends on $C_{t,3},C_{t+1,4},..., C_{m,4}, C_5$.   
For other $l'\neq l$, $L^{m,l'}_j$ does not meet $ 100B_i$. So   $\Gamma^{m,l'}_k$ does not meet $99B_i$ by induction hypothesis for (M2) and (M3) of step $k$ and $\Gamma^{m,l'}_{k+1}$ does not meet $98B_i$ by (\ref{g^m_k移动点y很近}). Then (\ref{M2只有一个分支和5Bi交}) and (\ref{M2中x的存在性}) follow.

Fix $x$ in (\ref{M2中x的存在性}) and let $B_x:=B(x,(5+1/50)r_i).$
Recall (\ref{M2取j}), then $5B_i\subset B_x\subset 5.5B_i \subset 5B_j$. By the same argument in  (\ref{式子for m:5.5Bi中的点来自5Bj}), we also have
\begin{equation}\label{M2:5Bi中的点来自Bx}
    \Gamma^{m,l}_{k+1}\cap 5B_i=g^m_k(\Gamma^{m,l}_k\cap B_x)\cap 5B_i.
\end{equation}
For (\ref{M2的式子}), let us show that there is $C>0$ depending only on $C_7,C_8$,$\{C_{t,4}\}_{0\leq t\leq m-1,1\leq e\leq 5}$ and some other geometric constants such that, for each $z\in \Gamma^{m,l}_k\cap 5.5B_i$, we have
\begin{equation}\label{M2Dg-Dpi的式子}
    |Dg^m_k(z)-D\pi^m_i|\leq \sum_{s\in I_z} |\theta_s(z)|\cdot |D\psi^m_s(z)-D\pi^m_i|+\sum_{s\in I_z}|D\theta_s(z)|\cdot|\psi^m_s(z)-g^m_k(z)|<C\varepsilon.
\end{equation}
For (\ref{M2Dg-Dpi的式子}), we first show that there is $C>0$ depending on $C_7,\{C_{u,e}\}_{0\leq u\leq m-1,1\leq e\leq 5}$ such that
\begin{equation}\label{M2Dg-Dpi的第一个式子} 
    |D\psi^m_s(z)-D\pi^m_i|<C\varepsilon{\text{ for } s\in I_z}.
\end{equation}
For each $s\in I_z$, we have $z\in 3B_s$, therefore $ 10B_i\cap 10B_s\neq\emptyset$. By Lemma \ref{lem:球近则对应的锥接近}, there is a unique branch of $L^m_s$ interacting with $20B_i$, denoted as $L^{m,l}_s$, such that $d_{x_i,20r_i}(L^{m,l}_s,L^m_i)<C_7\varepsilon$. If $s\in I_m(k)$, $L^m_s$ is an $m$-plane and $\psi^m_s(z)=\pi^m_s (z)$, implying that $|D\psi^m_s(z)-D\pi^m_i|<C\varepsilon$, where $C$ depends on $C_7$. So (\ref{M2Dg-Dpi的第一个式子}) for $s\in I_m(k)$ follows. 
If $s\notin I_m(k)$, by (\ref{取出来的锥每一层都接近}) and (\ref{级别低的球离级别高的点远}), $10^3 n_0 B_i\cap L^{m-1}_s=\emptyset$. Therefore, $L^m_s$ coincides with an $m$-plane in $100B_i$ by Proposition \ref{用次一级的锥替代}. That is, $L^{m}_s\cap 100 B_i=L^{m,l}_s\cap 100 B_i$. Because $z\in 5.5B_i$ and $|\eta^{m-1}_s(z)-z|<C\varepsilon 2^{-k}$ in \eqref{eta的4个性质}, we have $\dd(\eta^{m-1}_s(z), L^{m-1}_s)\gg 100r_i$. On the other hand, by induction hypothesis of (M3) of step $k$,  since $z\in \Gamma^{m,l}_k$, $\dd(z,L^m_s)=\dd(z,L^{m,l}_s)<C\varepsilon 2^{-k}$, so $\eta^{m-1}_s(z)$ is also $C\varepsilon 2^{-k}$-close to $L^{m,l}_s$. As a conclusion, we have $\dd(\eta^{m-1}_s(z),L^{m,l}_s)<C\varepsilon \dd(\eta^{m-1}_s(z),L^{m-1}_s)$, then   $\psi^m_s(z)=\pi^{m,l}_s\circ\eta^{m-1}_s(z)$ and 
\begin{equation}
    |D\psi^m_s(z)-D\pi^m_i|\leq |D\psi^m_s(z)-D\ppi^{m,l}_s|+|D\ppi^{m,l}_s-D\pi^m_i|,
\end{equation}
where $|D\psi^m_s(z)-D\ppi^{m,l}_s|\leq |D\ppi^{m,l}_s|\cdot |D\eta^{m-1}_s(z)-I|<C\varepsilon$ by \eqref{eta的4个性质} and $|D\ppi^{m,l}_s-D\pi^m_i|<C\varepsilon $ by Lemma \ref{面的dx,r决定夹角和距离}. Then ({\ref{M2Dg-Dpi的第一个式子}) follows.
As for the second formula in (\ref{M2Dg-Dpi的式子}), by (\ref{单位分解的导数}) Lemma \ref{lem for m:对同一个点映射之后距离相差很小}, we have $|D\theta_s(z)|\cdot|g^m_k(z)-\psi^m_s(z)|<C\varepsilon 2^{-k}$. 
Combining with (\ref{M2Dg-Dpi的第一个式子}), (\ref{M2Dg-Dpi的式子}) follows.

According to the induction hypothesis for (\ref{M2的式子}) and (\ref{M3的式子}) for step $k$, for any two points  $w_1,w_2\in\Gamma^{m,l}_k\cap B_x\subset \Gamma^{m,l}_k\cap 5B_j$, there is a curve $\gamma:[0,1]\to \Gamma^{m,l}_k\cap 5.5 B_i$ such that $\gamma(0)=w_1,\gamma(1)=w_2$ while $|\gamma|\leq (1+C\varepsilon)|w_1-w_2|$, where $C$ depends on $C_{m,2}$ if $j\in I_m(k)$ and $C$ depends on $C_{m,4}$ if $j\in \cup_{t=0}^{m-1}I_t(k)$. Denote by ${(L^m_i)}^{\perp}$ the orthogonal subspace of $L^m_i$ and $(\pi^m_i)^{\perp}$ the orthogonal projection onto $(L^m_i)^{\perp}$. Then we can estimate that
\begin{equation}\label{M2包含于Lip graph}
    \begin{split}
     &|(\pi^m_i)^{\perp} \circ g^m_k(w_1) - (\pi^m_i)^{\perp} \circ g^m_k(w_2)|
     <C\varepsilon|w_1-w_2|,\\
    &|\pi^m_i \circ g^m_k(w_1)-\pi^m_i\circ g^m_k(w_1)|\geq (1-C\varepsilon)|w_1-w_2|>0.99|w_1-w_2|
    \end{split}
\end{equation}
by (\ref{M2Dg-Dpi的式子}), and $C$ depends only on $C_7,C_8$,$\{C_{u,e}\}_{0\leq u\leq m-1,1\leq e\leq 5}$ and some other geometric constants. (Although we seem to use $C_{m,2}$ and $C_{m,4}$ in the definition of $\gamma$, the main point is that $ 1-C\varepsilon>0.99$). 

Already (\ref{M2包含于Lip graph}) indicates that $g^m_k(\Gamma^{m,l}_k\cap B_x)$ is contained in a $C\varepsilon$-Lipschitz graph $G$ over $L^m_i$. Let us check that there is no hole. That is,
\begin{equation}\label{式子for m:满射1}
    g^m_k( \Gamma^{m,l}_k \cap B_x) \cap 5B_i = G\cap 5B_i.
\end{equation} 
{  
Recall from (\ref{M2只有一个分支和5Bi交}) that exactly one branch of $L^{m}_j$ meets $5.5 B_i$. For convenience, we denote this branch by $L^{m,l}_j$, regardless of whether $j\in I_m(k)$ or not. 
By the induction hypotheses (M2) and (M3) at step $k$, $\Gamma^{m,l}_k$ is a $C_{m,2}\varepsilon$-Lipschitz graph of $L^{m,l}_j$ (if $j\in I_m(k)$) or is a $C_{m,4}\varepsilon$-Lipschitz graph of $D^{m,l}_j$ (if $j\notin I_m(k)$). Let $\varphi^{m,l}_j: P^{m,l}_j\to (P^{m,l}_j)^{\perp}$ denote  the corresponding $C\varepsilon$-Lipschitz map in either case, and define the parametrization $\phi^{m,l}_j=\varphi^{m,l}_j+id$. From (\ref{M2只有一个分支和5Bi交}), we deduce that $\ppi^{m,l}_j(x) \in L^{m,l}_j$, and thus  $\ppi^{m,l}_j(x)=\pi^{m,l}_j(x)$. 

Let us define the domain $D=B( \pi^{m,l}_j(x), (5+1/60)r_i )\cap P^{m,l}_j $. Note that $D$ is contained in $L^{m,l}_j$. Moreover, for all $y\in D$, we have the estimate $|\phi^{m,l}_j(y)-y|<C\varepsilon 2^{-k}$, where $C$ depends only on $C_{m,2}$ or $C_{m,4}$. Provided $\varepsilon$ is sufficiently small, we have $C\varepsilon 2^{-k}\ll r_i$. Thus, the following inclusion holds:
\begin{equation}
    \phi^{m,l}_j (D)
    \subset \Gamma^{m,l}_k\cap B_x.
\end{equation}
Now consider $h=\pi^m_i \circ g^m_k\circ \phi^{m,l}_j$. Since it is straightforward to verify that for any $y \in D$, we have $y\in D^{m,l}_j\cap 4.9B_j$, $\phi^{m,l}_j(y)\in 3B_j$, and $h(y)\in 10B_i$, it follows from (\ref{取出来的锥每一层都接近}), (\ref{g^m_k移动点y很近}) and (M1) that $|h(y)-y|<C\varepsilon 2^{-k} $.
Degree theory yields the following inclusion: 
\begin{equation}
   B(h \circ \pi^{m,l}_j(x), (5+\frac{1}{70})r_i)\cap L^m_i\subset h(D).
\end{equation}
To be precise, consider the linear homotopy $H_t(y) = (1-t)y + t h(y)$. Note that $|h(y)-y|$ is bounded by $C\varepsilon 2^{-k}$, which is much smaller than the gap between the boundary $\partial D$ and the target ball $B(h \circ \pi^{m,l}_j(x), (5+\frac{1}{70})r_i)$. This geometric gap ensures that for any $z$ in the target ball, the homotopy path of the boundary, $H_t(\partial D)$, never passes through $z$. Consequently, the topological degree is well-defined and invariant, so that $ \deg(h, D, z) = \deg(id, D, z) = 1$. This non-zero degree implies that $z \in h(D)$, which yields the desired inclusion.

Recall that $x\in B(x_i,r_i/200)$ in (\ref{M2中x的存在性}), thus $|h\circ \pi^{m,l}_j(x) -x_i|<r_i/100$ and we have
\begin{equation}
    L^m_i\cap 5B_i\subset  L^m_i\cap B( h\circ \pi^{m,l}_j(x),(5+\frac{1}{70})r_i)
\end{equation}
As a consequence, $\pi^m_i\circ g^m_k (\Gamma^{m,l}_k\cap B_x)$ contains $L^m_i\cap 5B_i$, then (\ref{式子for m:满射1}) follows. Combining with (\ref{M2:5Bi中的点来自Bx}), we have
\begin{equation}\label{M2Gammak+1是graph}
    \Gamma^{m,l}_{k+1}\cap 5B_i=G\cap 5B_i,
\end{equation}
which indicates that $\Gamma^{m,l}_{k+1}$ is a $C_{m,2}\varepsilon$-Lipschitz graph over $L^m_i$ in $5B_i$ and we set $G^m_i=G$.
} 

At last let we check that $\Gamma^{m,l}_{k+1}\cap B(x_i,C_{m,3}\varepsilon 2^{-k-1})\neq\emptyset$. By (\ref{g^m_k移动点y很近}) and (\ref{M2中x的存在性}), $|g^m_k(x)-x_i|<r_i/100 $. By (M1), $g^m_k(x)$ is $C_{m,1} \varepsilon 2^{-k-1}$ close to $E_m$. So  $\pi^m_i\circ g^m_k(x)$ is contained in $ L^m_i\cap B(x_i,r_i/100)$. Furthermore, by (M1) and (\ref{取出来的锥每一层都接近}), $\dd(g^m_k(x), L^m_i)=|g^m_k(x)-\pi^m_i\circ g^m_k(x)|<C\varepsilon 2^{-(k+1)}$ for some $C$ that depends only on $C_{m,1}$ and $C_5$. Then (\ref{M2Gammak+1是graph}) implies that there 
is $C_{m,3}$ depending on $C_5,C_{m,1}, C_{m,2}$ such that $\dd(x_i,\Gamma^{m,l}_{k+1})<C_{m,2}\varepsilon  r_i/100+ \dd( g^m_k(x),L^m_i)<C_{m,3}\varepsilon 2^{-k-1}$. So we end the proof of (M2) for $k+1$.

\subsubsection{Proof of (M3) for the inductive step $k+1$}

Now we begin to prove (M3) for $k+1$. Suppose $i\in I_t(k+1)$ for some $0\leq t\leq m-1$ and $L^m_i=\cup_{l} L^{m,l}_i$. Recall that $b_i^m$ is the number of branches of $L_i^m$. First, we aim to show that
\begin{equation}\label{和5Bi相交的分支数}
\begin{aligned}
&\text{there are exactly $b^m_i$ branches of $\Gamma^{m}_{k+1}$ intersecting with $5B_i$,}\\
&\text{while the remaining branches do not intersect with $5.5B_i$.}
\end{aligned}
\end{equation}
Since $x_i\in E_t$, (\ref{2Bi构成Em覆盖}) says that there is $j\in \ccup_{s=0}^t I_s(k)$ such that 
$x_i\in 2B_j$. By Lemma \ref{lem:球近则对应的锥接近}, for each branch of $L^{m,l}_i\subset L^m_i$, we can find a unique branch of $L^{m,l}_j\subset L^m_j$ such that 
\begin{equation}\label{M3Li和Lj近的分支}
    d_{x_i,20r_i}(L^{m,l}_i,L^{m,l}_j)<C_7\varepsilon
\end{equation}
And the number of branches of $L^m_j$ that intersects with $20B_i$ is equal to $b^m_i$.
Since $x_i\in L^{t}_i\subset L^{m,l}_i$, we have
\begin{equation}\label{M3xi离L^m_j近}
    \dd(x_i, L^{m,l}_j)<C\varepsilon 2^{-k},
\end{equation}
where $C$ depends on $C_7$. Thus, if a branch of $L^m_j$ meets $20B_i$, it also meets $5B_i$. So we have proved that there are exactly $b^m_i$ branches of $L^m_j$ intersecting with $5B_i$ and $20 B_i$. 
By the induction hypothesis for (M3) of dimension $m$ and step $k$, each $\Gamma^{m,l}_k$ and $L^{m,l}_j$ is contained in the $C\varepsilon 2^{-k}$ neighborhood of each other in $5B_j$, where the main point is that $C\varepsilon 2^{-k}$ is much smaller than $r_j$. 
And the remaining branches does not meet $19B_i$ so does not meet $5.5B_i$. By (\ref{g^m_k移动点y很近}), we have (\ref{和5Bi相交的分支数}).
 
By the same proof as for (\ref{式子for m:5.5Bi中的点来自5Bj}), we still have 
\begin{equation}\label{for m:点来源于5Bj}
    \Gamma^{m,l}_{k+1}\cap 5.5 B_i=g^m_k(\Gamma^{m,l}_k\cap 5B_j)\cap 5.5 B_i.
\end{equation}
for each $l$. Furthermore, since $x_i\in 2B_j$, we can get that $\Gamma^{m,l}_k\cap B(x_i,C\varepsilon 2^{-k})\neq \emptyset$ for $1\leq l\leq b^m_i$ by (\ref{M3xi离L^m_j近}), (M2) (if $j\in I_t(k)$) or (M3) (if $j\notin I_t(k)$) for dimension $t$, and induction hypothesis for dimension $m$ and step $k$. Thus, for each $1\leq l\leq b^m_i$ with $\Gamma^{m,l}_k\cap 5 B_i\neq \emptyset$, we can find $x\in\Gamma^{m,l}_k$ such that 
\begin{equation}\label{M3用于定位的x}
    x\in\frac{1}{200}B_i\text{ and } g^m_k(x)\in\frac{1}{100}B_i.
\end{equation}

We want to study $g^m_k$ on each $\Gamma^{m,l}_k$ in $5.5 B_i$. Let $1\leq l\leq b^m_i$ and $x$ as in (\ref{M3用于定位的x}) be fixed. Let 
$
    B_x=B(x,(5+{1}/{50})r_i).
$
Then $ 5.01 B_i\subset B_x\subset 5.5 B_i\subset 5B_j$. And we have
\begin{equation}\label{M3:5Bi的来自Bx}
    \Gamma^{m,l}_{k+1}\cap 5 B_i=g^m_k(\Gamma^{m,l}_k\cap B_x)\cap 5 B_i
\end{equation}
by the same proof as for (\ref{for m:点来源于5Bj}). 
We aim to show that there is a constant $C>0$ depending on $C_7,C_8,\{C_{u,e}\}_{0\leq u\leq m-1,1\leq e\leq 5}$ and some other geometric constants such that, for each $z\in \Gamma^{m,l}_k\cap 5.5B_i$
\begin{equation}\label{M3估算|Dg-Dpi|}
    |Dg^m_k(z)-D\ppi^{m,l}_i|\leq \sum_{s\in I_z} |\theta_s(z)|\cdot |D\psi^m_s(z)-D\ppi^{m,l}_i|+\sum_{s\in I_z}|D\theta_s(z)|\cdot|\psi^m_s(z)-g^m_k(z)|<C\varepsilon.
\end{equation} 
For (\ref{M3估算|Dg-Dpi|}), we first show that there is $C>0$ depending only on $C_7,\{C_{u,e}\}_{0\leq u\leq m-1,1\leq e\leq 5}$ and some other geometric constants such that
\begin{equation}\label{M3Dpsi-Dpi}
    | D\psi^m_s(z)-D\ppi^{m,l}_i|<C\varepsilon.
\end{equation}
For (\ref{M3Dpsi-Dpi}), we need to know $\psi^m_s(z)$ in different $s\in I_z$ and we will write it precisely in (\ref{M3:Bs里的映射}). For each $s\in I_z$, since $z\in 3B_s$, we have $10B_i\cap 10B_s\neq\emptyset$. 
And we have $s\in I_w(k)$ for some $0\leq w\leq m$ by (M1) for dimension $m$ and (\ref{2Bi构成Em覆盖}). Let us show that 
\begin{equation}\label{M3L^m_s离L^m,l_i近的分支}
    d_{x_i,20r_i}(L^{m,l_s}_s,L^{m,l}_i)<C_7\varepsilon\text{ for a unique branch }L^{m,l_s}_s\subset L^{m}_s.
\end{equation}
If $w\leq t$, 
(\ref{M3L^m_s离L^m,l_i近的分支}) is clear by Lemma \ref{lem:球近则对应的锥接近}. 
If $w>t$, in order to use Lemma \ref{lem:球近则对应的锥接近}, we need to show that $L^{m,l}_i\cap 20B_s\neq\emptyset.$ Since $z\in 5.5 B_i$ and $x_i\in 2B_j$, we have $z\in 4.75 B_j$. By induction hypothesis for (M2) of dimension smaller than $m$ and for (M3) of dimension $m$ and step $k$, $\Gamma^{m,l}_k$ is a $C\varepsilon$-Lipschitz graph passing through the $C\varepsilon 2^{-k}$ neighborhood of $x_j$, so we have $\dd(z, L^{m,l}_j)<C\varepsilon 2^{-k}$. Furthermore, by (\ref{取出来的锥每一层都接近}), $ \dd(z, L^{m,l}_i)<C\varepsilon 2^{-k}$, where the main point is that $C\varepsilon 2^{-k}\ll r_s$. Thus we have $L^{m,l}_i\cap 4B_s\neq \emptyset$ because $z\in 3B_s$. And (\ref{M3L^m_s离L^m,l_i近的分支}) follows by Lemma \ref{lem:球近则对应的锥接近}.

Thanks to (\ref{M3L^m_s离L^m,l_i近的分支}), we are ready to prove that 
\begin{equation}\label{M3:Bs里的映射}
    \psi^m_s(z)=\ppi^{m,l_s}_s \circ \eta^{m-1}_s(z).
\end{equation}
If $s\in I_m(k)$, we have $\eta^{m-1}_s=id$ and $\ppi^{m,l_s}_s=\pi^m_s$ so (\ref{M3:Bs里的映射}) holds. If $s\notin I_m(k)$, suppose $s\in I_w(k)$ for some $0\leq w<m$, we consider (\ref{M3:Bs里的映射}) for two cases that $t<w<m $ and $0<w\leq t$ respectively.

If $t<w<m$, recall the proof for (\ref{psi_jx}) in Lemma \ref{lem for m:对同一个点映射之后距离相差很小}, to determine which branch of $L^m_s$ that $\eta^{m-1}_s(z)$ to project onto, we only need to know the projection of one point on $\Gamma^{m,l}_k $ onto $L^m_s$. For this, pick a point $ p\in\Gamma^{m,l}_k\cap 3B_s$ such that $\dd(p, L^{m-1}_s)>C(\alpha,n)r_s$, where $C(\alpha,n)$ is a geometric constant depending only on $\alpha,n$. Since $\eta^{m-1}_s$ moves $p$ no more than $C\varepsilon 2^{-k}$ by \eqref{eta的4个性质}, we have $\dd(\eta^{m-1}_s(p),L^{m-1}_s)> C(\alpha,n)r_s/2$. 
In the meanwhile, $ |p-x_j|<4.8r_j$, thus we can use the induction hypothesis for (M2) of dimensions smaller than $m$, (M3) of dimension $m$ and step $k$ to get that $\dd(p, L^{m,l}_j)<C\varepsilon 2^{-k}$. Therefore, we have $\dd(p,L^{m,l_s}_s)<C\varepsilon 2^{-k}$ by (\ref{M3Li和Lj近的分支}), (\ref{M3L^m_s离L^m,l_i近的分支})  and $p\in 3B_s$. Furthermore, we can get that
\begin{equation}
    \dd(\eta^{m-1}_s(p),L^{m,l_s}_s)<C\varepsilon 2^{-k}\ll \sin(\frac{\alpha}{10})\cdot \dd(\eta^{m-1}_s(p),L^{m-1}_s).
\end{equation}
It indicates that $\psi^m_s(p)=\ppi^{m,l_s}_s\circ \eta^{m-1}_s(p) $ and (\ref{M3:Bs里的映射}) for $t<w<m$ follows.

If $0<w\leq t$, we consider the position of another point on $\Gamma^{m,l}_k$. Pick a point $q\in L^{m,l}_i\cap 6B_i\cap 4B_s $ such that $\dd(q,L^{m-1}_i)\geq C(\alpha,n)r_i$, where $C(\alpha,n)$ is a geometric constant depending only on $\alpha,n$. 
Then there is $q'\in L^{m,l}_j$ such that $|q-q'|<C\varepsilon 2^{-k}$ with $|q'-x_j|\leq |q'-q|+|q-x_i|+|x_i-x_j|<4.8r_j$ by (\ref{M3Li和Lj近的分支}).
By induction assumption for (M2) of dimensions smaller than $m$ and (M3) of dimension $m$ and step $k$, there exists $p\in\Gamma^{m,l}_k$ such that $ |p-q'|<C\varepsilon 2^{-k}$. By (\ref{M3Li和Lj近的分支}) and (\ref{M3L^m_s离L^m,l_i近的分支}), we have $\dd(p,L^{m,l_s}_s)<C\varepsilon 2^{-k}$ while $ \dd(p,L^{m-1}_s)> C(\alpha,n)r_i/3$. So $p\in \Gamma^{m,l}_k$ is such that $|p-x_s|\leq |p-q'|+|q'-q|+|q-x_s|<4.1 r_s$ and $\dd(\eta^{m-1}_s(p), L^{m,l_s}_s)<C\varepsilon \dd(\eta^{m-1}_s(p),L^{m-1}_s)$, where the main point is that $C\varepsilon\ll \sin(\alpha/10)$. And we can get that $\dd(\eta^{m-1}_s(z), L^{m,l_s}_s)<C\varepsilon \dd(\eta^{m-1}_s(z),L^{m-1}_s) $ because $p,z$ are both on $\Gamma^{m,l}_k\cap 5B_j$. Then (\ref{M3:Bs里的映射}) follows.


 For simplicity, we replace $l_s$ with $l$. Now we can use (\ref{M3:Bs里的映射}) to show (\ref{M3Dpsi-Dpi}). By (\ref{M3:Bs里的映射}), \eqref{eta的4个性质} and (\ref{M3L^m_s离L^m,l_i近的分支}), there is $C>0$ depending on $C_7,\{C_{u,e}\}_{0\leq u\leq m-1,1\leq e\leq 5}$ and some other geometric constants such that
 \begin{equation}
     |D\psi^m_s(z)-D\ppi^{m,l}_i|\leq |D\ppi^{m,l}_s|\cdot |D\eta^{m-1}_s(z)-I|+|D\ppi^{m,l}_s-D\ppi^{m,l}_i|<C\varepsilon.
 \end{equation}
Then (\ref{M3Dpsi-Dpi}) follows.
Since $g^m_k(z)$ is an average of $\psi^m_s(z)$ for all $s\in I_z$, by Lemma \ref{lem for m:对同一个点映射之后距离相差很小}, $|g^m_k(z)-\psi^m_s(z)|<C_8\varepsilon 2^{-k}$. So (\ref{M3估算|Dg-Dpi|}) follows. Then by the same argument in (\ref{M2包含于Lip graph}), we can find $C_{m,4}>0$ depending only on  $C_7,C_8,\{C_{u,e}\}_{0\leq u\leq m-1,1\leq e\leq 5}$ and some other geometric constants and a $C_{m,4}\varepsilon$-Lipschitz graph $G$ over $P^{m,l}_i$ such that
\begin{equation}\label{M3:被包含在graph中}
    g^m_k(\Gamma^{m,l}_k\cap B_x) \subset G.
\end{equation}
Next we need to show that there is no hole, that is
\begin{equation}\label{M3满射}
    g^m_k(\Gamma^{m,l}_k\cap B_x) \cap 5B_i = G\cap 5B_i.
\end{equation}
Let us consider (\ref{M3满射}) for two cases when $m=1$ and $m>1$.

\noindent\textbf{Proof of (\ref{M3满射}) when $\bm{m}$=1.} In this case, $i\in I_0(k+1)$, $j\in I_0(k)$. The two balls $B_i$ and $B_j$ are both centered at 0 and $r_j=2r_i$. In $3B_i$, we have $\theta_j=1$, so points on $\Gamma^{1,l}_k$ are just projected onto $L^{1,l}_j$ by $g^1_k$.
By (\ref{g^m_k移动点y很近}), (\ref{取出来的锥每一层都接近}), and $g^1_k(0)=0$, we have
\begin{equation}\label{1维(3)的满射0}
    L^{1,l}_i \cap 2.9B_i\subset \pi^{1,l}_i \circ g^1_k(\Gamma^{1,l}_k\cap 3 B_i).
\end{equation}
It is clear that 0 is an endpoint of the curve $\pi^{1,l}_i\circ g^1_k(\Gamma^{1,l}_k)$.
Next we consider $\Gamma^{1,l}_{k+1}$ in $B_x\backslash 2B_i$. For every $z\in \Gamma^{1,l}_k\cap (B_x\backslash 2B_i)$, (\ref{g^m_k移动点y很近}) says that $g^1_k$ moves $z$ no more than $C\varepsilon 2^{-k}$.
By the induction assumption for (M3) of dimension $m=1$ and step $k$, $\Gamma^{1,l}_k$ is a $C_{1,4}$-Lipschitz graph of $L^{1,l}_j$ in $5B_j$.
Let $\varphi^{1,l}_j:L^{1,l}_j \to (L^{1,l}_j)^{\perp}$ be this  $C_{1,4}\varepsilon$-Lipschitz map and let $\phi^{1,l}_j=\varphi^{1,l}_j+id$, $h=\ppi^{1,l}_i \circ g^1_k\circ \phi^{1,l}_j$. Then by (M1), (\ref{取出来的锥每一层都接近}), (\ref{g^m_k移动点y很近}) and induction hypothesis for (M3) of dimension $m=1$ and step $k$, we have 
\begin{equation}\label{h移动点近}
    |h(y)-y|<C\varepsilon 2^{-k}
\end{equation}
for all $y\in L^{1,l}_j\cap 4.9B_j$, where the main point is that $C\varepsilon 2^{-k}$ is much smaller than $r_i$.
Now pick a point $p\in L^{1,l}_j$ such that $|p|=3.9r_i$, then we can estimate that
\begin{equation}\label{M3h的前一步}
    \phi^{1,l}_j (L^{1,l}_j\cap B(p, 1.11 r_i))\subset \Gamma^{1,l}_k\cap ( B_x\backslash  2B_i)
\end{equation}
since $\phi^{1,l}_j$ moves points no more than $C\varepsilon 2^{-k}$.
Let the map $\pi^{1,l}_i\circ g^1_k$ act on both sides of (\ref{M3h的前一步}), we have
\begin{equation}\label{1维(3)的满射1}
    h( L^{1,l}_j \cap B(p,1.11r_i) )\subset  \pi^{1,l}_i\circ g^1_k(\Gamma^{1,l}_k\cap ( B_x\backslash 2 B_i) ).
\end{equation}
By using degree theory, we also have
\begin{equation}\label{1维(3)的满射2}
    L^{1,l}_i\cap B(h(p), 1.109 r_i )\subset h(L^{1,l}_j\cap B(p,1.11r_i)).
\end{equation}
And now we can get from  (\ref{1维(3)的满射1}) and (\ref{1维(3)的满射2}) that
\begin{equation}\label{1维(3)的满射4}
     L^{1,l}_i\cap ( 5 B_i \backslash 2.8 B_i) \subset \pi^{1,l}_i\circ g^1_k (\Gamma^{1,l}_k\cap ( B_x\backslash  2B_i))
\end{equation}
because $|p|=3.9r_i$ and $h$ moves a point no more than $C\varepsilon 2^{-k}$.
As a result of (\ref{1维(3)的满射0}) and (\ref{1维(3)的满射4}), we have
\begin{equation}
    L^{1,l}_i \cap 5B_i \subset \pi^{1,l}_i\circ g^1_k( \Gamma^{1,l}_k \cap B_x)).
\end{equation}
So (\ref{M3满射}) follows, which means that $\Gamma^{1,l}_{k+1}$ is a graph over $L^{1,l}_i$ by (\ref{M3:5Bi的来自Bx}), that is,
\begin{equation}
     \Gamma^{1,l}_{k+1}\cap 5B_i=G\cap 5B_i.
\end{equation}
Then we can set $G^{1,l}_i=G$ and know that $\Gamma^{1,l}_{k+1}$ is a $C_{1,4}\varepsilon$-Lipschitz graph of $L^{1,l}_i$ in $5B_i$. Moreover, 0 is an endpoint of $\Gamma^{1,l}_{k+1}$. Since different branches of $L^1_j$ make angles greater than $\alpha$, other branches of $\Gamma^{1}_{k+1}$ will not be $C_{1,4}\varepsilon$-Lipschitz graphs of $L^{1,l}_i$. Thus, we end the proof of (M3) when $m=1$.

\noindent\textbf{Proof of (\ref{M3满射}) when $\bm{m}$>1.} When $m>1$, let $\varphi^{m,l}_j:P^{m,l}_j\to (P^{m,l}_j)^{\perp}$ be the $C_{m,4}\varepsilon$-Lipschitz map in induction assumption for (M3) of dimension $m$ and step $k$, and let $\phi^{m,l}_j=\varphi^{m,l}_j+id$. Then in $5B_j$, $\phi^{m,l}_j$ maps points on $D^{m,l}_j$ to $\Gamma^{m,l}_k$ and $|x-\ppi^{m,l}_j(x)|<C\varepsilon 2^{-k}$, where $x$ is as in (\ref{M3用于定位的x}) and $C$ depends on $C_{w,3}$ for some $w\leq t$ and $ C_{m,4}$. Let 
\begin{equation}\label{M3B'和D的定义}
    B'=B(\ppi^{m,l}_j(x),(5+1/60)r_i)\text{ and }D=D^{m,l}_j\cap B',
\end{equation}
since $\phi^{m,l}_j$ moves points no more than $C\varepsilon 2^{-k}$, we have 
$
    \phi^{m,l}_j (D)\subset \Gamma^{m,l}_k\cap B_x.
$
By induction hypothesis for (M3) of dimension $m$ and step $k$, $\Gamma^{m,l}_k$ has boundaries that can be represented as $\{\Gamma^{m-1,l'}_k\}_{l'}$ in $5B_j$, then we have
\begin{equation}
    \partial D=S_1\cup S_0= (\cup_{l'} S^{l'}) \cup S_0,
\end{equation}
where $ S^{l'}=\ppi^{m,l}_j (\Gamma^{m-1,l'}_k) \cap B'$ and $S_0$ is a connected region of $\partial B'$.
By induction hypothesis for dimension $m-1$, $ \Gamma^{m-1,l'}_k$ is a $C_{m-1,4}\varepsilon$-Lipschitz graph of $D^{m-1,l'}_j\subset P^{m-1,l'}_j$ and $L^{m-1,l'}_j$ is an $(m-1)$-boundary of $L^{m,l}_j$. 
We aim to use Lemma \ref{lem:球近则对应的锥接近} to show that for each $l'$,
\begin{equation}\label{M3L^m_i有边界和L^m-1_j近}
    \text{there is a unique $ L^{m-1,l'}_i\subset L^{m,l}_i$ such that $d_{x_i,20r_i}( L^{m-1,l'}_i ,  L^{m-1,l'}_i )<C_7\varepsilon$.}
\end{equation}
For this, we only need to show that $L^{m-1,l'}_j$ meets $20B_i$. Since $ i\in I_t(k+1)$ and $t<m$, there is a branch $\Gamma^{t,h}_k\subset \Gamma^{m-1,l'}_k$ and $\Gamma^{t,h}_k$ passes $C\varepsilon 2^{-k}$ close to $x_i$ by (M2) for dimension $t$ and (\ref{g^m_k移动点y很近}). So $\Gamma^{m-1,l'}_k$ also passes through $x_i$ closely. By induction hypothesis for (M3) of dimension $m-1$, $L^{m-1,l'}_j$ and $ \Gamma^{m-1,l'}_k$ is contained in the $C\varepsilon 2^{-k}$ neighborhood of each other.
Thus, $L^{m-1,l'}_j$ meets $ B_i$ for each $l'$. Thus, (\ref{M3L^m_i有边界和L^m-1_j近}) follows. 

Now let us show that in $5B_i$,
\begin{equation}\label{M3L^{m-1,l'}_i恰好是对应的原像}
    \begin{split}
    &\text{$\Gamma^{m-1,l'}_{k+1}$ is exactly the $C_{m-1,2}\varepsilon$-Lipschitz graph of $L^{m-1,l'}_i$ (when $i\in I_{m-1}(k+1)$) }\\
   &\text{or the  $C_{m-1,4}\varepsilon$-Lipschitz graph of $ D^{m-1,l'}_i$ (when $i\notin I_{m-1}(k+1)$).}
    \end{split}
\end{equation}
By (\ref{psi^m-1和psi^m在m-1维graph上相等}), $g^m_k$ coincides with $g^{m-1}_k$ on $ \Gamma^{m-1}_k $. Thus, by the same argument for (\ref{M3:被包含在graph中}), we get that $g^{m-1}_k(\Gamma^{m-1,l'}_k\cap B_x) $ is contained in a $C\varepsilon$-Lipschitz graph of $ L^{m-1,l'}_i$. In the meanwhile, according to our induction hypothesis for (M2) and (M3) of dimension $m-1$, there is a unique branch of $ L^{m-1}_i$ such that $\Gamma^{m-1,l'}_{k+1}$ is a $C\varepsilon$-Lipschitz graph of this branch. Since $ (m-1)$-branches of $L^{m-1}_i$ make angles larger than $\alpha$ along their intersections, (\ref{M3L^{m-1,l'}_i恰好是对应的原像}) follows directly.

Moreover, we can even claim that 
\begin{equation}\label{M3的claim}
\begin{split}
    &\text{$ \Gamma^{m-1,l'}_{k+1}$ is a $C'\varepsilon$-Lipschitz graph of $L^{m-1,l'}$ (when $i\in I_{m-1}(k+1)$)}\\
    &\text{or a $C'\varepsilon$-Lipschitz graph of $ D^{m-1,l'}_i$ (when $i\notin I_{m-1}(k+1)$)}
    \end{split}
\end{equation}
in a ball a little bit larger than $5B_i$, for example, $6B_i$. 
Here $C'$ may be greater than $C_{m-1,2}$ and $C_{m-1,4}$, but it does not matter. The main point is that $ \Gamma^{m-1,l'}_{k+1}$ is a Lipschitz graph of $L^{m-1,l'}_i$ with small enough constant.
Claim (\ref{M3的claim}) holds because $\Gamma^{m-1,l'}_{k+1}\cap (6B_i \backslash 5B_i)$ can be covered by finitely many balls $\{B_u:u\in \ccup_{t=0}^{m-1}I_t(k+1)\}$. And our induction hypothesis for (M2) and (M3) of dimension $m-1$ implies that  $\Gamma^{m-1,l'}_{k+1}$ is a $C\varepsilon$-Lipschitz graph over $D^{m-1,l'}_u$ for each $u$. In addition, Lemma \ref{lem:球近则对应的锥接近} indicates that $L^{m-1,l'}_u$ and $ L^{m-1,l'}_i$ is close to each other. And the condition about boundary is also the same. So we can extend the condition (\ref{M3L^{m-1,l'}_i恰好是对应的原像}) in $5B_i$ to $6B_i$. 



Let $h=\ppi^{m,l}_i \circ g^m_k\circ \phi^{m,l}_j$. If $y\in D^{m,l}_j\cap 4.9 B_j$ is such that $\phi^{m,l}_j(y)\in 3B_j$ and $h(y)\in 10B_i$, then we have $|h(y)-y|<C\varepsilon 2^{-k}$ by the same argument as for (\ref{h移动点近}). Therefore, $|h(y)-y|<C\varepsilon 2^{-k}$ for each $y\in D$.
By induction hypothesis for (M3) of dimension $m$ and step $k$, every $\Gamma^{m-1,l'}_k$ is a boundary of $\Gamma^{m,l}_k$. Combining with (\ref{M3的claim}), $h(S_1)$ separates $P^{m,l}_i\cap B(x_i,(5+1/200)r_i)$ into 2 parts.
Call $i(\xi)$ the index of $ h(S_1\cup S_0)$ with respect to a point $\xi\in P^{m,l}_i\cap B(x_i,(5+1/200)r_i)$. By the discussion in last paragraph, $i(\xi)\neq 0$ on one part and $i(\xi)= 0$ on the other. Denote by $F_1$ the first and $F_2$ the second.
 We first prove that
 \begin{equation}\label{M3F_1被含于}
     F_1\subset h( \overline{D}).
 \end{equation}
 We can deform $S_1\cup S_0 $ into a point $ p\in \overline{D}$. Therefore, $h(S_1\cup S_0)$ can be deformed into $h(p)$. By the Lipschitz condition of $\Gamma^{m-1,l'}_{k+1}$, if $\xi \in  B(x_i,(5+1/200)r_i)\cap P^{m,l}_i \backslash  h( \overline{D})$, then the deformation of $h(S_1\cup S_0)$ will not meet $\xi$ so $i(\xi)$ will not change. As a result, $i(\xi)=0$. Thus $ F_1$ does not meet $B(x_i,(5+1/200)r_i)\cap P^{m,l}_i \backslash  h( \overline{D}) $ and (\ref{M3F_1被含于}) follows.

We then prove that
\begin{equation}\label{F_2和h(D)不交}
    F_2\cap h((D^{m,l}_j)^{\circ })=\emptyset,
\end{equation}
where $ (D^{m,l}_j)^{\circ }$ is the interior of $D^{m,l}_j$. For this, let us show $F_2\cap h(D^{\circ })=\emptyset$, where $D$ is as in (\ref{M3B'和D的定义}).
If $F_2$ meets $h(D^{\circ})$, we can find $q\in F_2$ such that $q=h(p)$ for some $ p\in D^{\circ}$. In the meanwhile, we can find a point $u\in D^{\circ}$ such that $h(u)\in F_1$. Let $\gamma$ be a path connecting $ p$ and $u$ in $D^{\circ}$, then we have $ h(\gamma)\subset \ppi^{m,l}_i (G)\cap 5.4 B_i$ and $h(\gamma)$ intersects with $ h(S_1)$, where $G$ is introduced in (\ref{M3:被包含在graph中}) is a $C_{m,4}\varepsilon$-Lipschitz over $P^{m,l}_i$. Let $w$ be the intersection of $h(\gamma)$ and $h(S_1)$, then $w\in\ppi^{m,l}_i(\Gamma^{m-1}_{k+1})$ and $w$ is the image of a point in $D^{\circ}$ by $h$ because $\gamma\subset D^{\circ} $. At the same time, $w$ is the image of a point in $ S_1$ by $h$ because $ w\in h(S_1)$, which contradicts the injectivity of $\ppi^{m,l}_i\circ g^m_k$. On the other hand, for each $y\in (D^{m,l}_j)^{\circ}\backslash B'$, we have $ |h(y)-x_i|\geq |y-x_i|-|y-h(y)|\geq  |y-x_i|-C\varepsilon 2^{-k}\geq (5+1/100)r_i$ if $\phi^{m,l}_j(y)\in 3B_j$. And when $\phi^{m,l}_j(y)\notin 3 B_j$, $g^m_k$ will not move $\phi^{m,l}_j(y)$ into $6B_i\supset B'$.
As a result, $h(y)\notin F_2$ and  (\ref{F_2和h(D)不交}) follows.

By (\ref{M3的claim}), $\Gamma^{m-1,l'}_{k+1}$, $\cup_{l'}\Gamma^{m-1,l'}_{k+1}$ separate $ G\cap B(x_i,(5+1/300)r_i) $ into two parts. Denote by $G_1$ the one whose image is contained in $F_1$ by $\ppi^{m,l}_i$ and the other $G_2$. By (\ref{M3F_1被含于}) and the fact $ \phi^{m,l}_j(D)\subset \Gamma^{m,l}_k\cap B_x$ (mentioned below (\ref{M3B'和D的定义})), we have $G_1\subset g^m_k(\Gamma^{m,l}_k\cap B_x)\cap B(x_i, (5+1/300)r_i)$. Combining with (\ref{F_2和h(D)不交}), we have 
\begin{equation}\label{G1}
    (G_1\cup (\ccup_{l'}\Gamma^{m-1,l'}_{k+1}))\cap  B(x_i, (5+\frac{1}{300})r_i)= g^m_k(\Gamma^{m,l}_k\cap B_x)\cap  B(x_i, (5+\frac{1}{300})r_i).
\end{equation}
Set $G^{m,l}_i=G_1\cup (\ccup_{l'}\Gamma^{m-1,l'}_{k+1}))$, then $G^{m,l}_i\subset G$ and $G^{m,l}_i$ has boundaries $U_{l'}\Gamma^{m-1,l'}_{k+1}$. As a result of (\ref{M3:5Bi的来自Bx}) and (\ref{G1}),  we have 
\begin{equation}
    \Gamma^{m,l}_{k+1}\cap 5B_i=G^{m,l}_i\cap 5B_i.
\end{equation}
Then let $D^{m,l}_i=\ppi^{m,l}_i (G^{m,l}_i)\cap B(x_i,(5+1/300)r_i)=(F_1\cup \ppi^{m,l}_i (\ccup_{l'}\Gamma^{m-1,l'}_{k+1})) \cap  B(x_i,(5+1/300)r_i)$.   According to the argument above, there is a constant $C$ depending on $\{C_{u,e}\}_{0\leq u\leq m-1,1\leq e\leq 5}$ such that
\begin{equation}
    d_{x_i,(5+\frac{1}{500})r_i}(D^{m,l}_i,L^{m,l}_i)\leq \max_{l'} d_{x_i,(5+\frac{1}{500})r_i}( D^{m-1,l'}_i, L^{m-1,l'}_i)+C_{t,3}\varepsilon2^{-k}/(5r_i)+2C_{m-1,e}\varepsilon<C\varepsilon,
\end{equation}
where $e=2$ if $i\in I_{m-1}$ and $e=4$ if $i\not\in I_{m-1}$. 
And we have already proved that $\Gamma^{m,l}_{k+1}$ has boundaries like $\Gamma^{m-1,l'}_{k+1}$ and the corresponding branch $L^{m-1,l'}_i$ is an $(m-1)$-boundary of $L^{m,l}_i$ in (\ref{M3L^m_i有边界和L^m-1_j近}) and (\ref{M3L^{m-1,l'}_i恰好是对应的原像}). At last, we need to prove that if $ L^{m,l}_i\cap L^{m,l'}_i=L^{d,s}_i$, then
\begin{equation}
    \Gamma^{m,l}_{k+1}\cap \Gamma^{m,l'}_{k+1}=\Gamma^{d,s}_{k+1}.
\end{equation}
Since $L^{m,l}_i$ and $ L^{m,l'}_i$ make angle larger than $\alpha$ along $L^{d,s}_i$, we know that $\Gamma^{m,l}_{k+1}$ and $ \Gamma^{m,l'}_{k+1}$ leave from $\Gamma^{d,s}_{k+1}$ in direction that make an angle larger $\alpha$ so they only meet at $\Gamma^{d,s}_{k+1}$. Thus we end the proof of (M3) for $k+1$.

\subsubsection{Proof of (M4) for the inductive step $k+1$}

For each $x\in \Gamma^m\cap B(0,\rho^{m+1}_0+b_0^{-n}2^{10})$, we have $| f^m_k(x)-x|<C\varepsilon$ by (\ref{g^m_k移动点y很近}), where $C$ depends only on $C_5, C_8, C_{m,1}$ and $\{C_{u,e}\}_{0\leq u\leq m-1,1\leq e\leq 5}$. Let $z=f^m_k(x)$, then for each $i\in I_z\cap I(k+1)$, $x_i\in B(0,\rho^m_{k+1})$. Thus, we can get from (\ref{fk,gk}), (\ref{psi的正则性}) and induction hypothesis of (M4) for step $k$ that $ f^m_{k+1}$ is continuous at $x$.
Furthermore, if $ x\notin \Gamma^{m-1}$, then $f^m_{k+1}$ is of class $C^1$ at $x$. If $x\in \Gamma^{m-1}$, by (\ref{psi^m-1和psi^m在m-1维graph上相等}) and (\ref{fk,gk}), $f^m_{k+1}(x)=f^{m-1}_{k+1}(x)$.

Then we only need to check that the derivative of the restriction of $f^m_{k+1}$ to $\Gamma^m\cap B(0,\rho^{m+1}_0+b_0^{-n}2^{10})$ does not vanish. Without loss of generality, suppose that $x\in \Gamma^{m,l}\backslash \Gamma^{m-1}$. Let $v$ be a tangent vector to $\Gamma^m$ at $x$ and let $\tau= Df^m_k(x)(v)$. By induction hypothesis of (M4) for step $k$, $\tau\neq 0$. So $\tau$ is a tangent vector to $\Gamma^m_k$ at $z$. By (\ref{M2Dg-Dpi的式子}) and (\ref{M3估算|Dg-Dpi|}), we have either $|D g^m_k(z)(\tau) -D\pi^m_i (\tau) |<C\varepsilon|\tau|$ or $ |D g^m_k(z)(\tau)-D\ppi^{m,l}_i(\tau)|<C\varepsilon|\tau|$. And by Lemma \ref{面的dx,r决定夹角和距离}, Lemma \ref{lem:球近则对应的锥接近} and the induction hypothesis of (M2), (M3) for step $k$, we have either $ |\tau-D\pi^m_i(\tau)|<C\cdot (C_7+C_{t,2})\varepsilon|\tau|$ or  $|\tau-D\ppi^{m,l}_i(\tau)|<C\cdot (C_7+C_{t,4})\varepsilon|\tau|$, where $C$ is a geometric constant depending only on $n$. As a conclusion, $ |Dg^m_k(z)(\tau)|\geq (1-C\varepsilon)|\tau|>0.99 |\tau|>0$. Thus, $D f^m_{k+1}(x)(v)\neq 0$, as needed. And we end the proof of (M4) for $k+1$.

\subsubsection{$f^m$ is bi-H$\ddot{\text{o}}$lder}

Now consider $f^m$, the limit of the $f^m_k$. Let 
\begin{equation}
    B^m= B(0,1.95+[1-2m n_0^{-n}+n_0^{-n}]\cdot 2^{-10}).
\end{equation}
Then $B^m\subset B(0,\rho^m_k)$ for all $k\geq 0$. By (\ref{g^m_k移动点y很近}), we have $|f^m_{k+1}(x)-f^m_k(x)|<C\varepsilon 2^{-k}$ for every $x\in L^m\cap B^m$ and $f^m$ moves $x$ no more than $C\varepsilon$.
Since $\lim\limits_{k\to\infty} \dd(f^m_k(x),E_m)=0$ by (M1) and $E_{t-1}$ is contained in the closure of $E_t$ for each $t\geq 1$ and $\ccup_{t=0}^m E_t$ is closed in $\RR^N$ by Proposition \ref{每一层Em的关系}. We can immediately show that 
\begin{equation}
    f^m(L^m\cap B^m)\subset (\ccup_{t=0}^m E_t) \cap B(0,1.95+[1-2m n_0^{-n}+2 n_0^{-n}]\cdot 2^{-10}).
\end{equation}
Then we continue to prove that 
\begin{equation}\label{f^m包含}
    (\ccup_{t=0}^m E_t ) \cap B(0,1.95+[1-2m n_0^{-n}]\cdot 2^{-10})   \subset f^m(L^m \cap B^m).
\end{equation}
Fix $k\geq 0$, for each $z$ contained in the left-side of (\ref{f^m包含}), there exists $i\in \ccup_{t=0}^m I_t(k)$ such that $z\in 2B_i$ by (\ref{2Bi构成Em覆盖}). Recall $d_{x_i,100r_i}(E_m,L^m_i)<C_5\varepsilon$ in (\ref{取出来的锥每一层都接近}), we can find $y\in L^m_i$ such that $|z-y|<C\varepsilon 2^{-k}$. Suppose $i\in I_t(k)$ and $y\in L^{m,l}_i$, then we have  $\dd(y,\Gamma^{m,l}_k)<C\varepsilon 2^{-k}$ because $\Gamma^t_k\cap B(x_i,C\varepsilon 2^{-k})\neq \emptyset$ and $\Gamma^{m,l}_k\cap 5B_i$ is a $C\varepsilon$-Lipschitz graph of $D^{m,l}_i$ by (M2) and (M3). So  $\dd(z,f^m_k(L^m\cap B^m))<C\varepsilon 2^{-k} $. Let $k$ tends to $\infty$, then $z\in f^m(L^m\cap B^m)$, and (\ref{f^m包含}) follows.

Next we proceed to show $f^m$ is bi-H$\ddot{\text{o}}$lder.  We first check that for $y,z\in f^m_k(L^m\cap B^m)$ such that $|y-z|\leq n_0^{-(m+1)n}\cdot  2^{-k-10^{m+2}}$,  there exists a constant $C$ such that
\begin{equation}\label{for m, holder}
    (1-C\varepsilon)|y-z|<|g^m_k(y)-g^m_k(z)|<(1+C\varepsilon)|y-z|.
\end{equation}
By (M1), $\dd(y,E_m)<C\varepsilon 2^{-k}$, so there is $i\in \cup_{t=0}^m I_t(k+1)$ such that $y\in 2B_i$, thus $z\in 4B_i$. Suppose $i\in I_s(k+1)$ for some $0\leq s\leq m$, then there exists $j\in \ccup_{t=0}^s I_t(k)$ such that $x_i\in 2B_j$. Therefore, 
$ y\in 3B_j$ and $z\in 4B_j$. 

If  $s=m$, recall (M2), $g^m_k(y),g^m_k(z)$ are on the same branch of $\Gamma^m_{k+1}$, so  $y,z$ are also on the same branch of $\Gamma^m_{k}$. We denote it by $\Gamma^{m,l}_{k}$. { The estimates} 
(\ref{M2Dg-Dpi的式子}) and (\ref{M3估算|Dg-Dpi|}) imply that  $|Dg^m_k(x)-D\ppi^m_i|<C\varepsilon$ on $\Gamma^{m,l}_k\cap 5.5 B_i$, so we can estimate that
\begin{equation}\label{gk和pi}
    | |g^m_k(y)-g^m_k(z)| - |\pi^{m,l}_j (y)-\pi^{m,l}_j(z)| |<C\varepsilon
    |\pi^{m,l}_j (y)-\pi^{m,l}_j(z)| 
\end{equation}
because $d_{x_i,20r_i}(L^m_i,L^{m,l}_j)<C\varepsilon$ by Lemma \ref{lem:球近则对应的锥接近}. Here we use $\pi^{m,l}_j$ rather that $\ppi^{m,l}_j$ because $t\leq s=m$ and $ W_j$ coincides with a set of type $m$ in $100B_i$, which means that $L^{m,l}_j$ coincides with an $m$-plane in $100 B_i$.
In the meanwhile, we can use the Lipschitz property of $\Gamma^{m,l}_k$ in (M2) and (M3) to get that
\begin{equation}\label{id和pi}
    ||y-z|-|\pi^{m,l}_j (y)-\pi^{m,l}_j (z)| |<C\varepsilon |\pi^{m,l}_j (y)-\pi^{m,l}_j (z)|.
\end{equation}
Then (\ref{for m, holder}) follows by (\ref{gk和pi}) and (\ref{id和pi}).

If $s<m$ and $y,z$ are on the same branch of $ \Gamma^m_k$, proof { of }
(\ref{for m, holder}) is the same as for the case $s=m$. Now we are left with the case when $y,z$ are on the different branches of $\Gamma^{m}_k$. suppose that  $y\in\Gamma^{m,l}_k$ and $z\in\Gamma^{m,l'}_k$ with $l\neq l'$. Since $ W_j$ coincides with a set of type $s$ in $100B_i$, the dimension of the intersection of $L^{m,l}_j$ and $L^{m,l'}_j$ is at least $s$.
Suppose $L^{m,l}_j\cap L^{m,l'}_j=L^{d,h}_j$, then $d\geq s$ and $\dd(x_i, L^{d,h}_j)<C\varepsilon 2^{-k}$ by Lemma \ref{lem:球近则对应的锥接近}.
Let us choose $x\in \Gamma^{d,l}_k$ such that $|x-y|+|x-z|$ is minimal. Since $L^{d,h}_j$ is close to $x_i$ and $y,z\in B_i$, then we have $x\in 4.1 B_i\cap 4.1 B_j$. 
Since every two different branches make an angle greater than $\alpha$ along their intersection, and $\Gamma^{m,l}_k$ is a $C\varepsilon$-Lipschitz of $L^{m,l}_j$, $\Gamma^{m,l'}_k$ is a $C\varepsilon$-Lipschitz of $L^{m,l'}_j$, we have $\angle yxz>\alpha/2$. Thus, $|y-z|\geq 
C(|y-x|+|x-z|)$ for some geometric constant $C$ depends only on  $\alpha$. Since $x,y$ is on the same branch, (\ref{for m, holder}) holds for $x,y$. Also, (\ref{for m, holder}) holds for $x,z$. So we have $| |g^m_k(y)-g^m_k(z)|- |y-z| |<C\varepsilon |y-z|$ and (\ref{for m, holder}) follows.

For every $y,z\in L^m\cap B^m$, let $y_k=f^m_k(y)$ and $z_k=f^m_k(z) $. {  Assume that} $|y-z|<n_0^{-(m+1)n}2^{-10^{m+2}}$ (we shall address the other case later). Let $k_0$ be an integer such that $ |y_t-z_t|<n_0^{-(m+1)n} 2^{-t-10^{m+2}}$ for all $t\in\{ 0,..., k_0-1\}$ and $|y_{k_0}-z_{k_0}|\geq n_0^{-(m+1)n} 2^{-k_0-10^{m+2}} $. Then we can apply (\ref{for m, holder}) to get
\begin{equation}
    (1-C\varepsilon)^{k_0}|y-z|<|y_{k_0}-z_{k_0}|<(1+C\varepsilon)^{k_0} |y-z|.
\end{equation}
Since $|y_{k_0-1}-z_{k_0-1}|<n_0^{-(m+1)n}\cdot 2^{-k_0+1-10^{m+2}}$, we can estimate that 
\begin{equation}
     (1-C\varepsilon)^{k_0-1}|y-z|<|y_{k_0-1}-z_{k_0-1}|< 2^{-k_0+1}\cdot n_0^{-(m+1)n}2^{-10^{m+2}},
\end{equation}
then we have
\begin{equation} 
    \log_2\frac{1}{|y-z|}>k_0+k_0\log_2(1-C\varepsilon) >k_0(1-C\varepsilon).
\end{equation}
When $k\geq k_0$,  we have
\begin{equation}
    ||y_k-z_k|-|y_{k_0}-z_{k_0}||\leq 2\cdot C\varepsilon \sum_{k=k_0}^{\infty} 2^{-k}=C\varepsilon 2^{-k_0} .
\end{equation}
Let $k$ tends to infinity, then we get 
\begin{equation}
    ||f^m(y)-f^m(z)|-|y_{k_0}-z_{k_0} ||<C\varepsilon 2^{-k_0}. 
\end{equation}
Since $|y_{k_0}-z_{k_0}|\geq n_0^{-(m+1)n}\cdot 2^{-k_0-10^{m+2}}$, we get that $2^{-k_0}<n_0^{-(m+1)n}\cdot 2^{10^{m+2}}|y_{k_0}-z_{k_0}|$. So we can show that there is a geometric constant $C>0$ such that
\begin{equation}\label{m维holder}
    (1-C\varepsilon)|y-z|^{1+C\varepsilon}<|f^m(y)-f^m(z)|<(1+C\varepsilon)|y-z|^{1-C\varepsilon}.
\end{equation}
If $|y-z|\geq n_0^{-(m+1)n}\cdot 2^{-10^(m+2)}$, recall that $|y-z|<4$ and $n_0^{(m+1)n}\cdot 2^{10^{m+2}}|y-z|\geq 1$, we can also get (\ref{m维holder}) by (\ref{g^m_k移动点y很近}). 
{  
Note that since the spines are nested ($L^0 \subset L^1 \subset \dots \subset L^n$ by Definition \ref{spine}), any point $y$ belonging to a lower-dimensional spine $L^t$ ($t < m$) is automatically contained in $L^m$. Therefore, the estimate derived for $L^m$ applies directly to the pair $(y, z)$ as elements of $L^m$, covering all cross-spine cases.
} 

Now we can define the extension of $f^n$ in $B(0,1.95+(1-2n n_0^{-n})2^{-10})$ to get $f$. Set
\begin{equation}
    \begin{split}
        f_0=id, \enspace f_{k+1}=g_k\circ f_k,\enspace
        g_k(x)=\sum_{i\in I(k)}\theta_i(x)\cdot \psi_i(x),\enspace f=\lim_{k\to\infty}f_k.
    \end{split}
\end{equation}
For $k\geq 0$, when $i\in I_{n+1}(k)$, set $\psi_i=id$. 
When $i\in \ccup_{m=0}^{n} I_m(k)$, recall the construction of $\eta^{m-1}_i$ and define $\eta^n_i$ in the same way as in subsection \ref{对齐映射构造}. Let $\psi_i=\eta^n_i$, then $\psi_i$ also satisfies the properties in \eqref{eta的4个性质}.
Then we have $|g_k(x)-x|<C\varepsilon 2^{-k}$ and $|Dg_k-I|<C\varepsilon$ on $B(0,1.95)\backslash \Gamma^{n}_k$ since $g_k$ is an average of $\psi_i$ and $\theta_i$ is supported on $3 B_i$.
By repeating the discussion as for $m$, we get the map $ f$ defined on $B(0,1.95+(1-2n n_0^{-n})2^{-10})$. Then we have $|f(x)-x|<C\varepsilon$ on $B(0,1.95)$ and (\ref{f移动一点距离}) follows. So we can use degree theory to get (\ref{f的包含与被包含}). 
And $f$ coincides with $f^n$ on $Z(0,2)\cap B(0,1.95)$, so we have (\ref{f作用于Z(0,2)}).
By using the same argument as for (\ref{m维holder}), we have (\ref{f Holder}).

We finally completed our verification of Theorem \ref{主定理1}.

\bibliographystyle{alpha} 
\bibliography{ref}

\end{document}